\documentclass[11pt]{amsart}
\usepackage{amsfonts}
\usepackage{fullpage,amsmath, amsthm,amssymb,amsfonts,stmaryrd, mathrsfs}
\usepackage{hyperref}

\usepackage[all]{xy}

\numberwithin{equation}{subsection}

\newtheorem{theorem}[subsection]{Theorem} 
\newtheorem{fact}[subsection]{Fact}
\newtheorem{lemma}[subsection]{Lemma}
\newtheorem{corollary}[subsection]{Corollary}
\newtheorem{conjecture}[subsection]{Conjecture}

\newtheorem{proposition}[subsection]{Proposition}
\theoremstyle{definition}

\newtheorem{definition}[subsection]{Definition}
\newtheorem{hypothesis}[subsection]{Hypothesis}

\newtheorem{remark}[subsection]{Remark}

\newtheorem{notation}[subsection]{Notation}

\def\calA{\mathcal{A}}

\def\calC{\mathcal{C}}
\def\calD{\mathcal{D}}
\def\calE{\mathcal{E}}
\def\calF{\mathcal{F}}
\def\calG{\mathcal{G}}
\def\calH{\mathcal{H}}
\def\calI{\mathcal{I}}

\def\calM{\mathcal{M}}
\def\calN{\mathcal{N}}
\def\calO{\mathcal{O}}
\def\calQ{\mathcal{Q}}
\def\calR{\mathcal{R}}
\def\calS{\mathcal{S}}

\def\calU{\mathcal{U}}
\def\calV{\mathcal{V}}

\def\gotha{\mathfrak{a}}
\def\gothb{\mathfrak{b}}
\def\gothc{\mathfrak{c}}
\def\gothd{\mathfrak{d}}
\def\gothn{\mathfrak{n}}

\def\gothC{\mathfrak{C}}

\def\gothm{\mathfrak{m}}

\def\gothp{\mathfrak{p}}
\def\gothq{\mathfrak{q}}

\def\gothN{\mathfrak{N}}

\def\CC{\mathbb{C}}
\def\FF{\mathbb{F}}
\def\GG{\mathbb{G}}

\def\NN{\mathbb{N}}
\def\PP{\mathbb{P}}
\def\QQ{\mathbb{Q}}
\def\RR{\mathbb{R}}
\def\TT{\mathbb{T}}
\def\ZZ{\mathbb{Z}}

\def\rmn{\mathrm{n}}

\def\scrF{\mathscr{F}}
\def\scrH{\mathscr{H}}
\def\scrI{\mathscr{I}}
\def\ttD{\mathtt{D}}
\def\ttF{\mathtt{F}}
\def\ttG{\mathtt{G}}

\def\ttS{\mathtt{S}}
\def\ttT{\mathtt{T}}
\def\ra{\rightarrow}
\def\lra{\longrightarrow}

\newcommand{\SD}{\mathrm{SD}}

\DeclareMathOperator{\End}{End}
\DeclareMathOperator{\Gal}{Gal}
\DeclareMathOperator{\Ker}{Ker}
\DeclareMathOperator{\Hom}{Hom}

\DeclareMathOperator{\rank}{rank}

\DeclareMathOperator{\Res}{Res}
\DeclareMathOperator{\Spec}{Spec}

\DeclareMathOperator{\Tor}{Tor}

\newcommand{\tor}{\mathrm{tor}}

\newcommand{\longto}{\longrightarrow}

\newcommand{\can}{\mathrm{can}}

\newcommand{\dR}{\mathrm{dR}}

\newcommand{\DP}{\mathrm{DP}}
\newcommand{\cris}{\mathrm{cris}}

\newcommand{\ord}{\mathrm{ord}}
\newcommand{\Nm}{\mathrm{Nm}}

\newcommand{\Tate}{\mathrm{Tate}}

\newcommand{\SL}{\mathrm{SL}}

\newcommand{\PR}{\mathrm{PR}}

\newcommand{\Ra}{\mathrm{Ra}}
\newcommand{\univ}{\mathrm{univ}}

\newcommand{\Sh}{\mathrm{Sh}}
\newcommand{\Frob}{\mathrm{Frob}}
\DeclareMathOperator{\GL}{GL}

\newcommand{\calHom}{\calH om}
\newcommand{\Ha}{\mathrm{Ha}}

\begin{document}
\title{Unramifiedness of Galois representations arising from Hilbert modular surfaces}

\author{Matthew Emerton}
\address{Matthew Emerton, Department of Mathematics, The University of Chicago, 5734 S. University
Avenue, Chicago, Illinois 60637, USA}
\email{emerton@math.uchicago.edu}
\author{Davide A. Reduzzi}
\address{Davide A. Reduzzi}
\email{davide.reduzzi@yahoo.com}
\author{Liang Xiao}
\address{Liang Xiao, Department of Mathematics, University of Connecticut, Storrs,
341 Mansfield Road, Unit 1009, Storrs, Connecticut 06269, USA}
\email{liang.xiao@uconn.edu}

\begin{abstract}
Let $p$ be a prime number and $F$ a totally real number field. For each prime $\gothp$ of $F$ above $p$ we construct a Hecke operator $T_\gothp$ acting on $(\mathrm{mod}\, p^m)$ Katz Hilbert modular classes which agrees with the classical Hecke operator at $\gothp$ for global sections that lift to characteristic zero. Using these operators and the techniques of patching complexes of F. Calegari and D. Geraghty we prove that the Galois representations arising from torsion Hilbert modular classes of parallel weight ${\bf 1}$ are unramified at $p$ when $[F:\QQ]=2$. Some partial and some conjectural results are obtained when $[F:\QQ]>2$.

\end{abstract}

\thanks{M.E. was
partially supported by NSF grants DMS--1249548 and DMS--1303450. 
D.R. was partially supported by an AMS--Simons Research Travel Grant.
L.X. was partially supported by Simons Collaboration Grant \#278433, NSF grant DMS--1502147, and CORCL research grant from University of California, Irvine.}

\maketitle

\setcounter{tocdepth}{1}
\tableofcontents

\section{Introduction}

With this paper we begin the investigation of local-global compatibility results for Galois representations attached to \emph{torsion} classes (of non-cohomological weight) occurring in the \emph{coherent} cohomology of Hilbert modular schemes. The existence of such representations was previously proved in \cite{emerton-reduzzi-xiao} (see also generalizations in \cite{reduzzi-xiao, boxer, goldring-kostivirta}). Our results are motivated by the conjectures made in \cite{CG} (e.g. \cite[Conjectures A and B]{CG}) by  Calegari and  Geraghty, and by the consequent applications to modularity lifting theorems as in \emph{op.cit.}. \\

Let $F$ be a totally real number field of degree $g$ and let $p$ be a prime number. To simplify the notation of this introduction, we assume that $p$ is unramified in $F$. Denote by $\calO$ the ring of integers in a large enough finite extension of $\QQ_p$ and let $\varpi$ be a uniformizer in $\calO$.
Let $\calS$ denote a finite set of places containing all archimedean places, all $p$-adic places, and all places there the level structure of Hilbert modular schemes is not hyperspecial. Let $\calM$ denote the corresponding Hilbert modular scheme (cf. \S\ref{S:moduli HBAS}) with hyperspecial level structure at $p$, \emph{i.e.}, the fine moduli scheme over $\Spec\calO$ classifying $g$-dimensional abelian schemes endowed with a certain prime-to-$p$ polarization, an endomorphism action by $\calO_F$, and a suitable tame level structure $K^p$. We let $\pi:\calA\to\calM$ denote the universal abelian scheme over $\calM$ and we set $\omega:=\pi_\ast\Omega^1_{\calA/\calM}$.

Let $\Sigma$ denote the set of $p$-adic embeddings of $F$. For each paritious weight (cf. \S\ref{S:geometric HMF}) \[
\kappa=((k_\tau)_{\tau \in \Sigma}; w):\Res_{F/\QQ}\GG_{m/F} \times \GG_{m/\QQ}\to\GG_{m/\QQ},\] the natural automorphic line bundle $\dot\omega^\kappa$ descends to a line bundle $\omega^\kappa$ over a certain finite quotient $\Sh$ of $\calM$, whose $\CC$-points form a (union of) connected components of the Shimura variety for $\GL_{2, F}$.  We choose an appropriate toroidal compactifications $\calM^\tor$ and $\Sh^\tor$ of $\calM$ and $\Sh$, respectively.
The \emph{coherent} sheaf cohomology of $\omega^\kappa$ over $\Sh_{\calO/\varpi^m}^\tor$ is naturally endowed with an action of commuting Hecke operators $T_\gothq$ parametrized by the prime ideals $\gothq \notin \calS$. As proved in \cite{emerton-reduzzi-xiao}, Hecke eigenclasses arising from any degree of coherent cohomology have canonically attached pseudo-representations of $G_F=\Gal(\overline{F}/F)$ unramified outside $\calS$ whose Frobenii traces match the Hecke eigenvalues. The existence of such representations follows from classical results when the automorphic line bundle $\omega^\kappa$ is twisted by a large enough power of $\omega^{({\bf 1}, 1)}$ with ${\bf 1}: =(1,\dots,1)$, since the corresponding modular forms lift to characteristic zero by the ampleness of $\omega^{\bf 1}$ on the minimal compactification of $\Sh$. On the other hand, when the weight $\kappa$ is small, such Hilbert modular forms or cohomological classes might not lift to characteristic zero, and Galois pseudo-representations are constructed in \emph{op.cit.} by exploiting the stratification of $\Sh_{\overline\FF_p}$ induced by partial Hasse invariants, and by constructing trivializations of suitable automorphic line bundles restricted to such strata. Our method was later generalized by \cite{boxer,goldring-kostivirta}. (We recall that the case of torsion in Betti cohomology is extensively treated in the groundbreaking work of  Scholze \cite{PSh}).

In \cite{CG},  Calegari and  Geraghty use the (conjectural) existence and local properties of Galois representations arising from torsions in the coherent or Betti cohomology of locally symmetric spaces to prove general modularity lifting results over an arbitrary number field. 

For instance, when $F=\QQ$ a local-global compatibility theorem is proved in \cite[Theorem~3.11]{CG} for Katz modular forms of weight $1$ defined over an arbitrary $\ZZ_p$-algebra, and a minimal modularity lifting result is deduced using commutative algebra techniques (cf. Theorem 3.26 of \emph{op.cit.}). 
Similarly, assume that $g>1$ and let $\gothp\mid p$ be a prime of $F$. If we chose a weight $\kappa$ whose $\gothp$-components are all equal to $1$, it is natural to conjecture that the corresponding  Galois representations arising from torsions in cohomology are unramified at $\gothp$. But one major problem in approaching this conjecture is the lack of a good $T_\gothp$ Hecke-operator acting on such torsion classes. This problem is solved when $F=\QQ$, thanks to \cite[Proposition 4.1]{gross}.

In this paper, we 
construct a natural Hecke operator at $\gothp\mid p$ acting on the cohomology of $\Sh_{\calO/\varpi^m}$ in any degree, and then we use this and the techniques of patching complexes of Calegari--Geraghty (\cite[\S6]{CG}) to prove the following.

\begin{theorem}\label{TI:main}
Assume that $F$ is a real quadratic field and $p>3$ is a prime inert in $F$. Let $\rho$ be a Galois representation arising from $H^\bullet(\Sh^\tor_{\calO/\varpi^m},\omega^{({\bf 1}, -1)})$ for some $m\leq \infty$. Suppose that $\bar \rho: = \rho\otimes\bar\FF_p$ is irreducible, $\bar{\rho}(\Frob_p)$ has distinct eigenvalues,  and that $\bar{\rho}(G_F)$ contains $\SL_2(\FF_p)$. Then the representation $\rho$ is unramified at $p$.
\end{theorem}
Here $\Sh^\tor$ is a toroidal compactification of $\Sh$. This theorem is proved in Theorem \ref{T:main}.
\begin{remark}

\begin{enumerate}
\item Here we used the weight $({\bf 1}, -1)$ as opposed to $({\bf 1}, 1)$; this is due to our new normalization of Hecke operators. See \S \ref{S:tame Hecke operators} for more discussion.
\item
The assumption that $p$ is inert in $F$ is just meant to keep our arguments as simple as possible. The case when $p$ is ramified can be treated by working over the Pappas--Rapoport splitting model of the Deligne--Pappas modular schemes. 
\item
It would be interesting to see whether the similar techniques together with some results on the global geometry of the Goren--Oort strata of Hilbert modular varieties from \cite[\S1.5]{tian-xiao} could be applied to prove Theorem~\ref{TI:main} for a general totally real field $F$ assuming that for each prime $\gothp\,|\, p$ in $F$ we have $e_\gothp f_\gothp\leq 2$.
\item The assumption on $\bar\rho(\Frob_p)$ might be removed by suitable adaptation of the doubling arguments of \cite[\S3.6--7]{CG} and \cite{wiese}. 

\item
When $g>2$, one can show that the representations $\rho$ coming from $H^0$ is unramified (see Proposition~\ref{P:nr} for a proof under the $p$-distinguished hypothesis and see \cite{dimitrov-wiese} without the $p$-distinguished hypothesis).
\item

For more general $\rho$ coming from higher degree cohomology groups (and in some partial weight $1$ cases), we provide a strategy to approach such a question (see Proposition~ \ref{T:conj_main}).
\end{enumerate}  

\end{remark}

\subsection{Sketch of the construction of  the Hecke operator $T_\gothp$ on torsion classes}
Let $\gothp$ be a prime of $F$ above $p$ with residual degree $f_\gothp$ and inertial degree $e_\gothp$. The na\"ive attempt of defining $T_\gothp f$ for a Hilbert modular form $f\in H^0(\Sh_{\calO/\varpi^m}^\tor, \omega^\kappa)$ via the ``sum''
$$(T_\gothp f)(A,\dots):=\sum_Cf(A/C,\dots)$$
over suitable subgroup schemes $C\subset A[\gothp]$ of rank $p^{f_\gothp}$ does not make sense when $\varpi^m=0$ over the base scheme.
Here we lose the usual  factor $\frac{1}{\Nm_{F/\QQ}(\gothp)}$ comparing to many other literatures; we refer to \S \ref{S:tame Hecke operators} for the discussion.
 When twisting the automorphic line bundle $\omega^\kappa$ by a sufficiently large power of the ample bundle $\omega^{({\bf 1}, 1)}$, a global section over $\calO/\varpi^m$ can be first lifted characteristic zero and then we can apply the usual $T_\gothp$ operator there, and finally reduce modulo $\varpi^m$. The problem is that, when the weight $\kappa$ is fixed,  not all Hilbert modular forms (for instance, in $H^0(\Sh^\tor_{\calO/\varpi^m},\omega^{({\bf 1}, -1)})$) lift to characteristic zero. It is not clear \emph{a priori} why a Hecke operator $T_\gothp$ should act with the expected $q$-expansion on non-liftable forms. Moreover, for our purposes we would like to have a $T_\gothp$ Hecke-action on the entire cohomology $H^\bullet(\Sh^\tor_{\calO/\varpi^m},\omega^{({\bf 1}, -1)})$, and not just on $H^0$, so that the trick adopted in \cite[Proposition 4.1]{gross} when $g=1$ (which relies on lifting a mod $p$ Katz elliptic modular form of weight $1$ to an \emph{a priori} meromorphic form in characteristic zero) does not apply in our situation.

Our construction is inspired by the work of B. Conrad \cite{conrad}.
Let us motivate by recalling the definition of the Hecke operator at a prime $\gothq \notin \calS$. We construct a covering $\Sh(\gothq)$ of $\Sh$ by adding an Iwahori level structure at $\gothq$, and consider the Hecke correspondence attached to the two canonical projections $\pi_1,\pi_2:\Sh(\gothq)\to\Sh$ given respectively by forgetting the Iwahori level structure, and by quotienting by it. More precisely, the Hecke action of $T_\gothq$ is obtained by composing the natural map (twisted by the Kodaira--Spencer isomorphism, cf. \S\ref{S:def of Tp})
\begin{equation}\label{E:H1}
\pi_{1\ast}\pi_2^\ast\omega'^\kappa\to\pi_{1\ast}\pi_1^\ast\omega^\kappa
\end{equation}
with the finite flat trace map
\begin{equation}\label{E:H2}
\mathrm{tr_{ff}}:\pi_{1*}\pi_1^*\omega^\kappa\to\omega^\kappa,
\end{equation} 
and finally taking the cohomology. 

It is a non-obvious fact that this geometric approach to define the Hecke action works also in characteristic $p$ when $g=1$ and $\gothq=(p)$, as proved in \cite[\S4.5]{conrad}. The key here is to show that the composition $\pi_{1\ast}\pi_2^\ast\omega'^k\to\omega^k$ of the morphisms (\ref{E:H1}) and (\ref{E:H2}) between sheaves on $\calM_{\calO}$ has image inside $p\cdot\omega^k$ if $k\geq 1$, because it is so over the open dense ordinary locus. 

Unfortunately, when $g>1$ there are additional complications as the projections $\pi_1,\pi_2$ are not finite flat anymore over the special fiber of $\Sh$. For instance, when $g=2$ and $p$ is inert in $F$, the preimage under $\pi_1$ (or $\pi_2$) of the superspecial points of the special fiber of $\Sh$ is the union of two projective lines. To overcome this issue, we use in this paper the dualizing trace map $$\mathrm{tr}:R\pi_{1\ast}\mathscr{D_\gothp}\to\mathscr{D}$$
instead of the finite flat trace map to construct the Hecke operator. (Here $\mathscr{D}_\gothp$ and $\mathscr{D}$ denote the dualizing complexes on $\Sh(\gothp)$ and on $\Sh$ respectively). Twisting by the Kodaira--Spencer map we can identify $\mathscr{D}$ with the automorphic line bundle $\omega^{((2, \dots, 2),0)}$ on $\Sh$, and the compatibility between the dualizing trace map and the finite flat trace map guarantees that our construction of $T_\gothp^\rmn$ (an appropriately normalized $T_\gothp$-operator) coincides with the usual one on the ordinary locus of $\Sh$. 
Following \cite[\S4.5]{conrad}, it suffices to show that the composition $R\pi_{1*}\pi_2^*\omega'^\kappa \to \omega^\kappa$ factors through $\Nm_{F/\QQ}(\gothp)  \cdot \omega^\kappa$ in the derived category of bounded complexes of coherent modules of $\Sh$. It is not difficult to check this over the ordinary locus (cf. Proposition \ref{P:image p^g}).

Now we come to the second technical issue. This was a gap in the previous version of the paper and we thank Pilloni for pointing this out to us. Unlike in the case of $F = \QQ$, since we are working with the derived category, knowing the factorization over the ordinary locus {\it a priori} does {\bf not} imply the factorization over the entire $\Sh$.
Luckily, when $\kappa$ satisfies certain natural and mild conditions (cf \eqref{E:ampleness condition}), this implication holds.
Indeed, by some standard homological algebra argument (cf Proposition~\ref{P:abstract hom}), the factorization question is reduced to proving that the set-theoretical support of $R^j\pi_{1*}\pi_2^*\omega'^\kappa_{\overline\FF_p}$ for $j >0$ has codimension at least $j+1$ in the special fiber $\Sh_{\overline\FF_p}$, which is the content of Proposition~\ref{P:ample=>small support}.

Let us point out that the proof of Proposition~\ref{P:ample=>small support} makes use of some deep geometric result of the map $\pi_1: \Sh(\gothp)_{\overline\FF_p} \to \Sh_{\overline\FF_p}$, along the line of \cite{goren-kassaei,helm, tian-xiao}. We single out this proof in Section \ref{Sec:section 4} as we believe that it has its own interest.
To summarize the key points: we show that
\begin{itemize}
\item
all fibers of $\pi_{1,\overline\FF_p}$ are unions of products of $\PP^1$-bundles, and
\item
the restriction of $\pi_2^*\omega'^\kappa_{\overline \FF_p}$ to every $\PP^1$-bundle which affects the proof of Proposition~\ref{P:ample=>small support}, is $\calO(n)$ for some $n>-1$, and hence has zero higher derived pushforward.
\end{itemize}


\subsection{Sketch of the proof of Theorem~\ref{TI:main}} Recall that $p$ is assumed to be inert in $F$. 
We fix a mod $p$ irreducible Galois representation $\bar\rho:G_F\to\GL_2(\overline\FF_p)$ which is modular of weight $({\bf 1}, -1)$. We assume that $\bar{\rho}$ is $p$-distinguished, i.e. the Frobenius eigenvalues of $\bar \rho(\Frob_p)$ are distinct. (The case for $H^0$ without this further assumption on $\bar \rho(\Frob_p)$ is treated by Dimitrov and Wiese \cite{dimitrov-wiese}.)
The cohomology groups $H^*(\Sh^\tor_{\calO/\varpi^m},\omega^{({\bf 1}, -1)})_\gothm$ are modules for the Hecke algebra $\TT$ generated by the Hecke operators attached to primes away from $\calS$ (cf. \S\ref{S:Hecke TW primes}). 
We consider the localization at the maximal ideal $\gothm$ of $\TT$ corresponding to $\bar \rho$.

Having available a Hecke operator at $\gothp$ acting on torsion classes, we can prove (cf. Proposition \ref{P:nr}) that representations arising from $H^0(\Sh^\tor_{\calO/\varpi^m}, \omega^{({\bf 1}, -1)})_\gothm$ are unramified at $p$. As in the elliptic modular case (cf. \cite{edixhoven} and \cite[proof of Theorem 3.11]{CG}) the proof of this fact uses that the Hecke operator $T_\gothp^\rmn$ has the expected $q$-expansion on $H^0$ (Remark \ref{R:q-exp}), and consists of computing the commutator between $T_\gothp^\rmn$ and the total Hasse invariant acting on forms of weight $({\bf 1}, -1)$. By Serre duality (Lemma~\ref{L:hecke operator duality}), the result on $H^0$ can also be read as an unramifiedness result for representations arising from $H^g$ (see Proposition~\ref{L:unramified}).

From the Galois deformation point of view, the (framed) cohomology group 
\[
H^*(\Sh^\tor_{\calO/\varpi^m},\omega^{({\bf 1}, -1)})_\gothm^\square: = H^*(\Sh^\tor_{\calO/\varpi^m},\omega^{({\bf 1}, -1)})_\gothm \otimes_{\calO/\varpi^m} \calO/\varpi^m[[z_1, z_2, z_3,z_4]]
\]
becomes a module over the full universal (framed) $\calO$-deformation ring $R_p^\square$ of $\bar\rho|_{G_{F_p}}$ (with fixed determinant), where $G_{F_p} = \Gal(\overline \QQ_p/ F_p)$ is the local Galois group. This ring contains a \emph{proper} ideal $\mathscr{I}$ cutting out the locus of unramified lifts. The existence of the $T_\gothp^\rmn$-operator implies that the action of $R_p^\square$ on $H^{0,\square}$ and $H^{g,\square}$ through the quotient $R_p^\square/\mathscr{I}$.
Moreover, when $g=2$, a ``squeezing'' argument (cf. Proposition \ref{P:n=4}) gives: 
\begin{equation}\label{E:intro}
\mathscr{I}^3\cdot H^1(\Sh^\tor_{\calO/\varpi^m},\omega^{({\bf 1}, -1)})^\square_\gothm=0.
\end{equation} 

At this point, we make use of the Calegari--Geraghty version \cite[\S 6]{CG} of the Taylor--Wiles--Kisin patching argument. Denote by $r$ the number of finite prime-to-$p$ places in $\calS$. For each $N\geq 1$ let $\calQ_N$ be a set of cardinality $q\geq 1$ of Taylor--Wiles primes congruent to $1$ modulo $p^N$. 
Let $\Sh(\calQ_N)_1^\tor \to \Sh(\calQ_N)^\tor$ denote the corresponding \'etale cover of compactified Shimura varieties with a (variant of) $K^p \cap \Gamma_1(\calQ_N)$-level structure and a $K^p \cap \Gamma_0(\calQ_N)$-level structure (cf \S \ref{S:Sh(Q)01}). Then one can construct as in \cite[\S7.2]{CG} complexes computing the cohomology of $\Sh(\calQ_N)^\tor_{1,\calO/\varpi^m}$ localized at a suitable maximal ideal $\gothm_{\calQ_N}$. Such complexes are endowed with an action of a Hecke algebra $\TT_{\calQ_N,\gothm_{\calQ_N}}$ and of the group-algebra $\calO[(\ZZ/p^N\ZZ)^q]$, where $q$ is the dimension of a suitable dual Selmer group. Following the proof of \cite[Theorem 6.3]{CG}, we can patch these complexes together when $N$ increases. Introducing framing by $j:=4(r+1)-1$ variables, this process produces a perfect complex $P_\infty^{\square,\ast}$ of $S_\infty^\square:=\calO[[x_1,\dots,x_{q+j}]]$-modules. Moreover, $H^\bullet(P_\infty^{\square,\ast})$ is endowed with an action of a complete local Noetherian $R_p^\square$-algebra $R'_\infty$ of relative Krull dimension $4r+q-g$ (cf. \S\ref{S:patch}). Since $g=2$ and $\mathscr{I}^3\cdot H^{\ast,\square}=0$ by (\ref{E:intro}), we conclude that the action of $R_\infty'$ on the patched modules factors via $R_\infty:=R_\infty'/{\mathscr{I}^3}$.

The fact that $$\dim(R_p^\square/\mathscr{I}^n)=\dim(R_p^\square/\mathscr{I})=4$$ implies that the following numerical condition is satisfied: $$\dim R_\infty=\dim S_\infty^\square - g.$$
This guarantees that $$\mathrm{codim}_{S_\infty^\square}H^\ast(P_\infty^{\square,\ast})=g.$$
Since $g$ is also the length of our complexes computing coherent cohomology,  \cite[Lemma~6.2]{CG} implies that the patched complex $P_\infty^{\square,\ast}$ is a \emph{projective} resolution of its top-degree cohomology. In particular we have an isomorphism of $R_p^\square$-modules:

$$\Tor_i^{S_\infty^\square}(H^g(P_\infty^{\square,\ast}),\calO/\varpi^m)\simeq H^i(\Sh^\tor_{\calO/\varpi^m},\omega^{({\bf 1}, -1)})_\gothm^\square.$$
 
Loosely speaking, the coherent cohomology of $\Sh_{\calO/\varpi^m}$ is controlled by the top-degree cohomology of some patched complex $P_\infty^{\square,\ast}$. Since $H^g(P_\infty^{\square,\ast})$ is built by patching the top cohomology groups of the schemes $\Sh(Q_N)_{1,\calO/\varpi^m}^\tor$ for varying $N\geq 1$, and since these give rise to unramified representations, we deduce Theorem \ref{TI:main} (cf. Theorem \ref{T:main}, and Theorem \ref{T:conj_main}).

\begin{remark}

We remark that the only point in which we used $g=2$ in the argument is to show that a $4$-dimensional quotient of the local deformation ring $R_p^\square$ was acting on our (framed) cohomology (namely, $R_p^\square/\mathscr{I}^3$). The argument just described would go through exactly in the same way for an arbitrary degree $g=[F:\QQ]$ and arbitrary $p>3$ if one could prove the following conjecture (cf. Conjecture \ref{C:alpha}):
\end{remark}

\noindent{\bf Conjecture.} \emph{Let $F$ be a totally real number field of degree $g$ and denote by $\mathscr{I}$ the ideal cutting out the unramified locus of the deformation ring $R_p^\square$. There exists a positive integer $n$ depending on $g$ such that:} $$\mathscr{I}^n\cdot H^\bullet(\calM^\tor_{\calO/\varpi^m},\omega^{({\bf 1}, -1)})_\gothm^\square = 0.$$

\subsection{Organization of the paper}
In 
Section \ref{Sec:section 2} we recall the definition of the Hilbert Shimura varieties, and the stratification induced on them by generalized partial Hasse invariants. Since we allow $p$ to ramify in $F$, we work with the Pappas--Rapoport splitting model $\Sh^\PR$. In Section \ref{Sec:section 3} we construct the Hecke operator $T_\gothp^\rmn$ acting on the cohomology of $\Sh^{\PR}_{\calO/\varpi^m}$ ($m\leq\infty$) with coefficient in an automorphic sheaf of paritious weight $\kappa$ (satisfying some conditions). This construction relies on  key Proposition~\ref{P:ample=>small support}, which is subsequently proved in Section~\ref{Sec:section 4}. In Section \ref{Sec:section 5} we use the construction of $T_\gothp^\rmn$ operator to prove unramifiedness of the representations arising from (non-liftable) Katz Hilbert modular \emph{forms} of weight ${\bf 1}$, and then we specialize to the case $g=2$ to prove unramifiedness of Hilbert modular \emph{classes}.

\subsection*{Acknowledgments}
We would especially like to thank Frank
Calegari and David Geraghty for their interest in and comments on the results
of this paper. Our debit to their work is clear: the project of studying Galois representations arising from torsion classes in coherent cohomology originated from their Conjectures A and B of \cite{CG}, and their techniques for patching complexes constitute one of the main ingredients for the proof of our main result. 
We thank Vincent Pilloni for pointing out a serious gap in an early version of this paper. We thank the anonymous referees for careful reading of the paper and many useful comments.
We are
also grateful to Mladen Dimitrov, Kai-Wen Lan, Yichao Tian, Akshay Venkatesh, and  Gabor Wiese for helpful conversations and their interest in our results.

\section{Splitting models of Hilbert modular schemes}
\label{Sec:section 2}

We present here some preliminaries on splitting models of Hilbert modular schemes, as constructed by  Pappas and  Rapoport in \cite{pappas-rapoport}; we follow \cite{reduzzi-xiao}. We also recall the stratification induced by suitable generalized partial Hasse invariants on the special fiber of such models (cf. \cite{reduzzi-xiao}). Along the way, we make a small generalization following \cite[\S 2.3]{tian-xiao2} to allow more general tame level structure, which is required for later application to Taylor--Wiles--Kisin and Calegari--Geraghty patching argument.

\subsection{Setup}
\label{S:setup}
Let $\overline \QQ$ denote the algebraic closure of $\QQ$ inside $\CC$.
We fix a rational prime $p$ and a field isomorphism $\overline \QQ_p \simeq \CC$.
Base changes of algebras and schemes will often be denoted by a subscript, if no confusion arises.

Let $F$ be a totally real field of degree $g>1$, with ring of integers $\calO_F$ and group of totally positive units $\calO_F^{\times, +}$.
Denote by $\gothd:=\gothd_F$ the different ideal of $F/ \QQ$.
Let $\gothC: = \{ \gothc_1, \dots, \gothc_{h^+}\}$ be a fixed set of representatives for the elements of the narrow class group of $F$, chosen to be coprime to $p$.
For a nonzero ideal $\gotha$ of $F$, we write $\NN(\gotha)$ for its norm $\#(\calO_F/\gotha)$.

We fix a large enough coefficient field $E$ which is a finite Galois extension of $\QQ_p$ inside $\overline \QQ_p$. We require that $E$ contains the images of all $p$-adic embeddings of $F(\sqrt u; u \in \calO_F^{\times, +})$ into $\overline \QQ_p$.
Let $\calO$ denote the valuation ring of $E$; choose a uniformizer $\varpi$ of $\calO$ and denote by $\FF$ the residue field.

We write the prime ideal factorization of $p \calO_F$ as $\gothp_1^{e_1} \cdots \gothp_r^{e_r}$, where $r$ and $e_i$ are positive integers.
We also set $\FF_{\gothp_i} = \calO_F / \gothp_i$ and $f_i = [\FF_{\gothp_i}: \FF_p]$.
Let $\overline \FF_p$ denote the residue field of $\overline \ZZ_p$, and let $\sigma$ denote the arithmetic Frobenius on $\overline \FF_p$.
We label the embeddings of $\FF_{\gothp_i}$ into $\overline \FF_p$ (or, equivalently, into $\FF$) as $\{\tau_{\gothp_i, 1}, \dots, \tau_{\gothp_i, f_i}\}$ so that $\sigma \circ \tau_{\gothp_i, j} = \tau_{\gothp_i, j+1}$ for all $j$, with the convention that $\tau_{\gothp_i, j}$ stands for $\tau_{\gothp_i, j \pmod {f_i}}$.
For each $\gothp_i$, denote by $F_{\gothp_i}$ the completion of $F$ for the $\gothp_i$-adic topology. Let  $W(\FF_{\gothp_i})$ denote the ring of integers of the maximal subfield of ${F_{\gothp_i}}$ unramified over $\QQ_p$. The residue field of $W(\FF_{\gothp_i})$ is identified with $\FF_{\gothp_i}$. Each embedding $\tau_{\gothp_i, j}:\FF_{\gothp_i} \to \FF$ of residue fields induces an embedding $W(\FF_{\gothp_i}) \to \calO$ which we denote by the same symbol.

Let $\Sigma$ denote the set of embeddings of $F$ into $\overline \QQ$, which is further identified with the set of embeddings of $F$ into $\CC$ or $\overline \QQ_p$ or $E$.
Let $\Sigma_{\gothp_i}$ denote the subset of $\Sigma$ consisting of all the $p$-adic embeddings of $F$ inducing the $p$-adic place $\gothp_i$.
For each $i$ and each $j =1, \dots, f_i$, there are exactly $e_i$ elements in $\Sigma_{\gothp_i}$ that induce the embedding $\tau_{\gothp_i, j}:W(\FF_{\gothp_i}) \to \calO$; we label these elements as $\tau_{\gothp_i, j}^1, \dots, \tau_{\gothp_i, j}^{e_i}$.  \emph{There is no canonical choice of such labelling, but we fix one for the rest of this paper}.

We choose a uniformizer $\varpi_i$ for the ring of integers $\calO_{F_{\gothp_i}}$ of $F_{\gothp_i}$.  Let $E_{\gothp_i}(x)$ denote the minimal polynomial of $\varpi_i$ over the ring $W(\FF_{\gothp_i})$: it is an Eisenstein polynomial.
Using the embedding $\tau_{\gothp_i,j}$, we can view this polynomial as an element $E_{\gothp_i,j}(x) := \tau_{\gothp_i,j}(E_{\gothp_i}(x))$  of $\calO[x]$.  We have:
\[
E_{\gothp_i,j}(x) = (x - \tau_{\gothp_i, j}^1(\varpi_i)) \cdots (x - \tau_{\gothp_i, j}^{e_i}(\varpi_i))
.
\]

\subsection{Moduli space of Hilbert--Blumenthal abelian varieties with additional information}
\label{S:moduli HBAS}
Let $S$ be a locally Noetherian $\calO$-scheme. A \emph{Hilbert--Blumenthal abelian $S$-scheme} (HBAS) with real multiplication by $\calO_F$ is the datum of an abelian $S$-scheme $A$ of relative dimension $g$, together with a ring embedding $ \calO_F \to \End_SA$.
We have natural direct sum decompositions
\[
\omega_{A/S} = \bigoplus_{i=1}^r \omega_{A/S, \gothp_i} =  \bigoplus_{i=1}^r \bigoplus_{j = 1}^{f_i}\omega_{A/S, \gothp_i, j},
\]
where $W(\FF_{\gothp_i}) \subseteq \calO_{F_{\gothp_i}}$ acts on $\omega_{A/S, \gothp_i,j}$ via $\tau_{\gothp_i, j}$.

Let $\gothc \in \gothC$ be a fractional ideal of $F$ (coprime to $p$), with cone of positive elements $\gothc^+$.
We say that a HBAS $A$ over $S$ is \emph{$\gothc$-polarized} if
there is an $S$-isomorphism $\lambda: A \otimes _{\calO_F} \gothc \xrightarrow \sim A^\vee $ of HBASs under which the symmetric elements (resp. the polarizations) of $\Hom_{\calO_F}(A, A^\vee)$ correspond to the elements of $\gothc$ (resp. $\gothc^+$) in $\Hom_{\calO_F}(A, A \otimes_{\calO_F} \gothc)$.
For such a HBAS, each $\omega_{A/S, \gothp_i, j}$ is a locally free sheaf over $S$ of rank $e_i$.

We define the tame level structures following \cite[\S 2.3]{tian-xiao2}.
For a positive integer $N$ relatively prime to $p$, a \emph{principal level $N$-structure} on a HBAS $A$ over $S$ is an $\calO_F$-linear isomorphism of finite \'etale group schemes over $S$:
\[
\alpha_N: (\calO_F/N\calO_F)^{\oplus 2} \xrightarrow{\sim}A[N].
\]
In general, let $\widehat{\calO}_F$ (resp. $\widehat{\calO}_F^{(p)}$) denote the direct product of completions of $\calO_F$ at all finite places (resp. all finite places relatively prime to $p$).
For $K^p$ an open compact subgroup of $\GL_2(\widehat{\calO}_F^{(p)})$, we choose a positive integer $N$ relatively prime to $p$ such that $K^p$ contains 
\begin{equation}
\label{E:K(N)p}
K(N)^p = \big\{ \big(\begin{smallmatrix}a & b\\c&d
\end{smallmatrix} \big) \in \GL_2(\widehat{\calO}_F^{(p)})\;|\; a-1, b, c, d-1 \equiv 0 \bmod N\big\}.
\end{equation}
Then a \emph{$K^p$-level structure} on a HBAS $A$ over $S$ is a collection of, for each connected component $S_i$ of $S$ with a fixed
geometric point $\bar s_i \in S_i$, a $\pi_1(S_i
, \bar s_i)$-invariant $K^p/K(N)^p$-orbit of $\alpha_{N,\bar s_i}$ above. This does not depend on the choices of $N$ and $\bar s_i$.
We put $K_p = \GL_2(\calO_{F, p})$ and $K = K^pK_p$.
When $K^p = \big\{ \big(\begin{smallmatrix}a & b\\c&d
\end{smallmatrix} \big) \in \GL_2(\widehat{\calO}_F^{(p)})\;|\; c, d-1 \equiv 0 \bmod \calN\big\}$, the notion of a $K^p$-level structure recovers the notion of $\Gamma_{00}(\calN)$-level structure considered in \cite{emerton-reduzzi-xiao} and \cite{reduzzi-xiao}. Most of the statements of \cite{emerton-reduzzi-xiao} and \cite{reduzzi-xiao} continue to hold for this general $K^p$-level structure; we shall point out the change in case it is needed.




Fix an open subgroup $K^p \subseteq \GL_2(\widehat{\calO}_F^{(p)})$. Denote by $\underline \calM^\PR_\gothc = \underline \calM^\PR_{\gothc, K}$ the functor that assigns to a locally Noetherian $\calO$-scheme $S$ the set of isomorphism classes of tuples $(A, \lambda, \alpha, \underline \scrF)$, where
\begin{itemize}
\item
$(A, \lambda)$ is a $\gothc$-polarized HBAS over $S$ with real multiplication by $\calO_F$,
\item
$\alpha$ is a $K^p$-level structure, and
\item $\underline \scrF$ is a collection $(\scrF_{\gothp_i, j}^{(l)})_{i = 1, \dots, r; j = 1, \dots, f_i;l = 0, \dots, e_i}$ of locally free sheaves over $S$ such that
\begin{itemize}
\item $0 =\scrF_{\gothp_i, j}^{(0)} \subsetneq \scrF_{\gothp_i, j}^{(1)} \subsetneq \cdots \subsetneq
\scrF_{\gothp_i, j}^{(e_i)} = \omega_{A/S, \gothp_i, j}$ and each $\scrF_{\gothp_{i, j}}^{(l)}$ is stable under the $\calO_F$-action (not just the action of $W(\FF_{\gothp_i})$),
\item each subquotient $\scrF_{\gothp_i, j}^{(l)} / \scrF_{\gothp_i, j}^{(l-1)}$ is a locally free $\calO_S$-module of rank one (and hence $\rank_{\calO_S} \scrF_{\gothp_i, j}^{(l)} = l$), and
\item the action of $\calO_F$ on each subquotient $\scrF_{\gothp_i, j}^{(l)} / \scrF_{\gothp_i, j}^{(l-1)}$ factors through $\calO_F \xrightarrow{\tau_{\gothp_i, j}^l} \calO$, or equivalently, this subquotient is annihilated by $[\varpi_i] - \tau_{\gothp_i,j}^l(\varpi_i)$, where $[\varpi_i]$ denotes the action of $\varpi_i$ as an element of $\calO_{F_{\gothp_i}}$.
\end{itemize}
\end{itemize}
We use $\underline \calM_\gothc^\DP$ to denote the functor obtained from $\underline \calM_\gothc^\PR$ by forgetting the filtrations $\underline \scrF$.

Both $\underline \calM_\gothc^\PR$ and $\underline \calM_\gothc^\DP$ carry an action of $\calO_F^{\times, +}$:
\begin{equation}
\label{E:action <u>}
\textrm{for }u \in \calO_F^{\times,+}, \quad \langle u \rangle: (A, \lambda, \alpha, \underline \scrF) \longmapsto (A, u\lambda, \alpha, \underline \scrF).
\end{equation}
This action is trivial on the subgroup $(K \cap \calO_{F}^\times)^2$ of $\calO_F^{\times,+}$.

It is well known (cf. \cite[Proposition~2.4]{reduzzi-xiao}) that, when $K^p$ is sufficiently small, the functor $\underline \calM_\gothc^\mathrm{DP}$ (resp. $\underline \calM_\gothc^\mathrm{PR}$) is represented by an $\calO$-scheme of finite type, that we denote $\calM_\gothc^\DP$ (resp. $\calM_\gothc^\PR$). Moreover, the moduli space  $\calM_\gothc^\mathrm{DP}$ is normal (\cite[Corollaire 2.3]{deligne-pappas}). Let $\calM_\gothc^\mathrm{Ra}$ denote its smooth locus, called the \emph{Rapoport locus} (\cite{rapoport}).  Then $\calM_\gothc^\mathrm{Ra}$ is the open subscheme of $\calM_\gothc^\mathrm{DP}$ parameterizing those HBASs for which  $\omega_{A/S}$ is a locally free $(\calO_{F} \otimes_{\ZZ} \calO_S)$-module of rank one. The natural morphism 
\[
\pi:\calM_\gothc^\mathrm{PR} \to \calM_\gothc^\mathrm{DP}
\] 
is projective, and it induces an isomorphism of an open subscheme of $ \calM_\gothc^\mathrm{PR}$ onto $\calM_\gothc^\mathrm{Ra}$ (\cite[Proposition~2.4(3)]{reduzzi-xiao}). As in \cite[Theorem 2.9]{reduzzi-xiao}, one can prove that $\calM^\PR_\gothc$ is smooth over $\calO$.

By \cite[Proposition~2.4 and Lemma~2.5]{tian-xiao2}, we may always shrink $K^p$ so that the following hypothesis holds, which we then assume from now on.
\begin{hypothesis}
\label{H:Kp neat}
We assume the following for the open subgroup $K^p \subseteq \GL_2(\widehat{\calO}_F^{(p)})$.
\begin{enumerate}
\item
$K^p$ is of the form $\prod_{\gothq \nmid p} K_\gothq$ for an open subgroup $K_\gothq \subseteq \GL_2(F_\gothq)$ such that $K_\gothq = \GL_2(\calO_{F_\gothq})$ for all but finitely many place $\gothq$,
\item
$K^p$ is sufficiently small so that each $\calM_\gothc^\DP$ and each $\calM_\gothc^\PR$ are $\calO$-schemes of finite type, and
\item
$K^p$ is sufficiently small so that the action of $\calO^{\times, +}_F / (K^pK_p \cap \calO_F^\times)^2$ on each $\calM_\gothc^\DP$ and each $\calM_\gothc^\PR$ are free on geometric points. 
\end{enumerate}
\end{hypothesis}

\subsection{Integral model of Shimura varieties}
\label{S:integral model Sh var}
In virtue of Hypothesis~\ref{H:Kp neat}, the quotients of $\calM_\gothc^\PR$ and $\calM_\gothc^\DP$ by the group $\calO^{\times, +}_F / (K^pK_p \cap \calO_F^\times)^2$, denoted by  $ \Sh_\gothc^\PR$ and $ \Sh_{\gothc}^\DP$ respectively, are $\calO$-schemes of finite type, and the quotient morphisms are \'etale.
We point out that when $K^p$ satisfies Hypothesis~\ref{H:Kp neat}, conditions (2) and (3) also hold for any open subgroup $K'^p \subseteq K^p$ as long as $K'^pK_p \cap \calO_F^\times = K^pK_p \cap \calO_F^\times$. 

For $? \in\{\mathrm{DP}, \mathrm{PR}, \mathrm{Ra}\}$ we denote by $\calA_\gothc^?$ the universal abelian scheme over $\calM_\gothc^?$.  We set
\[
\calM^? := \coprod_{\gothc \in \gothC} \calM_\gothc^?, \quad \Sh^? := \coprod_{\gothc \in \gothC} \Sh_\gothc^?, \quad \textrm{and} \quad \calA^? := \coprod_{\gothc \in \gothC} \calA_\gothc^?.
\]
Notice that the universal abelian scheme $\calA^?$ may not descent to $\Sh^?$.

Denote by $\omega_{\calA^?/\calM^?}$ the pullback via the zero section of the sheaf of relative differentials of $\calA^?$ over $\calM^?$. We let $\underline \calF = (\calF_{\gothp_i,j}^{(l)})$ denote the universal filtration of $\omega_{\calA^\PR/\calM^\PR}$. For each $p$-adic embedding $\tau = \tau_{\gothp_i, j}^l$ of $F$ into $\overline \QQ_p$,
we set
\[
\dot\omega_\tau: = \calF_{\gothp_i, j}^{(l)} / \calF_{\gothp_i, j}^{(l-1)};
\]
it is an automorphic line bundle on the splitting model $\calM^\PR$. The additional dot in the notation $\dot \omega_\tau$ is placed in order to distinguish this sheaf from its descent $\omega_\tau$ to $\Sh^\PR$, which will be introduced later. We provide $\dot \omega_\tau$ with an action of $\calO_F^{\times, +}$ following \cite[\S4]{DT}: a positive unit $u \in \calO_F^{\times, +}$ maps a local section $s$ of $\dot \omega_\tau$ to $\tau(u)^{-1/2} \cdot \langle u \rangle^*(s)$, where $\langle u \rangle$ is defined by \eqref{E:action <u>}.  It is clear that this action factors through $\calO_F^{\times, +} / (K \cap \calO_F^\times)^2$.

Similarly, for each $p$-adic embedding $\tau$ of $F$, we define
\begin{align}
\nonumber
\dot \varepsilon_\tau & := \big( \wedge^2_{\calO_F\otimes_\ZZ \calO_{\calM_\gothc^\PR} } \calH^1_\dR(\calA^\PR/{\calM_\gothc^\PR}) \big) \otimes_{\calO_F\otimes_\ZZ \calO_{\calM_\gothc^\PR}, \tau\otimes1} \calO_{\calM_\gothc^\PR}
\\
\label{E:dot varepsilon}
&
\cong (\gothc \gothd^{-1}\otimes_{\ZZ} \calO_{\calM_\gothc^\PR} ) \otimes_{\calO_F\otimes_\ZZ \calO_{\calM_\gothc^\PR}, \tau \otimes 1} \calO_{\calM_\gothc^\PR},
\end{align}
where the canonical isomorphism is induced by the universal polarization as in \cite[Lemma 2.5]{reduzzi-xiao}. While $\dot\varepsilon_\tau$ is a trivial line bundle as seen in \eqref{E:dot varepsilon}, it carries a nontrivial action of $\calO_F^{\times,+} / (K \cap \calO_F^\times)^2$ given as follows: a positive unit $u \in \calO_F^{\times, +}$ maps a local section $s$ of $\dot \varepsilon_\tau$ to $\tau(u)^{-1} \cdot \langle u \rangle^*(s)$.

It is proved in \cite[Theorem 2.9]{reduzzi-xiao} that the sheaf of relative differentials $\Omega^1_{\calM^\mathrm{PR} / \calO}$ admits a canonical Kodaira--Spencer filtration whose successive subquotients are given by
\begin{equation}
\label{E:KS filtration}
\dot  \omega_{\tau}^{\otimes 2}  \otimes_{\calO_{\calM^{\PR}}} \dot \varepsilon_{\tau}^{\otimes-1}\quad 
\textrm{ for } \tau \in \Sigma.
\end{equation}


\subsection{Toroidal compactifications}
\label{S:toroidal}
We need to slightly generalize the notation of cusps from standard references, such as \cite[\S 3]{dimitrov}, \cite[\S 1.6]{kisin-lai}, and \cite{rapoport}.
For an ideal $\mathfrak{c}\in\mathfrak{C}$, a \emph{$K^p$-cusp} is an equivalence class of tuples $\calC = (\gotha, \gothb, L, i, j, \lambda, \alpha)$, where
\begin{itemize}
\item
$\gotha$ and $\gothb$ are two fractional ideals of $F$, relatively prime to $p$,  such that $\gotha \gothb^{-1} = \gothc$,
\item
$L$ is an $\calO_F$-lattice of $F^2$ that sits in an exact sequence
\[
0 \to \gotha^{-1} \gothd^{-1} \xrightarrow{i} L \xrightarrow{j} \gothb \to 0,
\]
\item
$\lambda: \wedge^2_{\calO_F} L \xrightarrow{\sim}\gothc^{-1}\gothd^{-1}$ is an isomorphism of $\calO_F$-modules, and
\item
$\alpha$ is a $K^p/K(N)^p$-orbit of isomorphisms $(\calO_F/N\calO_F)^{\oplus 2} \xrightarrow{\sim} N^{-1}L/ L$ of $\calO_F$-modules, where $N$ is
a positive prime-to-$p$ integer such that $K^p$ contains $K(N)^p$ (as defined in \eqref{E:K(N)p}).
\end{itemize}
This definition does not depend on the choice of $N$.
The equivalence of the tuples is as explained in \cite[D\'efinition 3.1]{dimitrov} (plus requiring it to be compatible with the level structure $\alpha$).  In particular, for a fixed $(\gotha, \gothb, \lambda)$, $L$ is always non-canonically isomorphic to $\gotha^{-1} \gothd^{-1} \oplus \gothb$ by choosing a section of $j$, and such sections form a torsor under $\gotha\gothb\gothd$. So the essential choice of $\alpha$ is parametrized by the double cosets
\begin{equation}
\label{E:double cosets}
\big( \begin{smallmatrix}
1 & \calO_F/N\calO_F \\ 0 & 1
\end{smallmatrix} \big)\big\backslash\GL_2(\calO_F/N\calO_F) \big/ \big(K^p / K(N)^p\big).
\end{equation}
We write $(\gotha, \gothb, \lambda, [\gamma_\calC])$ for the corresponding cusp, where $\gamma  _\calC\in \GL_2(\calO_F/N\calO_F)$ is a representative of the double coset in \eqref{E:double cosets}.
We write $X_{[\gamma_\calC]}$ for the fractional ideal of $\calO_F$ contained in $\frac {1}{N}\gotha\gothb$ such that 
\[
\tfrac{1}{N}\gotha\gothb / X_{[\gamma_\calC]}   \cong \Big(\gamma_\calC\big( K^p/K(N)^p\big)\gamma_\calC^{-1} \Big)  \bigcap \Big( \begin{smallmatrix}
1 &\calO_F/N\calO_F \\0 & 1
\end{smallmatrix} \Big),
\]
and let $X_{[\gamma_\calC]}^{*+}$ denote the totally positive elements in $X_{[\gamma_\calC]}^*: = X_{[\gamma_\calC]}^{-1}\gothd ^{-1} \subseteq F \otimes_\QQ \RR$. 

To proceed, we fix  a smooth rational polyhedral
admissible cone decomposition $\Phi_{[\gamma_\calC]}$ of $X_{[\gamma_\calC]}^{*+}$ for each isomorphism
class of $K^p$-cusps. 
Now, one can apply the construction in \cite[\S 4--5]{rapoport} for the level $K^p = K(N)^p$ and then take invariants under $K^p / K(N)^p$-action,  (or make use of the more general reference \cite[\S 6]{lan}).
This way, we construct 
the toroidal compactification $\calM_\gothc^{\Ra, \tor}$ of $\calM_\gothc^{\Ra}$. Gluing over $\calM_\gothc^\Ra$ gives toroidal compactifications $\calM_\gothc^{\PR, \tor}$ and $\calM_\gothc^{\DP, \tor}$ of the $\calO$-schemes  $\calM_\gothc^{\PR}$ and $\calM_\gothc^{\DP}$.
The scheme $\calM_\gothc^{\PR, \tor}$ is proper and smooth over $\Spec \calO$. We set $\mathcal{M}^{\operatorname*{?,tor}}:=
{\textstyle\coprod\nolimits_{\mathfrak{c}\in\mathfrak{C}}}
\mathcal{M}_{\mathcal{\mathfrak{c}}}^{\operatorname*{?, tor}}$ for $?\in\{\DP, \PR, \Ra$\}.
The boundary $\dot \ttD:=\mathcal{M}^{?,\operatorname*{tor}}-\mathcal{M}^?$ is a
relative simple normal crossing divisor on $\mathcal{M}^{?,\operatorname*{tor}}$.

Let $\Sh^{?, \tor}$ denote the quotient of $\calM^{?,\tor}$ by the action of the group $\calO_F^{\times, +} / (K \cap \calO_F^\times)^2$.
Put $\ttD: = \Sh^{?,\tor} - \Sh^?$; it is the quotient of $\dot \ttD$ and it is a divisor with simple normal crossings.

For any $\tau \in \Sigma$, there are automorphic line bundles $\dot\omega_\tau^\tor$ and $\dot\varepsilon_\tau^\tor$ over $\calM^{\PR, \tor}$ constructed as in \cite[\S 2.10]{reduzzi-xiao}. To lighten the load on notation, we will simply write $\dot \omega_\tau$ and $\dot\varepsilon_\tau$ to denote these sheaves when no confusion arises.
Both these line bundles carry natural actions of $\calO_F^{\times, +} / (K \cap \calO_{F}^\times)^2$ as described earlier, and they descend to line bundles over $\Sh^{\PR,\tor}$, which we denote by $\omega_\tau$ and $\varepsilon_\tau$ respectively.  We warn the reader that $\varepsilon_\tau$ may not be the trivial line bundle over $\Sh^{\PR,\tor}$; however $\varepsilon^{\mathbf 1}: = \otimes_{\tau \in \Sigma} \varepsilon_\tau$ is a trivial line bundle because the action of $\calO_F^{\times,+}$ on the tensor product of all $\dot \varepsilon_\tau$ is the na\"ive pullback.

We point out that, while the compactification depends on the choice of the smooth rational polyhedral cone decompositions, the cohomology groups $H^*(\Sh^{\PR, \tor}, \omega ^\kappa)$ and $H^*(\Sh^{\PR, \tor}, \omega ^\kappa(-\ttD))$ we define below are independent of this choice (up to canonical isomorphisms) because the derived pushforward of $\omega^\kappa$ and $\omega^\kappa(-\ttD)$ from the compactification associated to a finer cone decomposition to a courser one are canonically isomorphic to $\omega^\kappa$ and $\omega^\kappa(-\ttD)$ (cf. \cite[Lemma~7.1.1.4]{lan} and \cite[Proposition~7.5]{lan2}.) 

\begin{remark}
If one only cares about Hilbert modular forms of parallel weights, then defining them using the splitting model and the usual Deligne-Pappas model are the same.
More precisely, the line bundle $\varepsilon^{\bf 1}$ on $\Sh^{\PR, \tor}$ is the pullback of an automorphic line bundle $\varepsilon^{(\bf 1), \DP}$ on $\Sh^{\DP, \tor}$, given by extending $\wedge^g \omega_{A/S}$ to the toroidal compactification, and then descending along the Galois cover $\Sh^{\DP, \tor} \to \calM^{\DP, \tor}$.
Then for any $k \in \ZZ$, pulling back along the natural map $\pi: \Sh^{\PR, \tor} \to \Sh^{\DP, \tor}$ gives an isomorphism 
\begin{equation}
\label{E:splitting model = DP model}\pi^*: H^0\big(\Sh^{\DP, \tor}, \omega^{{\bf k}}\big) \xrightarrow{\ \cong\ } H^0\big(\Sh^{\PR, \tor}, \omega^{{\bf k}}\big).
\end{equation}
Indeed, since $\Sh^{\DP, \tor}$ is normal (\cite[Corollaire 2.3]{deligne-pappas}) and $\pi$ is birational, the Zariski Main Theorem says that $\pi_* \calO_{\Sh^{\PR, \tor}} = \calO_{\Sh^{\DP, \tor}}$. The projection formula then proves \eqref{E:splitting model = DP model}.
\end{remark}

\begin{notation}
\label{N:simplify notation}
For the rest of this paper, if no confusion arises, we will drop the superscript $\PR$ appearing in the schemes introduced in the previous subsections. In particular, 
we set $\calA : = \calA^\PR$, $\calM: = \calM^\PR$, and $\Sh: = \Sh^\PR$. These are schemes over $\calO$, and we denote their special fibers by $\calA_\FF$, $\calM_\FF$, and $\Sh_\FF$ respectively.
\end{notation}

\subsection{Geometric Hilbert modular forms}
\label{S:geometric HMF}
A \emph{paritious weight} $\kappa$ is a tuple $((k_\tau)_{\tau \in \Sigma}, w) \in \ZZ^{\Sigma} \times \ZZ$ such that $k_\tau \equiv w \pmod 2$ for every $\tau \in \Sigma$.
We say that $\kappa$ is \emph{regular} if moreover $k_\tau >1$ for all $\tau\in\Sigma$.
For an integer $n$, we write ${\bf n} = (n, \dots, n)$.

For $\kappa = ((k_\tau)_{\tau \in \Sigma}, w)$ a paritious weight, we define
\[
\dot \omega^\kappa: = \bigotimes_{\tau \in \Sigma} \big(\dot  \omega_{\tau}^{\otimes k_\tau} \otimes_{\calO_{\calM^{\PR, \tor}}}\dot  \varepsilon_{\tau}^{\otimes (w-k_\tau)/2} \big) \textrm{\quad and\quad }\omega^\kappa: = \bigotimes_{\tau \in \Sigma} \big( \omega^{\otimes k_\tau}_{\tau} \otimes_{\calO_{\Sh^{\PR, \tor}}} \varepsilon_{\tau}^{\otimes (w-k_\tau)/2} \big).
\]
They are line bundles over $\calM^{\tor}$ and $\Sh^{ \tor}$, respectively.  We remind the reader that these line bundles (and also $\calM^{ \tor}$ and $\Sh^{\tor}$) depend on the fixed choice of an ordering of the $\tau_{\gothp_i, j}^l$'s.

We point out that our definition of $\omega^{((k_\tau)_{\tau \in \Sigma}, w)}$ is consistent with \cite{emerton-reduzzi-xiao, reduzzi-xiao}, but is equal to the $\omega^{((k_\tau)_{\tau \in \Sigma}, w+2)}$ in \cite{tian-xiao2} due to different normalizations.
In fact, the usual modular forms of weight $k$ (for $F =\QQ$) should be, rigorously speaking, of weight $(k, k-2)$ in our notation, i.e. a section of $\omega^k \otimes \varepsilon^{-1}$. However, $\varepsilon$ is canonically a trivial bundle over the modular curve, it is often omitted from the discussion.

For example, by the Kodaira--Spencer filtration on $\Omega^1_{\calM/\calO}$ we just recalled before \S\ref{S:toroidal} with successive subquotients \eqref{E:KS filtration}, we deduce canonical isomorphisms (extended to the toroidal compactification; cf \cite[2.11(4)]{tian-xiao2})
\begin{equation}
\label{E:KS isomorphisms}
KS: \wedge^g_{\calO_{\calM^\tor}} \Omega^1_{\calM^\tor / \calO}(\dot \ttD) \cong \dot \omega^{(\boldsymbol{2}, 0)} \quad \textrm{and} \quad KS: \wedge^g_{\calO_{\Sh^\tor}} \Omega^1_{\Sh^\tor / \calO}(\ttD) \cong  \omega^{(\boldsymbol{2}, 0)}.
\end{equation}

A \emph{(geometric) Hilbert modular form} over a Noetherian $\calO$-algebra $R$ of level $K^p$ and paritious weight $\kappa$ is an element of the finite $R$-module $H^0(\Sh_R^{ \tor}, \omega^\kappa_R)$,
where the subscript $R$ indicates base change to $R$ over $\calO$.
Such a form is called \emph{cuspidal} if it belongs to the submodule $H^0(\Sh_R^{ \tor}, \omega^\kappa_R(-\ttD))$.
By the K\"ocher principle (\cite[Th\'eor\`em~7.1]{DT}), if $[F:\QQ]>1$, we have
\[
H^0(\calM_R^{ \tor}, \dot \omega_R^\kappa) = H^0(\calM_R, \dot \omega_R^\kappa), \textrm{ and hence } H^0(\Sh_R^{ \tor}, \omega_R^{\kappa}) = H^0(\Sh_R, \omega^\kappa_R).
\]

Fix a positive prime-to-$p$ integer $N$ such that $K^p \supseteq K(N)^p$.
At each cusp of $\calM_\gothc$ labeled by the tuple $\calC = (\gotha, \gothb, L, i, j,\lambda, \alpha)$, there is a
Tate object (see e.g. \cite[proof of Th\'eor\`em~7.2]{dimitrov}) with additional structure $(\Tate_{\frak{a},\frak{b}},\lambda_\can,\alpha_\can, \underline\eta_\can,\underline\zeta_\can)$ defined over a suitable subring of the ring of formal power series $\calO[[q^\xi:\xi\in X_{[\gamma_\calC]}^+]]$ (this involves a choice of cone decomposition that we omit here since it will not play any role).
Here $X_{[\gamma_\calC]}^+$ is the subset of totally positive elements in $X_{[\gamma_\calC]}$,  and  $\underline \eta_\can$ (resp. $\underline \zeta_\can$) is a collection of canonical bases of $\dot \omega_\tau$ (resp. $\dot \varepsilon_\tau$) for the Tate abelian variety.
Evaluating a (geometric) Hilbert modular form $f$ at this Tate object, we obtain the $q$-expansion of $f$ at this cusp: \begin{equation}
\label{E:q-expansion general}
f(\Tate_{\frak{a},\frak{b}},\lambda_\can, \alpha_\can, \underline\eta_\can,\underline\zeta_\can)=\sum_{\xi\in X_{[\gamma_\calC]}^+} a_\xi(f,\calC, \Tate_{\frak{a},\frak{b}})\, q^\xi
\in\calO/(\varpi^m)[[q^\xi:\xi\in X_{[\gamma_\calC]}^+]].
\end{equation} 
We point out that when changing the tuple $\calC = (\gotha, \gothb, L, i, j,\lambda, \alpha)$ representing the cusp to an equivalent (cf. \cite[D\'efinition~3.1]{dimitrov}) tuple  $\calC_t = (t\gotha, t\gothb, L, ti, tj,\lambda, \alpha)$ for $t \in F^\times$ and hence $X_{[\gamma_{\calC_t}]}^+ = t^2 X_{[\gamma_\calC]}^+$, we have
\begin{equation}
\label{E:q-expansion equivalent cusp}
a_\xi(f,\calC, \Tate_{\gotha, \gothb}) = \prod_{\tau}\tau(t)^{-k_\tau} \cdot a_{t^2\xi}(f,\calC_t, \Tate_{t\gotha, t\gothb})
\end{equation}

\subsection{Tame Hecke operators}
\label{S:tame Hecke operators}
We recall the definition of tame Hecke operators from \cite[\S 2.14]{reduzzi-xiao}.
Let $\calS$ denote the a finite set of places of $F$ containing the archimedean places,  $p$-adic places, and all the places $\gothq$ where $K_\gothq \neq \GL_2(\calO_{F_\gothq})$.
We define the Hecke operator $T_\gothq$ for $\gothq \notin \calS$ in the same way as in  \cite[\S2.14]{reduzzi-xiao}. (Note that our new setup allows more general level structures in place of the usual $\Gamma_{00}(\calN)$-level structure, but one can adapt the  construction in {\it loc. cit.} to our new tame level structure formulation.)
We take this opportunity to point out a mistake in the formulas \cite[(2.14.1)]{reduzzi-xiao} and \cite[(2.2.9.1)]{emerton-reduzzi-xiao}. The factor $\frac{1}{\mathrm{Nm}^F_\QQ(\gotha)}$ there should be removed. This is because the Kodaira--Spencer isomorphisms used in {\it loc. cit.} are canonical (so there is no additional twists for Hecke operators).

To benefit the readers, we make explicitly this Hecke operator on the level of $q$-expansions.
We fix $\gothc \in \gothC$ and a positive element $\vartheta_\gothc \in F^\times$ such that  $\gothc'\vartheta_\gothc = \gothc \gothq$ for some $\gothc' \in \gothC$.
Consider a cusp of $\calM_\gothc$ labeled by $\calC = (\gotha, \gothb, L, i, j,\lambda, \alpha)$. Then the image of $\pi_2(\pi_1^{-1}(-))$ of this cusp consists of two (possibly isomorphic) cusps:
\[
\calC' =  (\gotha, \vartheta_\gothc \gothb\gothq^{-1}, L', i,\vartheta_\gothc j, \vartheta_\gothc \lambda, \alpha') \quad \textrm{and} \quad
\calC'' = (\gotha \gothq, \vartheta_\gothc\gothb, L'', i,\vartheta_\gothc j, \vartheta_\gothc \lambda, \alpha'') ,
\]
where $L'$ and $L''$ are pullback and pushout of $L$ along the natural inclusions $\gothb \subset \gothb\gothq^{-1}$ and $\gotha^* \subset \gotha^* \gothq^{-1}$ respectively,  and $\alpha'$ and $\alpha''$ are the naturally induced level structures.
Then we have inclusions $X_{[\gamma_{\calC''}]} \subset \vartheta_\gothc X_{[\gamma_\calC]} \subset X_{[\gamma_{\calC'}]}$ preserving the positive cones, where the cokernel of each inclusion is isomorphic to $\calO_F/\gothq$.
For the natural Tate objects at these cusps and $\xi \in X_{[\gamma_\calC]}^+$, we have,
\begin{align}
\label{E:q-expansion Tp}
a_{\xi}&(T_\gothq(f),\calC, \Tate_{\gotha, \gothb}) = 
\prod_{\tau}\tau( \vartheta_\gothc)^{\frac{w-k_\tau}2}\cdot \Big(
(\NN(\gothq) \cdot a_{\vartheta_\gothc\xi}(f,\calC', \Tate_{\gotha,\vartheta_\gothc \gothb\gothq^{-1}}) + a_{\vartheta_\gothc\xi}(f,\calC'', \Tate_{\gotha \gothq, \vartheta_\gothc \gothb}) \Big)
\\ \nonumber
&= 
\prod_{\tau}\tau( \vartheta_\gothc)^{\frac{w-k_\tau+2}2}\cdot  a_{\vartheta_\gothc\xi}(f,\calC', \Tate_{\gotha,\vartheta_\gothc \gothb\gothq^{-1}}) + \prod_{\tau} \tau(\vartheta_\gothc)^{\frac{w+k_\tau}2} a_{\vartheta_\gothc^{-1}\xi}(f,\calC''_{\vartheta_\gothc^{-1}}, \Tate_{\vartheta_\gothc^{-1}\gotha \gothq, \gothb}) \Big).
\end{align}
Here the additional factor $\tau(\vartheta_\gothc)^{\frac{w-k_\tau}2}$ on the first line comes from different $\underline \zeta_\can$'s at different cusps.  The equality of the two lines follows from  \eqref{E:q-expansion equivalent cusp}, and $\calC''_{\vartheta_\gothc^{-1}} = (\vartheta_\gothc^{-1} \gotha \gothq, \gothb, L'', \vartheta_\gothc^{-1} i, j, \vartheta_\gothc \lambda, \alpha'')$ is obtained from $\calC''$ using the procedure just before \eqref{E:q-expansion equivalent cusp}.
(We shall see later that a normalization of the second line of \eqref{E:q-expansion Tp} is also valid for the Hecke operator $T_\gothp^\rmn$ we construct later.)

If for some prime-to-$p$ finite place $\gothq \in \calS$, the corresponding level structure $K_\gothq$ satisfies
\[
\begin{pmatrix}
1 & \calO_{F, \gothq}\\ 0 & 1 
\end{pmatrix} \subset K_\gothq \subset \begin{pmatrix}
\calO_{F, \gothq}^\times & \calO_{F, \gothq}\\ \gothq\calO_{F, \gothq}  & \calO_{F, \gothq} ^\times
\end{pmatrix},
\]then
there is a natural $U_\gothq$-operator associated to the double coset $K_\gothq \big(\begin{smallmatrix}
\varpi_\gothq &0\\0&1
\end{smallmatrix} \big)K_\gothq$, acting on the cohomology groups  $H^*(\Sh^\tor, \omega^\kappa)$ and $H^*(\Sh^\tor, \omega^\kappa(-\ttD))$.

We also define a Hecke operator $S_\gothq^\rmn$ for any finite place of $F$ (allowing $\gothq$ to divide $p$).
For each $\gothc \in \gothC$, let $\gothc'' \in \gothC$ denote the fixed representative of the strict ideal class of $\gothc \gothq^2$, so that $\vartheta_{\gothc^2}\gothc\gothq^2 = \gothc'' $ of a positive element $\vartheta_{\gothc^2} \in F^\times$.
We define an isomorphism
\[
\xymatrix@R=0pt{ \dot 
S_\gothq: \calM_\gothc \ar[r]^-\cong & \calM_{\gothc'}
\\
(A, \lambda, \alpha, \underline \scrF) \ar@{|->}[r] & (A' :=A \otimes_{\calO_F} \gothq^{-1}, \lambda' , \alpha', \underline \scrF'),
}
\]
where $\alpha'$ and $\underline \scrF'$ are naturally induced level structure and filtrations, and 
\[
\lambda': A' \otimes_{\calO_F} \gothq^{-1} \gothc \xrightarrow{1 \otimes  \vartheta_{\gothc^2}} A \otimes_{\calO_F} \gothc \gothq \xrightarrow{\lambda} A^\vee \otimes_{\calO_F} \gothq \cong (A \otimes_{\calO_F}\gothq^{-1})^\vee \]
is the induced polarization. 
This map extends to the cusps, by sending the $K^p$-cusp with label $\calC = (\gotha, \gothb, L, i, j, \lambda, \alpha)$ to the $K^p$-cusp with label $\calC'' = (\gotha \gothq, \vartheta_{\gothc^2} \gothb \gothq^{-1}, L \otimes \gothq^{-1}, i \otimes \gothq^{-1}, j \otimes \gothq^{-1} \vartheta_{\gothc^2},  \vartheta_{\gothc^2}\lambda, \alpha')$, where $\alpha'$ is the obvious induced level structure.
Since $X_{[\gamma_{\calC''}]} = \vartheta_{\gothc^2} X_{[\gamma_\calC]}$, we may take the cone decomposition $\Phi_{[\gamma_{\calC''}]}$ of $X_{[\gamma_{\calC''}]}^{*+}$ to be the one induced by that at $\calC$, namely $\vartheta_{\gothc^2}^{-1} \Phi_{[\gamma_{\calC}]}$. This way,  the isomorphism $\dot S_\gothq$ extends to an isomorphism $\calM_\gothc^{\tor} \xrightarrow{\sim} \calM_{\gothc'}^{ \tor}$.
Moreover, the moduli description (together with the induced isogeny on the universal semi-abelian varieties) induces a natural pullback morphism
\[\dot 
S_\gothq:  \dot
S_\gothq^*\dot \omega'^\tor \to \dot \omega'^\tor,
\]
where the prime indicates the similar sheaf of relative differentials on $\calM_{\gothc'}^{ \tor}$.
Similarly, we have a natural morphism $\dot S_\gothq: S_\gothq^*\dot \varepsilon'^\tor_\tau \to \dot \varepsilon_\tau^\tor$.
From this, for a paritious weight $\kappa = ((k_\tau)_{\tau \in \Sigma}, w)$,  we define a natural (properly normalized) isomorphism 
\begin{equation}
\label{E:Sgothq at p}
\dot 
S_\gothq^\rmn: \dot S_\gothq^*\dot \omega'^\kappa \cong \dot S_\gothq^*(\dot \omega'^{\kappa - ({\mathbf 2}, 0)}) \otimes S_\gothq^* \big(\wedge^g \Omega_{\calM_{\gothc'}} \big) \xrightarrow{\NN(\gothq)^{-w} \cdot S_\gothq^* \otimes 1} \dot \omega^{\kappa - ({\mathbf 2}, 0)} \otimes \big( \wedge^g \Omega_{\calM_{\gothc}}\big) \cong \dot \omega^\kappa.
\end{equation}
(Note that when $\gothq$ is a $p$-adic place, it is essential to normalize the map $\dot S_{\gothq}$ so that it is an isomorphism.)
Taking quotient of this map by the action of $\calO_F^{\times,  +} / (K \cap \calO_F^\times)^2$ and taking the union over all $\gothc \in \gothC$, we obtain isomorphisms
 \[
S_\gothq : \Sh^{ \tor} \xrightarrow{\sim} \Sh^{\PR, \tor}
\quad \textrm{and} \quad S_\gothq^\rmn: H^*(\Sh_R^{\tor}, \omega'^\kappa_R(-\ttD)) \to H^*(\Sh_R^{\tor}, \omega^\kappa_R(-\ttD))
\]
for any $\calO$-algebra $R$.
We can define similar action of $S_\gothq^\rmn$ on $H^*(\Sh_R^{ \tor}, \omega^\kappa_R)$.
We point out that, although the map $\dot S_\gothq^\rmn$ depend on the choice of the element $\vartheta_\gothc$ (up to multiplication by $\calO_F^{\times, +}$), but the induced action $S_\gothq^\rmn$ is independent of this choice.

Finally, we make explicit this map $\dot S_\gothq^\rmn$ on the $q$-expansions.
To better present the theory, let $\gothc' \in \gothC$ be representative of the class of $\gothc \gothq$ in the strict ideal class, so that $ \vartheta_\gothc \gothc' = \gothc \gothq$ for a positive element $\vartheta_\gothc \in F^\times$.  Consider the two cusps $\calC$ and $\calC''$ above, then we have (for $\xi \in X_{[\gamma_\calC]}^+$)
\begin{align}
\label{E:q-expansion Sqn}
a_\xi(S_\gothq^\rmn(f), \calC', \Tate_{\gotha, \gothb}) &  =  \NN(\gothq)^{-w} \cdot \prod_{\tau } \vartheta_{\gothc^2}^{\frac{w-k_\tau}2}\cdot  a_{\vartheta_{\gothc^2}\xi} (f,\calC'', \Tate_{\gotha \gothq, \vartheta_{\gothc^2}\gothb \gothq^{-1}})
\\
\nonumber & = \NN(\gothq)^{-w}\prod_{\tau}\big(\vartheta_{\gothc^2}^{\frac{w-k_\tau}2} \vartheta_\gothc^{k_\tau}\big) \cdot a_{\vartheta_\gothc^{-2}\vartheta_{\gothc^2}\xi} (f,\calC''_{\vartheta_\gothc^{-1}}, \Tate_{\vartheta_\gothc^{-1}\gotha \gothq, \vartheta_{\gothc}^{-1}\vartheta_{\gothc^2}\gothb \gothq^{-1}}).
\end{align}

Now, we call the polynomial ring
\[
\TT_{\calS}^{\univ}:=\calO[t_\gothq, s_\gothq^\rmn; \gothq \textrm{ a place of }F \textrm{ not in }\calS]
\]
the \emph{universal tame Hecke algebra}. For an $\calO$-algebra $R$, it acts on the cohomology groups 
$H^{j}(\Sh^{\tor}_R,\omega^{\kappa}_R)$ and $H^{j}(\Sh^{\tor}_R,\omega^{\kappa}_R(-\ttD))$ via the assignment $t_\mathfrak{q}\mapsto T_\mathfrak{q}$ and $s_\gothq^\rmn \to S_\gothq^\rmn$.

\begin{lemma}
\label{L:hecke operator duality}
Let $\kappa = ((k_\tau)_{\tau \in \Sigma}, w)$
be a paritious weight and let $\gothq \notin \calS$ be a place of $F$. Since  $\omega^{({\bf 0}, 2)}$ is a canonically trivial vector bundle on $\Sh^\tor$ (as explained in \S\ref{S:toroidal}), we have the following compatibility of Hecke operator with such trivialization and Serre duality.
\begin{enumerate}
\item For any integer $w_0$, we have the following commutative diagram
\begin{equation}
\label{E:adding w}
\xymatrix@C=32pt{
H^*(\Sh_m^\tor, \omega_m^{((k_\tau), w)}) \ar[d]_\cong  \ar[r]^-{\NN(\gothq)^{w_0} \cdot  T_\gothq}&
H^*(\Sh_m^\tor, \omega_m^{((k_\tau), w)}) \ar[d]_\cong \ar[r]^-{S_\gothq^\rmn}&
H^*(\Sh_m^\tor, \omega_m^{((k_\tau), w)}) \ar[d]_\cong 
\\
H^*(\Sh_m^\tor, \omega_m^{((k_\tau), w+2w_0)})  \ar[r]^-{T_\gothq}&
H^*(\Sh_m^\tor, \omega_m^{((k_\tau), w+2w_0)}) \ar[r]^-{S_\gothq^\rmn}&
H^*(\Sh_m^\tor, \omega_m^{((k_\tau), w+2w_0)}).
}
\end{equation}
\item
We define $\kappa^\vee: = ((2-k_\tau)_{\tau \in \Sigma}, w)$ to be the \emph{dual weight} of $\kappa$. (Note that we do not change the normalizing weight $w$.)
The triviality of $\omega^{({\bf 0}, 2)}$ in \S \ref{S:toroidal} identifies the Serre dual of $\omega^\kappa$ with
\begin{equation}
\label{E:serre dual of sheaves}
\xymatrix{
\bigwedge^g \Omega^1_{\Sh^{ \tor}} \otimes \big( \omega^{\kappa} \big)^{-1} \ar[r]^-{\eqref{E:KS isomorphisms}}_-{\cong} & \omega^{({\bf 2}, 0)}(-\ttD) \otimes \omega^{-\kappa}\cong \omega^{\kappa^\vee}(-\ttD)
}
\end{equation}
the dual line bundle.
Then the Serre duality 
\begin{equation}\label{E:Serre duality}
H^i(\Sh^\tor, \omega^\kappa_m) \cong H^{g-i}\big(\Sh^\tor, \omega^{\kappa^\vee}(-\ttD) \otimes \varpi^{-m}\calO/ \calO\big)^\vee
\end{equation}
intertwines the action of $T_\gothq$ on the first term and the transpose $T_\gothq^\vee$ on the second term.
\end{enumerate}
\end{lemma}
\begin{proof}
(1) The commutativity of the left square follows from the fact that, for the Hecke correspondence $\Sh^\tor \xleftarrow{\pi_1} \Sh(\gothq)^\tor \xrightarrow{\pi_2} \Sh^\tor$ (cf. \cite[2.14]{reduzzi-xiao}) the maps $\pi_2^*\omega'^{((k_\tau), w)} \to \pi_1^*\omega^{((k_\tau), w)}$ and $\pi_2^*\omega'^{((k_\tau), w+2w_0)} \to \pi_1^*\omega^{((k_\tau), w+2w_0)}$ are differed by the factor $\NN(\gothq)^{w_0}$.  The commutativity of the right square follows from the fact that $\dot S_\gothq^\rmn: \dot S_\gothq^*\dot \omega'^{((k_\tau), w)} \to \dot \omega^{((k_\tau), w)}$ is independent of $w$ (as we normalized in \eqref{E:Sgothq at p}).

(2)
Following the proof of \cite[Proposition~7.3]{edixhoven}, the Lemma is a consequence of the commutative diagram
\begin{small}
\[
\xymatrix{
H^i(\Sh_m^\tor, \omega_m^\kappa) \ar[d]^{\pi_2^*} \ar[r]^-\SD  & H^{g-i}\big(\Sh_m^\tor, \wedge^g\Omega^1_{\Sh_m^{\tor}} \otimes (\omega_m^\kappa)^{-1}\big)^\vee \ar[r]^-{\eqref{E:serre dual of sheaves}}\ar[d]^{(\pi_{2,*})^\vee} & H^{g-i}(\Sh(\gothq)_m^\tor, \omega_m^{\kappa^\vee}(-\ttD))^\vee
\ar[d]^{(\pi_{2,*})^\vee}
\\
H^i(\Sh(\gothq)_m^\tor,\pi_2^* \omega_m^\kappa) \ar[d]^{T_\gothq^*} \ar[r]^-\SD  & H^{g-i}\big(\Sh(\gothq)_m^\tor, \wedge^g\Omega^1_{\Sh(\gothq)_m^{\tor}} \otimes (\pi_2^*\omega_m^\kappa)^{-1}\big)^\vee\   \ar[r]^-{\pi_2^*\eqref{E:serre dual of sheaves}}\ar[d]^{\NN(\gothq)^w \cdot (T_\gothq^*)^\vee} & H^{g-i}(\Sh(\gothq)_m^\tor, \pi_2^* \omega_m^{\kappa^\vee}(-\ttD))^\vee
\ar[d]^{(T_\gothq^*)^\vee}
\\
H^i(\Sh(\gothq)_m^\tor,\pi_1^* \omega_m^\kappa) \ar[d]^{\pi_{1,*}} \ar[r]^-\SD  & H^{g-i}\big(\Sh(\gothq)_m^\tor, \wedge^g\Omega^1_{\Sh(\gothq)_m^{\tor}} \otimes (\pi_1^*\omega_m^\kappa)^{-1}\big)^\vee\ \ar[r]^-{\pi_1^*\eqref{E:serre dual of sheaves}}\ar[d]^{(\pi_1^*)^\vee} & H^{g-i}(\Sh(\gothq)_m^\tor, \pi_1^*\omega_m^{\kappa^\vee}(-\ttD))^\vee
\ar[d]^{(\pi_1^*)^\vee}
\\
H^i(\Sh_m^\tor, \omega_m^\kappa) \ar[r]^-\SD  & H^{g-i}\big(\Sh_m^\tor, \wedge^g\Omega^1_{\Sh_m^{ \tor}} \otimes (\omega_m^\kappa)^{-1}\big)^\vee \ar[r]^-{\eqref{E:serre dual of sheaves}} & H^{g-i}(\Sh_m^\tor, \omega^{\kappa^\vee}_m(-\ttD))^\vee.
}
\]\end{small}Here $\mathrm{SD}$ stands for Serre duality \eqref{E:Serre duality}. The Hecke correspondence $\Sh^\tor \xleftarrow{\pi_1} \Sh(\gothq)^\tor \xrightarrow{\pi_2} \Sh^\tor$ is as defined in \cite[\S2.14]{reduzzi-xiao}.
\end{proof}

\subsection{Generalized partial Hasse invariants on splitting models}
\label{S:partial Hasse invariants}
We recall some constructions from \cite[\S 3]{reduzzi-xiao}. 
For each $\tau = \tau_{\gothp_i,j}^l$ with $l \neq 1$, multiplication by $\varpi_i$ induces a well-defined morphism:
\[
\xymatrix@R=0pt{
m_{\varpi_i,j}^{(l)}: \calF_{\gothp_i,j}^{(l)} / \calF_{\gothp_i,j}^{(l-1)}
\ar[r] &
\calF_{\gothp_i,j}^{(l-1)} / \calF_{\gothp_i,j}^{(l-2)}\\
 z\ar@{|->}[r] & [\varpi_i](\tilde z),
}
\]
where $\tilde z$ is a lift to $\calF_{\gothp_i,j}^{(l)}$ of the local section $z$. Hence $m_{\varpi_i,j}^{(l)}$ induces a section
\[
\dot h_\tau \in H^0(\calM_\FF, \dot \omega_{\tau_{\gothp_i, j}^l}^{\otimes-1} \otimes \dot \omega_{\tau_{\gothp_i, j}^{l-1}})
\]
that is invariant under the action of $\calO_F^{\times,+} / (K \cap \calO_F^\times)^2$, and therefore descend to a section $h_\tau \in H^0(\Sh_\FF, \omega^{\otimes-1}_{\tau_{\gothp_i, j}^l} \otimes \omega_{\tau_{\gothp_i, j}^{l-1}})$ (\cite[Construction~3.3]{reduzzi-xiao}).

For $\tau = \tau_{\gothp_i,j}^1$, a morphism
\[
\mathrm{Hasse}_{\varpi_i,j}: \dot \omega_{\tau_{\gothp_i,j}^1} = \calF_{\gothp_i, j}^{(1)} \longto
\omega^{(p)}_{\calA_\FF/\calM_\FF,\gothp_i,j-1} /  \big( \calF_{\gothp_i,j-1}^{(e_i- 1)}\big)^{(p)} \cong  \dot \omega_{\tau_{\gothp_i,j -1}^{e_i}}^{\otimes p}
\]
is constructed in \cite[Construction~3.6]{reduzzi-xiao} as follows: let $z$ be a local section of $\dot  \omega_{\tau_{\gothp_i,j}^1}$; since it is annihilated by $[\varpi_i]$ acting on $\calH^1_\dR(\calA_\FF/ \calM_\FF)_{\gothp_i,j}$, $z$ belongs to $[\varpi_i]^{e_i-1} \cdot \calH^1_\dR(\calA_\FF/ \calM_\FF)_{ \gothp_i,j}$.
Write $z = [\varpi_i]^{e_i-1} z'$ for a local section $z'$ of $ \calH^1_\dR(\calA_\FF/ \calM_\FF)_{\gothp_i,j}$ 
(note this $z'$ belongs to $\calH^1_\dR$ but it might not belong to $\omega$). 
We define $\mathrm{Hasse}_{\varpi_i,j}(z)$ to be the image of $V_{\gothp_i,j}(z')$ in 
$\omega^{(p)}_{\calA_\FF/\calM_\FF,\gothp_i,j-1} /  \big( \calF_{\gothp_i,j-1}^{(e_i- 1)}\big)^{(p)}$.
The homomorphism $\mathrm{Hasse}_{\varpi_i,j}$ is well defined and it induces a section
\[
\dot h_{\tau_{\gothp_i,j}^1} \in H^0(\calM_\FF, \dot \omega_{\tau_{\gothp_i,j}^1}^{\otimes-1} \otimes 
\dot \omega_{\tau_{\gothp_i, j-1}^{e_i}}^{\otimes p})
\]
which descends to a section $h_{\tau_{\gothp_i,j}^1} \in H^0(\Sh_\FF, \omega^{\otimes-1}_{\tau_{\gothp_i,j}^1} \otimes \omega^{\otimes p}_{\tau_{\gothp_i, j-1}^{e_i}})$ (\cite[Construction~3.6]{reduzzi-xiao}). Note also that the action of $\calO_F^{\times, +}$ on $\dot \omega^{-1}_{\tau^1_{\gothp_i, j}} \otimes \dot \omega^{\otimes p}_{\tau^{e_i}_{\gothp_i, j-1}}$ is just the simple pullback (without the twist by powers of $\tau(u)$ for $u \in \calO_F^{\times, +}$). 

We remark that the total Hasse invariant $h$ is the product of some powers of $h_\tau$ for $\tau \in \Sigma$. For a more explicit expression, see \cite[Lemma~3.8]{reduzzi-xiao}.

\subsection{Goren--Oort stratification of $\Sh_\FF$}
\label{S:Goren-Oort stratification}
For each $p$-adic embedding $\tau \in \Sigma$, denote by $X_\tau$ the zero locus of the generalized Hasse invariant $h_\tau$ on $\Sh_\FF$.
In general, for a subset $\ttT\subseteq\Sigma$, we set $X_\ttT := \cap _{\tau \in \ttT} X_\tau$, with the convention
that if $\ttT$ is the empty set, this intersection is interpreted to be the entire space $\Sh_\FF$.
We define $\dot X_\tau$ and $\dot X_{\ttT}$ on $\calM_\FF$ similarly.
These $X_\ttT$ (resp. $\dot X_\ttT$) form the \emph{Goren--Oort stratification} of $\Sh_\FF$ (resp. $\calM_\FF$).
We write
\[
X_\ttT^\circ : = X_\ttT \backslash \cup_{\ttT' \supsetneq \ttT} X_{\ttT'} \quad \textrm{and} \quad \dot X_\ttT^\circ : = \dot X_\ttT \backslash \cup_{\ttT' \supsetneq \ttT} \dot X_{\ttT'}
\]
for the corresponding open stratum.

We remark that, although the generalized partial Hasse invariants $h_\tau$ and $\dot h_\tau$ depend on the choice of uniformizers $\varpi_i$, their zero loci $X_\tau$ and $\dot X_\tau$ do not. Moreover, each section $\dot h_\tau$ extends to the fixed toroidal compactification
$\calM_{\FF}^{\mathrm{tor}}$ of $\calM_{\FF}$ and, by \cite[Theorem 3.10]{reduzzi-xiao}, each subscheme $\dot X_{\tau}$ is disjoint from the cusps, which are ordinary points of
the moduli space.

\begin{proposition}
\label{T:smoothness of ramified GO strata}
The following properties hold.
\begin{enumerate}
 \item The closed subschemes $X_\tau$ (resp. $\dot X_\tau$) are proper and smooth divisors with simple normal crossings on $\Sh_\FF$ (resp. $\calM_\FF$).
In particular, $ X_\ttT$ and $\dot X_\ttT$ for any proper subset $\ttT \subset \Sigma$ are proper and smooth varieties over $\FF$.

\item The union of the closed subschemes $X_\tau$ (resp. $\dot X_\tau$) when $\tau$ varies over all embeddings of the form $\tau_{\gothp_i,j}^l$ with $l\neq 1$ coincides with the complement of $\Sh^{\Ra}_\FF$ (resp. $\calM^\Ra_\FF$) in $\Sh_\FF$ (resp. $\calM_\FF$).
\item The ordinary locus $\Sh^{\mathrm{ord}}_\FF$ (resp. $\calM^{\mathrm{ord}}_\FF$) of $\Sh_\FF$ (resp. $\calM_\FF$) coincides with the complement of the set 
$\cup_{\tau\in\Sigma}X_\tau$ (resp. $\cup_{\tau\in\Sigma}\dot X_\tau$).
\end{enumerate}
\end{proposition}
\begin{proof}
It is enough to prove the proposition over $\calM_\FF$. The first statement is contained in \cite[Theorem 3.10]{reduzzi-xiao}. To prove the second statement we can work with
closed points of $\calM_{\FF}$, since the Rapoport condition on
the Lie algebra is an open condition. Let then $(A,\lambda
,\alpha^p,\mathscr{\underline{F}})$ be a $k$-point of $\calM_{\FF}$, for
$k$ an extension field of $\FF$. The abelian variety $A$ with RM satisfies the Rapoport condition
if and only if multiplication by $\varpi_{i}$ induces an isomorphism
${\mathscr{F}_{\mathfrak{p}_{i},j}^{(l)}/\mathscr{F}_{\mathfrak{p}_{i}%
,j}^{(l-1)}\rightarrow\mathscr{F}_{\mathfrak{p}_{i},j}^{(l-1)}/\mathscr{F}%
_{\mathfrak{p}_{i},j}^{(l-2)}}$ for all $i$, all $j$, and all $l\neq 1$. This is clearly equivalent to the given condition on the divisors $\dot X_\tau$.
Statement (3) follows from the fact that the total Hasse invariant is the product of powers of $h_\tau$'s, as remarked at the end of \S\ref{S:partial Hasse invariants}.
\end{proof}

\section{Hecke operators at $p$ in characteristic $p^m$}
\label{Sec:section 3}
In this section, we construct the $T_\gothp$-operator acting on the cohomology of some automorphic line bundles with torsion coefficients.
We work exclusively with the Pappas--Rapoport splitting model, as indicated in Notation~\ref{N:simplify notation}. The general construction of Iwahori level structure on the splitting model is already explained in \cite{pappas-rapoport} (in a much more general setup). We here make it more explicit in our particular setup. Our construction of $T_\gothp$ is inspired by the work of B. Conrad \cite{conrad}. The essential main technical input is Proposition~\ref{P:Tp factor} whose proof we complete in the next section.
 
We fix a prime ideal $\mathfrak{p}\in \{\mathfrak{p}_1,\dots,\mathfrak{p}_r\}$; denote by $f$ its residual degree, and by $e$ its inertial degree. We write $\varpi_\gothp$ for the chosen uniformizer at $\gothp$ as in \S\ref{S:setup}. For every fractional ideal $\gothc\in\gothC$ we choose once and for all a positive isomorphism $\theta_\gothc:\gothc\gothp\simeq\gothc'$ of $\gothc\gothp$ with a (uniquely determined) fractional ideal $\gothc'\in\gothC$.
 
\subsection{Splitting models with Iwahori level structure}
\label{S:setup section 3}

For a fixed $\gothc\in\mathfrak{C}$ we define the corresponding splitting models with Iwahori level structure following \cite{pappas, pappas-rapoport}.

Let $\underline{\calM}_{\gothc, \theta_\gothc}(\mathfrak{p})$ denote the functor which assigns to a
locally noetherian $\mathcal{O}$-scheme $S$ the set of isomorphism classes of tuples 
$((A,\lambda,\alpha,\underline\scrF);(A',\lambda',\alpha',\underline\scrF');\phi, \psi)$ where: 
\begin{enumerate}
\item $[(A,\lambda,\alpha,\underline\scrF)]$ is an $S$-point of $\mathcal{M}_\gothc$,
\item $[(A',\lambda',\alpha',\underline\scrF')]$ is an $S$-point of $\mathcal{M}_{\gothc'}$,
\item $\phi:A\rightarrow A'$ and $\psi: A' \to A \otimes \gothc(\gothc')^{-1}$ are $\calO_F$-equivariant $S$-isogenies satisfying:
\begin{enumerate}
\item both $\phi$ and $\psi$ have degree $p^f$,
\item the compositions $\psi\circ\phi$  and $(\phi\otimes \gothc (\gothc')^{-1})\circ\psi$ are the natural isogenies $A \to A \otimes \gothc (\gothc')^{-1}$ and $A' \to A' \otimes \gothc (\gothc')^{-1}$ induced by $\calO_F \subseteq \gothp^{-1} \stackrel{ \theta_\gothc}\simeq \gothc (\gothc')^{-1} $,
 \item $\phi$ is compatible with the polarizations, \emph{i.e.}, $\phi\circ\lambda\circ\phi^{\vee}=\tilde\lambda'$, where $\tilde\lambda'$ is the map 
	$(A')^\vee \rightarrow A'\otimes\gothc$ induced by composing $\lambda'$ with the map $\gothc'\overset{\theta^{-1}_\gothc}\simeq\gothc\gothp\subset\gothc$,
 \item both $\phi$ and $\psi$ are compatible with the level structures, \emph{i.e.}, $\phi \circ \alpha = \alpha'$ and $\psi \circ \alpha' = \alpha \otimes \gothc(\gothc')^{-1}$,
 and
 \item $\phi$ and $\psi$ are compatible with the filtrations, \emph{i.e.}, for any indices $i\in\{1,\dots,r\}$ and $j\in\{1,\dots,f_i\}$ the morphisms of locally free $\calO_S$-modules 
	\[
\phi^{\ast}:\omega_{A'/S,\gothp_i,j} \rightarrow\omega_{A/S,\gothp_i,j} \textrm{   and   } {\psi}^{\ast}:\omega_{A/S,\gothp_i,j} \cong \omega_{A/S,\gothp_i,j}\otimes \gothc (\gothc')^{-1} \rightarrow\omega_{A'/S,\gothp_i,j}
	\]
 preserve the filtrations $\scrF_{\gothp_i,j}^{\bullet}$ and $\scrF_{\gothp_i,j}'^{\bullet}$.\\
\end{enumerate}
\end{enumerate}


We point out that $\phi$ and $\psi$ determine each other, but we need to keep both isogenies in the data to state the condition that $\phi^*$ and $\psi^*$ preserve the filtrations $\scrF_{\gothp_i, j}^\bullet$ and $\scrF'^\bullet_{\gothp_i, j}$. In particular, for $\gothp_i \neq \gothp$, 
\[
\scrF'^{(l)}_{\gothp_i,j} \xrightarrow{\
 \phi^*\ } \scrF^{(l)}_{\gothp_i,j}\xrightarrow{\
  \psi^*\ }
\scrF'^{(l)}_{\gothp_i,j}
\]
are isomorphisms for all $j = 1, \dots, f_i$ and $l = 1, \dots, e_i$.
Moreover, the construction also gives a commutative diagram:
\begin{equation}
\label{E:lambda phi psi commutative}
\xymatrix{
A^\vee \ar[r]^-\lambda & A \otimes \gothc\\
A'^\vee \ar[u]^{\phi^\vee} \ar[r]^-{\lambda'} & A' \otimes \gothc' \ar[u]^{\psi}.
}
\end{equation}

Without the additional filtrations, the moduli problem $\underline \calM_{\gothc, \theta_\gothc}(\gothp)$ is the same as the one in \cite[Definition 2.2.1]{pappas} and is hence representable by an $\calO$-scheme of finite type by \cite[Theorem 2.2.2]{pappas}.  Introducing the additional filtrations amounts to building a further Grassmannian bundle and requiring compatibilities with $\phi$ and $\psi$ is a closed condition. So $\underline \calM_{\gothc, \theta_\gothc}(\gothp)$ is represented by an $\calO$-scheme $\calM_{\gothc, \theta_\gothc}(\gothp)$ of finite type.

There are two natural forgetful maps $\pi_{1, \theta_\gothc}: \calM_{\gothc, \theta_\gothc} (\gothp) \to \calM_\gothc$ and $\pi_{2, \theta_\gothc}: \calM_{\gothc, \theta_\gothc} (\gothp) \to \calM_{\gothc'}$ induced by only keeping $(A, \lambda, \alpha, \underline \scrF)$ and $(A', \lambda', \alpha', \underline \scrF')$, respectively.
A different choice of positive isomorphism $\theta'_\gothc: \gothc\gothp \simeq \gothc'$ is of the form $u \theta_\gothc$ for some $u \in \calO_F^{\times,+}$.
Then there is a natural isomorphism $\Theta: \calM_{\gothc, \theta_\gothc}(\gothp) \xrightarrow{\sim} \calM_{\gothc, \theta'_\gothc}(\gothp)$ given by
\[
((A,\lambda,\alpha,\underline\scrF);(A',\lambda',\alpha',\underline\scrF');\phi, \psi) \longmapsto ((A,\lambda,\alpha,\underline\scrF);(A',u^{-1}\lambda',\alpha',\underline\scrF');\phi, \psi).
\]
It then follows that
\begin{equation}
\label{E:equivariance for Iwahori level}
\pi_{1, \theta_\gothc} = \pi_{1, \theta'_\gothc}\circ \Theta  \quad \textrm{and} \quad  \pi_{2, \theta_\gothc} =u \cdot \pi_{2, \theta'_\gothc}\circ \Theta
\end{equation}

Now the group $\calO_F^{\times,+}/(K \cap \calO_{F}^{\times})^2$ acts freely on $\calM_{\gothc, \theta_\gothc}(\mathfrak{p})$ (by acting simultaneously on $A$ and $A'$). We denote by $\Sh_\gothc(\mathfrak{p})$ the corresponding quotient. We set
$\calM(\gothp):=\coprod_{\gothc\in\gothC}\calM_{\gothc, \theta_\gothc}(\gothp)$ and $\Sh(\gothp):=\coprod_{\gothc\in\gothC}\Sh_\gothc(\gothp)$. 
Both $\pi_{1, \theta_\gothc}$ and $\pi_{2, \theta_\gothc}$ are equivariant for the action of $\calO_F^{\times,+}/(K\cap \calO_{F}^{\times})^2$ and induce maps:
\[
\pi_1 = (\pi_{1, \theta_\gothc}):\Sh(\mathfrak{p})\rightarrow\Sh \quad \textrm{and} \quad \pi_2 = (\pi_{2, \theta_\gothc}): \Sh(\mathfrak{p})\rightarrow\Sh.
\]
By \eqref{E:equivariance for Iwahori level}, the natural maps $\pi_1$ and $\pi_2 $ are independent of the auxiliary choice of the isomorphisms $\theta_\gothc$.

\begin{remark}
\label{R:alternative Gamma 0p}
When $p$ is \emph{unramified} in $F$ (or if we only consider the ordinary locus), the moduli space $\calM_\gothc(\gothp)$ can also be described as the moduli space of tuples $(A,\lambda,\alpha;C)/S$ where
\begin{itemize}
 \item $[(A,\lambda,\alpha)]$ is an $S$-point of $\calM_\gothc$ (the filtration datum on $\omega_{A/S}$ is uniquely determined in this case);
 \item $C$ is a finite flat, closed, $\calO_F$-stable $S$-subgroup scheme of $A[\gothp]$ of rank $p^f$ (which is necessarily isotropic with respect to the Weil paring induced by $\lambda$).
\end{itemize}
It is not clear whether such interpretation could be extended to the splitting model. For this reason, we defined $\calM_\gothc(\gothp)$ as a moduli space of isogenies.
\end{remark}

\begin{proposition}
\label{P:LCI}
The scheme $\calM_\gothc(\gothp)$ (resp. $\Sh_\gothc(\mathfrak{p})$) is a flat local complete
intersection of relative dimension $g$ over $\calO$. In particular, it
is Gorenstein, and hence Cohen--Macaulay. Moreover, the special fiber 
$\calM_{\gothc}(\gothp)_\FF$ (resp. $\Sh_\gothc(\mathfrak{p})_\FF$) is smooth outside a closed subscheme
of codimension 1, and $\calM_\gothc(\gothp)$ (resp. $\Sh_\gothc(\mathfrak{p})$) is normal.
\end{proposition}

\begin{proof}
It is enough to prove the statements for
$\calM_\gothc(\mathfrak{p})$, since the natural quotient map
$\calM_\gothc(\mathfrak{p})\rightarrow\Sh_\gothc(\mathfrak{p})$ is
finite and \'etale. The proof is then reduced to an argument in terms of local models similar to \cite[Th\'eor\`em 3.3]{deligne-pappas} and \cite[Theorem~3.3.1]{pappas} (cf. also \cite{pappas-rapoport} and \cite{sasaki}). To benefit the readers, we provide with more details.

The key is to show that for any closed point $x_0$ of $\calM_\gothc(\gothp)$ with finite residue field, there exists a (sufficiently small) open neighborhood $\calU \subseteq \calM_\gothc(\gothp)$ of $x_0$ that admits a map $\calU \to \calN$ \'etale at $x_0$, where $\calN$ is a fixed moduli problem we explain below, a.k.a. the local model. For an $\calO$-scheme $S_0$, we consider two maps 
\begin{equation}
\label{E:constant tuple}
u_{\mathfrak{p}
_{i},j}^{(l)}, v_{\gothp_i,j}^{(l)}: \calO_{S_0}^2 \to \calO_{S_0}^2
\end{equation}
given by explicit matrices
\[
u_{\mathfrak{p}_{i},j}^{(l)}:=\left\{
\begin{array}
[c]{cc}%
{
\begin{pmatrix}
1 & 0\\
0 & \tau_{\mathfrak{p}_{i},j}^{l}(\varpi_{i})
\end{pmatrix}} & \text{if }\gothp_i=\gothp\\
{\begin{pmatrix}
1 & 0\\
0 & 1
\end{pmatrix}}  & \text{if }\gothp_i\neq\gothp
\end{array}
\right.
\quad \textrm{and}
\quad
v_{\mathfrak{p}_{i},j}^{(l)}:=\left\{
\begin{array}
[c]{cc}%
{
\begin{pmatrix}
\tau_{\mathfrak{p}_{i},j}^{l}(\varpi_{i}) & 0\\
0 & 1
\end{pmatrix}}  & \text{if }\gothp_i=\gothp\\
{\begin{pmatrix}
1 & 0\\
0 & 1
\end{pmatrix}}   & \text{if }\gothp_i\neq\gothp.
\end{array}
\right.
\]
Our local model $\calN$ is taken to be the product of $\calO$-schemes $\prod_{i,j,l}\calN_{i,j}^{(l)}$, where each $\calN_{i,j}^{(l)}$ represents the functor on the category of locally noetherian $\mathcal{O}$-schemes sending a scheme $S_0$ to the
set of isomorphism classes of pairs $(\scrF,\scrF^{\prime})$ of invertible $\mathcal{O}_{S_0}$-subbundles of $\mathcal{O}%
_{S_0}^{2}$ satisfying $u_{\mathfrak{p}_{i},j}^{(l)}(\scrF^{\prime})\subseteq \scrF $ and $v_{\mathfrak{p}_{i},j}^{(l)}(\scrF)\subseteq \scrF^{\prime} $.
Each $G_{i,j}^{(l)}$ is clearly a closed subscheme of the product of two Grassmannians.
Explicitly, if $\gothp_i \neq \gothp$, $\calN_{i,j}^{(l)} \cong \PP^1_\calO$ (as $\scrF$ and $\scrF'$ determine each other). If $\gothp_i =\gothp$, one can show that (see e.g. the case $N=1$ of \cite[Proposition~4.2.2]{pappas}), over a small enough open subspace $\calU \subseteq \gothN_{i,j}^{(l)}$ where both $\scrF$ and $\scrF'$ are trivialized, there is an \'etale map 
\[
h: \calU \to \calO[U_{\gothp,j}^{(l)},V_{\gothp,j}^{(l)}] \big/ \big(U_{\gothp,j}^{(l)} V_{\gothp,j}^{(l)}-\tau_{\mathfrak{p},j}^{l}(\varpi_i)\big) 
\]
such that the maps 
$
u_{\gothp,j}^{(l)}: \scrF' \to \scrF$ and $ v_{\gothp,j}^{(l)}: \scrF \to \scrF '$
are given by multiplication by  $U_{\gothp,j}^{(l)}$ and $V_{\gothp,j}^{(l)}$, respectively.
So to sum up, $\gothN$ and hence $\mathcal{M}_{\gothc}(\mathfrak{p})$ (if we had proved that $\calN$ is a local model) are \'{e}tale locally isomorphic
to the spectrum of the following ring:
\[
\mathcal{O}[(W_{\gothp_i,j}^{(l)})_{\gothp_i\neq\gothp,j=1,\ldots,f_{i},l=1,\ldots,e_{i}
}]\otimes_{\mathcal{O}}
\left(
{\bigotimes\limits_{j=1}^{f}}
{\bigotimes\limits_{l=1}^{e}}
\frac{\mathcal{O}[U_{\gothp,j}^{(l)},V_{\gothp,j}^{(l)}]}{(U_{\gothp,j}%
^{(l)}V_{\gothp,j}^{(l)}-\tau_{\mathfrak{p},j}^{l}(\varpi_\gothp))}\right).
\]
This ring is visibly normal and is a flat complete intersection of relative dimension $g$ over $\calO$, which is also smooth outside a closed subscheme of codimension 1. This would then prove the Proposition.

Now, it remains to show that for each closed point $x_0$ of $\calM_\gothc(\gothp)$ with finite residue field,
\begin{enumerate}
\item[(i)] there is an open neighborhood $\calU$ and a map $\varphi: \calU \to \calN$, such that
\item[(ii)]
the map induces an isomorphism $\varphi_{x_0, *}: T_{x_0}\calU \to T_{x_0} \calN$ on the tangent space (so that the map $\varphi_{x_0}:\calU_{x_0}^\wedge \to \calN_{\varphi(x_0)}^\wedge$ on the completion is a closed embedding), and
\item[(iii)]
(by deformation theory) that the map $\varphi_{x_0}$ is an isomorphism.
\end{enumerate}
For this, we follow the general method in \cite[\S3.3]{deligne-pappas}, but we need to study the deformation functor for points on $\calM_\gothc(\gothp)$ following \cite[Theorem~2.9]{reduzzi-xiao}. 
We first quickly recall Grothendieck--Messing deformation theory of abelian varieties.  Let $S_0 \hookrightarrow S$ be a closed embedding of locally noetherian $\calO$-schemes whose ideal sheaf of definition $\calI$ satisfies $\calI^2 = 0$.
Let $\mathsf{Ab}_S$ denote the category of abelian varieties $A$ over $S$. For an abelian variety $A_0$ over $S_0$, let $\calH^1_\cris(A_0/S_0)_S$ denote the evaluation of the first relative crystalline cohomology.
Let $\mathsf{Ab}_{S_0}^+$ denote the category of abelian varieties $A_0$ over $S_0$ together with a lift $\hat \omega \subseteq \calH^1_\cris(A_0/S_0)_S$ of $\omega_{A_0/S_0} \subseteq \calH^1_\dR(A_0/S_0)\cong\calH^1_\cris(A_0/S_0)_{S_0}$.
The main theory of crystalline deformation theory (cf. \cite[pp.116--118]{grothendieck}, \cite[Chap. II \S1]{mazur-messing}) says that the natural functor 
\[
\xymatrix@R=0pt{
\mathsf{Ab}_S \ar[r] & \mathsf{Ab}_{S_0}^+
\\
A \ar@{|->}[r] & (A \times_{S}S_0, \omega_{A/S})
}
\]
is an equivalence of category.
In other words, to lift an abelian varieties $A_0$ over $S_0$ to $S$, it suffices to lift the corresponding sheaf of differentials. (Extending additional endomorphisms on $A_0$ amounts to requiring the lift of the sheaf of differentials is stable under the action of the endomorphisms.)

Now, let $S_0$ denote a noetherian $\calO$-scheme and let $\calO_{S_0}^\cris$ denote the structure sheaf on the crystalline site. Consider an $S_0$-valued point 
$x_0=((A, \lambda, \alpha, \underline\scrF ),(A', \lambda', \alpha', \underline\scrF' );\phi, \psi)$ of $\calM_\gothc(\gothp)$. Let 
$\calH_\cris^1(A/ S_0)$ denote the crystalline cohomology sheaf of $A$ over $S_0$. The action of $\calO_F$ on $A$ induces a natural direct sum decomposition:
\[
\calH_\cris^1(A/S_0) = \bigoplus_{i =1}^r \bigoplus_{j = 1}^{f_i}
\calH_\cris^1(A/S_0)_{\gothp_i, j},
\]
where $W(\FF_{\gothp_i}) \subseteq \calO_{F_{\gothp_i}}$ acts on $\calH_\cris^1(A/S_0)_{\gothp_i, j}$ via $\tau_{\gothp_i,j}$.
Moreover $\calH_\cris^1(A/S_0)_{\gothp_i, j}$ is a locally free module of rank two over
\[
\calO_{F_{\gothp_i}} \otimes_{W( \FF_{\gothp_i}), \tau_{\gothp_i,j}} \calO^\mathrm{cris}_{S_0} \cong \calO^\mathrm{cris}_{S_0} [x] / (E_{\gothp_i,j}(x)).
\]
The polarization $\lambda: A^\vee \to A \otimes_{\calO_F} \gothc$ induces a non-degenerate, symplectic pairing 
\begin{equation}
\label{E:pairing on Hcris}
\langle\cdot, \cdot \rangle\colon
\calH^1_\cris(A/S_0)_{\gothp_i,j} \times \calH^1_\cris(A/S_0)_{\gothp_i,j} \to \calO_{S_0}^\cris, \ \textrm{such that}
\end{equation}
\begin{equation}
\label{E:OF hermitian}
\langle ax , y \rangle  = \langle x,ay \rangle \quad \textrm{and} \quad \langle ax,x\rangle =0.
\end{equation}
for $a \in \calO_{F_{\gothp_i}}$ and $x,y \in \calH^1_\cris(A_0/S_0)_{\gothp_i, j}$ (as proved in \cite[(2.9.1--2)]{reduzzi-xiao}).

Analogous constructions, notations, and properties can be introduced for the abelian scheme $A'/S_0$.
The isogenies $\phi$ and $\psi$ define morphisms 
\[
\phi^*: \calH^1_\cris(A'/S_0)_{\gothp_i,j} \to \calH^1_\cris(A/S_0)_{\gothp_i,j}
\quad \textrm{and} \quad 
\psi^*: \calH^1_\cris(A/S_0)_{\gothp_i,j} \to \calH^1_\cris(A'/S_0)_{\gothp_i,j}.
\]
The commutative diagram \eqref{E:lambda phi psi commutative} implies that 
\[
\big\langle x, \phi^* (y) \big\rangle = \big\langle \psi^*(x),  y \big\rangle'
\]
for $x \in \calH^1_\cris(A/S_0)_{\gothp_i,j}$ and $y \in \calH^1_\cris(A'/S_0)_{\gothp_i,j}$.

Consider $S$ an infinitesimal deformation of $S_0$, i.e., $S_0 \hookrightarrow S$ is a closed immersion of locally noetherian $\calO$-schemes whose ideal of definition $\calI$ satisfies $\calI^2=0$.
Since $S$ is a PD-thickening of $S_0$, we can evaluate the crystalline cohomology over $S$ to obtain $\calH_\cris^1(A/S_0)_S$ and its direct summands $\calH_\cris^1(A/S_0)_{S,\gothp_i, j}$. 
For a subspace $\scrF$ of $\calH^1_\cris(A_0/S_0)_{S,\gothp_i, j}$, we use $\scrF^\perp$ to denote its orthogonal complement under the pairing \eqref{E:pairing on Hcris}.

The submodule $\omega_{A/S_0,\gothp_i,j}$ of $\calH^1_\cris(A/S_0)_{S_0,\gothp_i,j}$  is (maximal) isotropic for this pairing. In particular, $\scrF_{\gothp_i,j}^{(l)} \subset  (\scrF_{\gothp_i,j}^{(l)})^\perp$ for all $i,j,l$.
Moreover, we have
\[
\phi^*(\scrF'^{(l)}_{\gothp_i,j}) \subseteq \scrF_{\gothp_i,j}^{(l)}, \quad \phi^*\big((\scrF'^{(l)}_{\gothp_i,j})^\perp\big) \subseteq (\scrF_{\gothp_i,j}^{(l)})^\perp,\quad
\psi^*(\scrF_{\gothp_i,j}^{(l)}) \subseteq \scrF'^{(l)}_{\gothp_i,j}, \quad \textrm{and} \quad \psi^*\big((\scrF_{\gothp_i,j}^{(l)})^\perp\big) \subseteq (\scrF'^{(l)}_{\gothp_i,j})^\perp.
\]

Let
\[
\scrH_{\gothp_i,j}^{(l)}(A/S_0): = \Big\{ z \in (\scrF_{\gothp_i,j}^{(l-1)})^\perp \big/  \scrF_{\gothp_i,j}^{(l-1)}\; \Big|\;
[\varpi_i]z - \tau_{\gothp_i, j}^{l}(\varpi_i)z  =0 \Big\},
\]
and similarly define $\scrH_{\gothp_i,j}^{(l)}(A'/S_0)$ for each index $i,j,l\geq1$; so that we have natural morphisms
\[
\phi^*_\scrH: \scrH_{\gothp_i,j}^{(l)}(A'/S_0) \to \scrH_{\gothp_i,j}^{(l)}(A/S_0) \quad \textrm{and} \quad \psi^*_\scrH:
\scrH_{\gothp_i,j}^{(l)}(A/S_0) \to \scrH_{\gothp_i,j}^{(l)}(A'/S_0).
\]
 It is shown in the Claim of the proof of \cite[Theorem 2.9]{reduzzi-xiao} that $\scrH_{\gothp_i,j}^{(l)}(A/S_0)$ is a rank-two $\calO_{S_0}$-subbundle of $\calH^1_\mathrm{cris}(A/S_0)_{S_0,\gothp_i,j} \big/  \scrF_{\gothp_i,j}^{(l-1)}$. In particular, the sheaf $\scrF_{\gothp_i,j}^{(l)}/\scrF_{\gothp_i,j}^{(l-1)}$ is a rank-one $\calO_{S_0}$-subbundle of $\scrH_{\gothp_i,j}^{(l)}(A/S_0)$.

By crystalline deformation theory for abelian schemes, lifting the $S_0$-point $x_0$ associated to the isogeny $\phi:A\rightarrow A'$ to $S$ 
is equivalent to the following procedure, applied to each choice of $i$ and $j$:

\begin{enumerate}
\item 
Write $\scrH_{\mathfrak{p}_{i},j}^{(1)}(A/S_0)_S$ and $\scrH_{\mathfrak{p}_{i},j}^{(1)}(A'/S_0)_S$ for the kernel of the map $[\varpi_i]-\tau_{\gothp_i,j}^1(\varpi_i)$ acting on
$\calH^1_\cris(A/S_0)_{S,\gothp_i,j}$ and $\calH^1_\cris(A'/S_0)_{S,\gothp_i,j}$, respectively.  

Lift $\scrF_{\mathfrak{p}_{i},j}^{(1)}\subset\scrH_{\mathfrak{p}_{i},j}^{(1)}(A/S_0)$ and $\scrF_{\mathfrak{p}_{i},j}^{\prime(1)}\subset\scrH_{\mathfrak{p}_{i},j}^{(1)}(A'/S_0)$ to (isotropic) rank-one $\calO_S$-subbundles  $\mathscr{\tilde{F}}_{\mathfrak{p}_{i},j}^{(1)}\subset \scrH_{\mathfrak{p}_{i},j}^{(1)}(A/S_0)_S$ and
$\mathscr{\tilde{F}}_{\mathfrak{p}_{i},j}^{\prime(1)} \subset\scrH_{\mathfrak{p}_{i},j}^{(1)}(A'/S_0)_S$ satisfying the conditions $\phi^*(\mathscr{\tilde{F}}_{\mathfrak{p}_{i},j}
^{\prime(1)})\subseteq\mathscr{\tilde{F}}_{\mathfrak{p}_{i},j}^{(1)}$ and $\psi^*(\mathscr{\tilde{F}}_{\mathfrak{p}_{i},j}
^{(1)})\subseteq\mathscr{\tilde{F}}_{\mathfrak{p}_{i},j}^{\prime(1)}$. 
Here, the isotropic condition is automatic by condition~\eqref{E:OF hermitian}; see the proof of \cite[Theorem~2.9]{reduzzi-xiao} for more details.
\item Once lifts $\mathscr{\tilde{F}}_{\mathfrak{p}_{i},j}^{(l)}$ and
$\mathscr{\tilde{F}}_{\mathfrak{p}_{i},j}^{\prime(l)}$ are chosen for all $l\leq t-1
$, set
\[
\mathscr{H}_{\mathfrak{p}_{i},j}^{(t)}(A/{S_0})_S:=\Big\{ z \in (\tilde\scrF_{\gothp_i,j}^{(t-1)})^\perp \big/  \tilde\scrF_{\gothp_i,j}^{(t-1)}\; \Big|\;
[\varpi_i]z - \tau_{\gothp_i, j}^{t}(\varpi_i)z  =0 \Big\}
\]
and similarly for $\mathscr{H}_{\mathfrak{p}_{i},j}^{(t)}(A'/S_0)_S$; by the proof of Claim (1) in  \cite[Theorem~2.9]{reduzzi-xiao}, the sheaves $\mathscr{H}_{\mathfrak{p}_{i},j}^{(t)}(A/S_0)_S$ and $\mathscr{H}_{\mathfrak{p}_{i},j}^{(t)}(A'/S_0)_S$ are rank-two $\calO_S$-subbundles of $\calH^1_\cris(A/S_0)_{S, \gothp_i,j} / \tilde \scrF^{(t-1)}_{\gothp_i,j}$ and $\calH^1_\cris(A'/S_0)_{S, \gothp_i,j} / \tilde \scrF'^{(t-1)}_{\gothp_i,j}$, respectively.
Then we need to lift $\mathscr{F}_{\mathfrak{p}_{i},j}^{(t)} / \mathscr{F}_{\mathfrak{p}_{i},j}^{(t-1)}\subset
\mathscr{H}_{\mathfrak{p}_{i},j}^{(t)}(A/{S_0})$ and
$\mathscr{F}_{\mathfrak{p}_{i},j}^{\prime(t)}/ \mathscr{F}_{\mathfrak{p}_{i},j}^{\prime(t-1)}\subset\mathscr{H}_{\mathfrak{p}_{i},j}^{(t)}(A'/{S_0})$ to (isotropic) rank-one $\calO_S$-subbundles $\mathscr{\tilde{F}
}^{(t)/(t-1)}_{\mathfrak{p}_{i},j}$ and $\mathscr{\tilde{F}}_{\mathfrak{p}_{i}%
,j}^{\prime(t)/(t-1)}$ of $\mathscr{H}_{\mathfrak{p}_{i},j}^{(t)}(A/{S_0})_S$ and $\mathscr{H}_{\mathfrak{p}_{i},j}^{(t)}(A'/{S_0})_S$ respectively, such that
\[
\phi^*\big(\mathscr{\tilde{F}%
}_{\mathfrak{p}_{i},j}^{\prime(t)/(t-1)} \big)\subseteq\mathscr{\tilde{F}}_{\mathfrak{p}_{i},j}^{(t)/(t-1)} \quad \textrm{and} \quad \psi^*\big(\mathscr{\tilde{F}%
}_{\mathfrak{p}_{i},j}^{(t)/(t-1)} \big)\subseteq\mathscr{\tilde{F}}_{\mathfrak{p}_{i},j}^{\prime(t)/(t-1)}.
\]
Once again, this isotropic condition is automatic by condition~\eqref{E:OF hermitian}.
After this, we define $\mathscr{\tilde{F}%
}_{\mathfrak{p}_{i},j}^{(t)} $ to be the preimage of $\mathscr{\tilde{F}%
}_{\mathfrak{p}_{i},j}^{(t)/(t-1)}$ under the natural projection $(
\mathscr{\tilde{F}%
}_{\mathfrak{p}_{i},j}^{(t-1)})^\perp
\twoheadrightarrow (\mathscr{\tilde{F}%
}_{\mathfrak{p}_{i},j}^{(t-1)})^\perp/\mathscr{\tilde{F}%
}_{\mathfrak{p}_{i},j}^{(t-1)}$ and define $\mathscr{\tilde{F}%
}_{\mathfrak{p}_{i},j}^{\prime(t)} $ similarly.\\
\end{enumerate}

Following \cite[Lemma 3.3.2]{pappas}, we claim that for any locally noetherian $\calO$-scheme $S_0$, the tuple
\[
\big(\mathscr{H}_{\mathfrak{p}_{i},j}^{(l)}(A/S_0), \mathscr{H}_{\mathfrak{p}_{i}%
,j}^{(l)}(A'/S_0),\phi^*_\scrH, \psi^*_\scrH \big)
\]
is Zariski locally isomorphic (over $S_0$) to the ``constant'' tuple $(\mathcal{O}_{S_0}^{2},\mathcal{O}_{S_0}^{2}, u_{\mathfrak{p}
_{i},j}^{(l)}, v_{\gothp_i,j}^{(l)})$ of \eqref{E:constant tuple}. Indeed, this is obvious at a closed point of $S_0$ of characteristic zero.
At a closed point $s\in S_0$ of characteristic $p$, we observe that $\psi^*_\scrH \circ \phi^*_\scrH$ and $\phi^*_\scrH \circ \psi^*_\scrH$ are both zero; and the kernels of $\phi^*_\scrH$ and $\psi^*_\scrH$ are at most one-dimensional, as the kernels of $\phi^*$ and $\psi^*$ on $H^1_\dR(A'/k(s))_{\gothp_i,j}$ and $H^1_\dR(A/k(s))_{\gothp_i,j}$ respectively are one-dimensional.  It follows that the images of $\psi^*_\scrH$ and $\phi^*_\scrH$ are both one-dimensional.
The claim is clear from that.

The isomorphism of this claim (by the definition of the moduli problem $\calN$) gives, for any given closed point $x_0$ of $\calM_\gothc(\gothp)$ in characteristic $p$,  a map $\varphi: \calU \to \calN$ we sought in (i).
Moreover, applying the discussion of deformation theory to the case of $S_0 = \Spec \kappa(x_0) \hookrightarrow S = \Spec \kappa(x_0)[\epsilon]/(\epsilon^2)$,  descriptions (1) and (2) of the deformation functor implies that $\varphi$ induces an isomorphism at the level of the Zariski tangent spaces at $x_0$ of the special fibers of $\mathscr{M}
_{x_0}$ and $\mathscr{N}_{\varphi(x_0)}$. 
This shows (ii).
Finally, the argument at the end of \cite[\S3]{deligne-pappas} (again using the description of the deformation theory above) implies that $\varphi_{x_0}$ is an isomorphism of formal schemes, and therefore $\varphi$ is \'etale at $x_0$. 
This completes the proof that $\calN$ is a local model of $\calM_\gothc(\gothp)$ and concludes the proof of this Proposition.
\end{proof}

\begin{notation}
\label{N:stratification Iwahori}
The proof of the above Proposition suggests us to define a stratification of $\calM(\gothp)_\FF$, following \cite{helm,goren-kassaei}.
The study of certain geometric structure of these strata is the key ingredient in proving our main technical result Proposition~\ref{P:ample=>small support} in \S \ref{Sec:section 4}.

For two subsets $\ttS, \ttS' \subseteq \Sigma_\gothp$ such that $\ttS\cup \ttS' = \Sigma_\gothp$, we write $\dot Y_{\ttS, \ttS'}$ for the subscheme of $\calM(\gothp)_\FF$ where
\begin{itemize}
\item
$\phi^*:  \scrF'^{(l)}_{\gothp_i,j} / \scrF'^{(l-1)}_{\gothp_i,j} \to  \scrF^{(l)}_{\gothp_i,j} / \scrF^{(l-1)}_{\gothp_i,j}$ vanishes if $\tau_{\gothp_i,j}^{l} \in \ttS$, and 
\item
$\psi^*:  \scrF^{(l)}_{\gothp_i,j} / \scrF^{(l-1)}_{\gothp_i,j} \to  \scrF'^{(l)}_{\gothp_i,j} / \scrF'^{(l-1)}_{\gothp_i,j}$ vanishes if $\tau_{\gothp_i,j}^{l} \in \ttS'$.
\end{itemize}
Note that $\phi^*\circ \psi^*$ and $\psi^*\circ \phi^*$ vanishes modulo $p$ by the moduli problem; so we need the condition $\ttS\cup \ttS' = \Sigma_\gothp$ otherwise $\dot Y_{\ttS, \ttS'}$ is nonempty.

We write $Y_{\ttS, \ttS'} \subseteq \Sh(\gothp)_\FF$ for the quotient of $\dot Y_{\ttS, \ttS'}$ by the action of $\calO_F^{\times, +} / (K \cap \calO_F^\times)^2$.
\end{notation}
\begin{lemma}
\label{C:strata iwahori}
\begin{enumerate}
\item
There are $3^{\#\Sigma_\gothp}$ closed strata $\dot Y_{\ttS, \ttS'}$.

\item
For another pair $\ttS_1, \ttS'_1 \subseteq \Sigma_\gothp$ such that $\ttS_1 \cup \ttS'_1 = \Sigma_\gothp$, we have \[
\label{E:intersection YttSttS'}
\dot Y_{\ttS, \ttS'} \cap  \dot Y_{\ttS_1, \ttS'_1} = \dot  Y_{\ttS\cup \ttS_1, \ttS' \cup \ttS'_1}.
\]
\item
Each $\dot  Y_{\ttS, \ttS'}$ is a smooth variety over $\FF$ of dimension
\[
g - (\#\ttS + \#\ttS'-\#\Sigma_\gothp).
\]

\item
The irreducible components of $\calM(\gothp)_\FF$ are exactly those given by $\dot  Y_{\ttS, \Sigma_\gothp \backslash \ttS}$ for $\ttS$ a subset of $\Sigma_\gothp$.
\item
The open strata are given by
\[
\dot Y_{\ttS, \ttS'}^\circ: = \dot  Y_{\ttS, \ttS'} \big\backslash \big( \cup_{(\ttS_1, \ttS'_1) \supsetneq (\ttS, \ttS')}  Y_{\ttS_1, \ttS'_1} \big),
\]
where $(\ttS_1, \ttS'_1) \supsetneq (\ttS, \ttS')$ means $\ttS_1 \supseteq \ttS$ and $\ttS'_1 \supseteq \ttS'$, and the equalities of sets cannot hold simultaneously.
\end{enumerate}
Similar statements hold for the strata $Y_{\ttS, \ttS'} \subseteq \Sh(\gothp)_\FF$.

\end{lemma}
\begin{proof}
(1) and (2) are immediate from the definition.

(3) follows from the local model computation in the proof of Proposition~\ref{P:LCI} that $\dot Y_{\ttS, \ttS'}$ is \'etale locally isomorphic to the spectrum of the following ring
\[
\FF[W_{\gothp_i,j}^{(l)}]_{\gothp_i\neq\gothp,j=1,\ldots,f_{i},l=1,\ldots,e_{i}%
}\otimes_{\FF} \bigotimes_{\tau_{\gothp, j}^l \in \Sigma_\gothp}
\begin{cases}
\FF[U_{\gothp,j}^{(l)},V_{\gothp,j}^{(l)}]/(U_{\gothp,j}
^{(l)}) & \textrm{if } \tau_{\gothp, j}^l \in \ttS,
\\
\FF[U_{\gothp,j}^{(l)},V_{\gothp,j}^{(l)}]/(V_{\gothp,j}
^{(l)}) & \textrm{if } \tau_{\gothp, j}^l \in \ttS',
\\
\FF[U_{\gothp,j}^{(l)},V_{\gothp,j}^{(l)}]/(U_{\gothp,j}
^{(l)}, V_{\gothp,j}
^{(l)}) & \textrm{if } \tau_{\gothp, j}^l \in \ttS \cap \ttS'.
\end{cases}
\]
It follows that $\dot Y_{\ttS, \ttS'}$ is smooth of the said dimension.

(4) and (5) are immediate corollaries of (2) and (3). The statements for $Y_{\ttS, \ttS'} \subseteq  \Sh(\gothp)_\FF$ follows from taking quotient by $\calO_F^{\times, +} / (K \cap \calO_F^\times)^2$.
\end{proof}

\begin{remark}
It seems that these strata $Y_{\ttS, \ttS'}$ are likely to be isomorphic to certain iterated $\PP^1$-bundles over (the special fiber of) some other quaternionic Shimura varieties. One expects that certain variants of the arguments in \cite{helm,tian-xiao} can be applied to our situation. 
\end{remark}


We have the following result.

\begin{proposition}
\label{P: fflat}
The morphisms of $\calO$-schemes $\pi_{1},\pi_{2}:\Sh(\gothp)\rightarrow\Sh$ are finite and flat over the ordinary
locus of $\Sh(\gothp)$.
\end{proposition}

\begin{proof}
It is enough to prove the statement for the analogous morphisms between fine moduli spaces. For a locally noetherian $\calO$-scheme $S$, we say that an $S$-point $((A, \lambda, \alpha, \underline\scrF ),(A', \lambda', \alpha', \underline\scrF' );\phi, \psi)$ of $\calM_\gothc(\gothp)$ is ordinary if and only if $A_1$ (or, equivalently, $A_2$) is an ordinary abelian scheme. 

Recall that we fixed a positive isomorphism $\theta_{\gothc}:\gothc\gothp\simeq\gothc'$, with $\gothc'\in\gothC$. We need to prove that, after restricted to the ordinary locus,  $\pi_1 = \pi_{1, \theta_\gothc}:\calM_\gothc(\gothp)^\ord = \calM_{\gothc, \theta_\gothc}(\gothp)^\mathrm{ord}\rightarrow\calM_\gothc^\mathrm{ord}$ and $\pi_2 = \pi_{2, \theta_\gothc}:\calM_\gothc(\gothp)^\mathrm{ord}\rightarrow\calM_{\gothc'}^\mathrm{ord}$ are finite and flat. 

By \cite[Remark 3.6]{AG} an ordinary abelian variety with RM automatically satisfies the Rapoport condition, and as such it admits exactly one filtration satisfying the Pappas--Rapoport conditions of \S\ref{S:moduli HBAS}. In what follows we can then forget about the filtrations appearing in the tuples classified by $\calM_\gothc(\gothp)^\mathrm{ord},\calM_{\gothc}^\mathrm{ord}$, and $\calM_{\gothc'}^\mathrm{ord}$. The proof of the proposition is now analogous to the one of \cite[V, Lemme 1.12]{deligne-rapoport}. For brevity, we only show the finite flatness of:
\[
\pi_1: \calM_\gothc(\gothp)^\mathrm{ord}\rightarrow\calM_\gothc^\mathrm{ord}.
\]
Since $\pi_1$ is proper, it is enough to show that it is quasi-finite and the rank of the geometric fibers of $\pi_1$ is constant. Then $\pi_1$ is forced to be finite and flat (as $\mathcal{M}_\gothc^{\mathrm{ord}}$ is reduced) by \cite[V, Lemme 1.13]{deligne-rapoport}.

Let $k$ be an $\mathcal{O}$-algebra which is an algebraically closed field of characteristic $p$,
and let $x:\Spec k\rightarrow\mathcal{M}_\gothc^{\mathrm{ord}}$ be a
 $k$-point of $\mathcal{M}_\gothc^\mathrm{ord}$ defining a tuple $(A,\lambda
,\alpha)$. (Recall that we can forget about the filtrations). By Remark~\ref{R:alternative Gamma 0p}, the fiber $T$ of $\pi_{1}$ over $x$ is the $k$-scheme representing the functor which assigns to a locally noetherian $k$-scheme $S$ the set of
isomorphism classes of finite-flat, closed, $\mathcal{O}_{F}$-stable, $\lambda$-isotropic, $S$-subgroup schemes 
$C\subset(A\times_{k}S)[\gothp]$ of rank
$p^f$.

Since $A$ is ordinary, the connected-\'{e}tale exact sequence of
$A[\mathfrak{p}]$ is of the form:
\begin{equation}
0\rightarrow\FF_\gothp\otimes_{
\mathbb{Z}}\mu_{p/k}\rightarrow A[\mathfrak{p}]\rightarrow
\FF_{\mathfrak{p}}\otimes_{
\mathbb{Z}}\left(  \tfrac{\ZZ}{\underline{p\ZZ}}\right)_{k}\rightarrow0,\label{E:conn-et}
\end{equation}
where the morphisms are equivariant for the natural action of $\calO_F/\gothp=\FF_\gothp$.
If $C\subset(A\times_{k}S)[\mathfrak{p}]$ represents an $S$-point of $T$, we can write $S=S^{\prime}\coprod S^{\prime\prime}$ where
$C\times_{S}S^{\prime}$ is equal to the connected part $(A\times_{k}S^{\prime})[\mathfrak{p}]^{\circ}\simeq\FF_{\mathfrak{p}}\otimes_{\ZZ}\mu_{p/S'}$ of $(A\times_{k}S^{\prime})[\mathfrak{p}]$, and $C\times_{S}S''$ is isomorphic to $\FF_\gothp\otimes_{\ZZ}(\underline{\ZZ/p\ZZ})_{S^{\prime\prime}}$. (To see this notice that if $y:\Spec l\rightarrow S$ is a closed point of $S$ for some field extension $l$ of $k$, and if the group of geometric points $C_y(\bar{l})$ of the fiber of $C$ at $y$ is non-trivial, then the existence of an action of $\FF_\gothp$ on $C_y(\bar{l})$ forces $C_y$ to be isomorphic to $\FF_\gothp\otimes_\ZZ (\underline{\ZZ/p\ZZ})_{/l}$. On the other hand, if $C_y$ has trivial \'etale quotient, then it is contained in the connected component of the identity of $(A\times_k l)[\gothp]$; the existence of the $\FF_\gothp$-action then forces this inclusion to be an equality).

The decomposition $S=S'\coprod S''$ induces a corresponding decomposition $T=T'\coprod T''$ where $T'$ is the reduced $k$-scheme that assigns $(A\times
_{k}S)[\mathfrak{p}]^{\circ}$ to any locally noetherian $k$-scheme $S$, while
$T^{\prime\prime}$ represents the functor of $\FF_\gothp$-equivariant splittings of the exact sequence (\ref{E:conn-et}).
Since $T^{\prime\prime}$ is a torsor under the group-scheme
$\operatorname*{Hom}_{\FF_\gothp\otimes_{
\mathbb{Z}}k}(A[\mathfrak{p}]^{\mathrm{\acute{e}t}},A[\mathfrak{p}]^{\circ})\simeq\FF_{\mathfrak{p}}\otimes_{
\mathbb{Z}}\mu_{p/k}$, we see that $T\simeq\Spec k\coprod
(\FF_\gothp\otimes_{
\mathbb{Z}}\mu_{p/k})$ is finite over $k$ of constant rank equal to $1+p^{f}$.

Let us now assume that $k$ is an algebraically closed field of characteristic zero, so that $A[\gothp]\simeq(\FF_\gothp\otimes_{\ZZ}(\underline{\ZZ /p \ZZ})_{k})^2$ is \'etale. By fixing an arbitrary $\FF_\gothp$-stable, closed, and $\lambda$-isotropic $k$-subgroup scheme of $A[\gothp]$ isomorphic to $\FF_\gothp\otimes_{\ZZ}(\underline{\ZZ /p \ZZ})_{k}$ and considering the corresponding exact sequence, one sees via arguments similar to the ones above that there is an isomorphism $T\simeq\Spec k\coprod
(\FF_\gothp\otimes_{\mathbb{Z}}(\underline{\ZZ/p\ZZ})_{k})$, so that the fiber $T$ is still finite over $k$ of rank $1+p^{f}$.

This shows that, in either case,  $\pi_1$ is quasi-finite of constant rank $1+p^f$ and hence concludes the proof this Proposition.
\end{proof}

\begin{remark}
When $g>1$, the morphisms $\pi_{1}$ and $\pi_{2}$ are not finite over the
non-ordinary part of $\Sh$. This phenomenon already occurs when $p$
is unramified in $F$ (\cite{stamm}).
\end{remark}

\subsection{Extension of Hecke correspondence to the toroidal compactification}
\label{S:extension to toroidal compactification}
For each $\gothc\in\mathfrak{C}$ fix sufficiently fine rational admissible polyhedral cone decompositions for the cusps of the Rapoport locus of $\calM_\gothc(\gothp)$. One can then
construct smooth toroidal compactifications $\calM_\gothc(\gothp)^{\tor},\calM(\gothp)^\tor$, and $\Sh(\gothp)^\tor$ of the splitting models with Iwahori level structure, as in \cite[\S2.11]{reduzzi-xiao}.
Here we require that the cone decomposition  we chose at each cusp of the $\calM_\gothc(\gothp)^\tor$ is a smooth refinement of the pull-back via $\pi_1$ and $\pi_2$ of the cone decomposition at the corresponding cusps of $\calM_\gothc$ and $\calM_{\gothc'}$, respectively. This way, $\pi_1$ and $\pi_2$ extend to maps $\pi_1, \pi_2: \Sh(\gothp)^\tor \to \Sh^\tor$.
The restriction of $\pi_1$ to the ordinary locus $\pi_1^{\tor,\ord}: \Sh(\gothp)^{\tor, \ord} \to \Sh^{\tor, \ord}$ may no longer  be finite and flat (due to the refinement of the cone decomposition), but it is still true that 
\begin{equation}
\label{E:pi1 acyclic}
R\pi^{\tor,\ord}_{1,*}\omega^\kappa \cong \pi^{\tor,\ord}_{1,*}\omega^\kappa.
\end{equation}
See the proofs of \cite[Lemma~7.1.1.4]{lan} and \cite[Proposition~7.5]{lan2}.

\subsection{Construction of $T_{\mathfrak{p}}$}\label{S:def of Tp}

We now construct the Hecke operator $T_\gothp$ over
the $\Spec\calO$-scheme $\Sh$, extending a geometric construction of B. Conrad (\cite[\S4.5]{conrad}).

Recall that for each fractional ideal $\gothc\in\mathfrak{C}$ we fixed an isomorphism $\theta_\gothc:\gothc\gothp\simeq\gothc'$ of fractional ideals with positivity such that
$\gothc'\in\mathfrak{C}$. Moreover we denoted by $\pi_1 = \pi_{1, \theta_\gothc}:\calM_{\gothc}(\gothp) = \calM_{\gothc, \theta_\gothc}(\gothp)\rightarrow\calM_\gothc$ and $\pi_2:=\pi_{2,\theta_\gothc}:\calM_\gothc(\gothp)\rightarrow\calM_{\gothc'}$ the ``taking the source'' and ``taking the target'' morphisms at the level of fine moduli spaces. Denote by $f:\calA_\gothc\rightarrow\calM_\gothc$ (resp. $f':\calA_{\gothc'}\rightarrow\calM_{\gothc'}$) the universal abelian scheme over $\calM_\gothc$ (resp. $\calM_{\gothc'}$). Set $\dot\omega:=f_\ast\Omega^1_{\calA_\gothc/\calM_\gothc}$ and $\dot\omega':=f'_{\ast}\Omega^1_{\calA_{\gothc'}/\calM_{\gothc'}}$: these are bundles of rank $g$ over $\calM_\gothc$ and $\calM_{\gothc'}$ respectively.
Similarly, let $\calA(\gothp)$ and $\calA'(\gothp)$ denote the two universal abelian varieties over $\calM_{\gothc}(\gothp)$ and let $f_\gothp: \calA(\gothp) \to \calM_\gothc(\gothp)$ and $f'_\gothp: \calA'(\gothp) \to \calM_\gothc(\gothp)$ denote the natural morphisms.

Define the following morphism of rank-$g$ bundles over $\calM_\gothc(\gothp)$:
\[
\beta':\pi_2^\ast\dot\omega' \cong f'_{\gothp\ast}\Omega^1_{\calA'(\gothp)/\calM_\gothc(\gothp)}\xrightarrow{\phi^*} f_{\gothp\ast}\Omega^1_{\calA(\gothp)/\calM_\gothc(\gothp)} \cong \pi_1^\ast\dot\omega,
\]
where the first and the last isomorphism are induced by base change and the middle arrow is induced by pullback of differentials. There is an analogously defined map at the level of de Rham sheaves:
\[
 \beta'':\pi_2^\ast\calH^1_\dR(\calA_{\gothc'}/\calM_{\gothc'})\rightarrow\pi_1^\ast\calH^1_\dR(\calA_\gothc/\calM_\gothc),
\]
which induces an isomorphism:
$$\pi_2^\ast\wedge^2_{\calO_F\otimes_\ZZ\calO_{\calM_{\gothc'}}}\calH^1_\dR(\calA_{\gothc'}/\calM_{\gothc'})
\overset{\simeq}\lra
\gothp\cdot\pi_1^\ast\wedge^2_{\calO_F\otimes_\ZZ\calO_{\calM_\gothc}}\calH^1_\dR(\calA_\gothc/\calM_\gothc).$$
In particular, the map $\pi_2^\ast\dot\varepsilon'_\tau\to\pi_1^\ast\dot\varepsilon_\tau$ induced by $\beta''$ is an isomorphism if $\tau\notin\Sigma_\gothp$, and it is an isomorphism onto $\tau(\varpi_\gothp)\cdot\pi_1^\ast\dot\varepsilon_\tau$ otherwise.

Fix a paritious weight $\kappa=((k_\tau)_{\tau\in\Sigma},w)$ such that $k_\tau \geq 1$ for all $\tau\in\Sigma$. 
Take a positive integer $N$ sufficiently large so that $w+N\geq k_\tau$ if $\tau\in\Sigma_\gothp$.
Observe that $\beta''$ induces an isomorphism 
$$\beta''_\kappa:\bigotimes_{\tau\in\Sigma}p^N \cdot \pi_2^\ast\dot\varepsilon'^{\otimes(w-k_\tau)/2}_\tau\overset{\simeq}\lra
\pi_{\gothp,\kappa, N}\cdot\bigotimes_{\tau\in\Sigma}\pi_1^\ast\dot\varepsilon_\tau^{\otimes(w-k_\tau)/2}$$ 
where $\pi_{\gothp,\kappa, N}:=p^{N \cdot\# \Sigma}\prod_{\tau\in\Sigma_\gothp}\tau(\varpi_\gothp)^{(w-k_\tau)/2}$.
Twisting the sheaf $\dot\omega$ and $\dot \omega'$ by the character attached to the tuple $(k_\tau)_{\tau\in\Sigma}$, and applying the maps $\beta'$ and $p^{N\cdot \#\Sigma}\pi_{\gothp,\kappa, N}^{-1}\cdot\beta''_\kappa$ one obtains a morphism of invertible sheaves: \begin{equation}
\label{E:beta}    \beta:\pi_2^\ast\dot\omega'^\kappa\rightarrow\pi_1^\ast\dot\omega^\kappa.
\end{equation}
This map does not depend on the auxiliary choice of $N$.  But one should note that this map depends on the choice of the uniformizers $\varpi_\gothp$, except when all $k_\tau$ for $\tau\in \Sigma_\gothp$ are equal, we may take the canonical choice $\pi_{\gothp, \kappa, N}: = p^{N \cdot \#\Sigma} \cdot \NN(\gothp)^{(w-k_\tau)/2}$ to eliminate the ambiguity.
 
Denote by $\mathscr{D}$ the relative dualizing sheaf of the smooth scheme $\calM_\gothc\rightarrow\Spec\calO$. The canonical identification 
$\mathscr{D}=\bigwedge^g_{\calO_{\calM_\gothc}}\Omega_{\calM_\gothc/\Spec\calO}^1$, together with the isomorphisms \eqref{E:KS isomorphisms} give rise to a canonical isomorphism
\[
 KS:\mathscr{D}\overset{\simeq}\longrightarrow\dot\omega^{({\bf 2},0)}.
\]
Denote by $\mathscr{D}_\gothp$ the relative dualizing sheaf of $\calM_\gothc(\gothp)\rightarrow\Spec\calO$: it exists as an invertible sheaf on
$\calM_\gothc(\gothp)$ since the latter is a flat local complete intersection over $\Spec\calO$, by Proposition \ref{P: fflat}. 

We now construct a canonical morphism of sheaves $\xi:\pi_1^\ast\mathscr{D}\rightarrow\mathscr{D}_\gothp$ as follows. First, Proposition \ref{P:LCI} implies that the complement of the $\calO$-smooth locus $\calM_\gothc(\gothp)^\mathrm{sm}$ of $\calM_\gothc(\gothp)$ is of codimension 2 inside $\calM_\gothc(\gothp)$. Since $\calM_\gothc(\gothp)$ is Cohen--Macaulay, it suffices to construct the desired morphism $\xi$ over $\calM_\gothc(\gothp)^\mathrm{sm}$. But over $\calM_\gothc(\gothp)^{\mathrm{sm}}$, the dualizing module $\mathscr{D}_\gothp$  is given by $\bigwedge^g_{\calO_{\calM_\gothc(\gothp)^{\mathrm{sm}}}}\left(\Omega_{\calM_\gothc(\gothp)^{\mathrm{sm}}/\Spec\calO}^1\right)$; so the natural morphism $(\pi_1^{\mathrm{sm}})^\ast\Omega^1_{\calM_\gothc/\Spec\calO}\to\Omega^1_{\calM_\gothc(\gothp)^\mathrm{sm}/\Spec\calO}$ induces the sought-for morphism $\xi^\mathrm{sm}$ between their top exterior powers (here $\pi_1^\mathrm{sm}$ denotes the restriction of $\pi_1$ to $\calM_\gothc(\gothp)^\mathrm{sm}$).\\

We now combine the above maps and denote by $\eta$ the composition of $R\pi_{1\ast}\beta:R\pi_{1\ast}\pi_{2}^{\ast}\dot\omega'^{\kappa}  \to R\pi_{1\ast}\pi_1^\ast\dot\omega^{\kappa}$ with the following morphism in the derived category 
$D_\mathrm{coh}^b(\calM_\gothc)$:
\begin{align*}
R\pi_{1\ast}\pi_1^\ast\dot\omega^{\kappa}
=\dot\omega^{\kappa-({\bf 2},0)}\otimes R\pi_{1\ast}\pi_1^\ast\dot\omega^{({\bf 2},0)}
&\overset{\small{1\otimes KS^{-1}}}\longrightarrow\dot\omega^{\kappa-({\bf 2},0)}\otimes R\pi_{1\ast}\pi_1^\ast\mathscr{D}\\
\overset{1\otimes\xi}\longrightarrow\dot\omega^{\kappa-({\bf 2},0)}\otimes R\pi_{1\ast}\mathscr{D}_\gothp
&\overset{1\otimes\mathrm{tr}_{\pi_1}} \longrightarrow\dot\omega^{\kappa-({\bf 2},0)}\otimes\mathscr{D} \\
\overset{1\otimes KS}\longrightarrow\dot\omega^{\kappa-({\bf 2},0)}\otimes\dot\omega^{({\bf 2},0)}
& =\dot\omega^\kappa,
\end{align*}
where $\mathrm{tr}_{\pi_1}:R\pi_{1\ast}\mathscr{D}_\gothp\rightarrow\mathscr{D}$ denotes the trace morphism of normalized dualizing complexes associated to
the proper dominant morphism $\pi_1:\calM_\gothc(\gothp)\rightarrow\calM_\gothc$. Recall that $\mathrm{tr}_{\pi_1}$ is non-zero by \cite[Proposition 2.13]{BST}, and it is compatible with localizations on the base scheme (\emph{loc.cit.}).

Applying the construction of $\eta:R\pi_{1\ast}\pi_{2}^{\ast}\dot\omega'^{\kappa}\rightarrow\dot\omega^\kappa$ to each component $\calM_\gothc$ of $\calM$ and quotienting by the action of $\calO_F^{\times,+}/(K \cap \calO^\times_{F})^2$ we obtain a well-defined morphism $$\eta:R\pi_{1\ast}\pi_{2}^{\ast}\omega'^{\kappa}\rightarrow\omega^\kappa$$ in $D_\mathrm{coh}^b(\Sh)$. 
We remark that this morphism $\eta$ over $\Sh$ do not depend on the choice of identifications $\theta_\gothc:\gothc\gothp\simeq\gothc'$ for $\gothc\in\mathfrak{C}$.

Ideally, we would like to extend $\eta$ to a morphism over the toroidal compactification and show that this construction does not depend on the choice of the cone decomposition of $\calM_\gothc(\gothp)$. (In particular, we could consider the universal semi-abelian varieties and show that the corresponding map $\beta''_\kappa$ is divisible by the correct power of $\varpi$.) However, we content ourselves by making a remark that one can certainly define $\eta$ over the \emph{generic} fiber of the toroidal compactification (just as how we define tame Hecke operators before), therefore, for some positive integer $M$, $p^M\eta$ extends uniquely to a morphism $p^M\eta: R\pi_{1*}\pi_2^*\omega'^\kappa \to \omega^\kappa$ in $D^b_\mathrm{coh}(\Sh^\tor)$.
(But then Proposition~\ref{P:Tp factor}, or more precisely Proposition~\ref{P:image p^g}, below shows that $\eta$ itself extends uniquely to a morphism $\eta: R\pi_{1*}\pi_2^*\omega'^\kappa \to \omega^\kappa$.)

Now the key technical result is the following.
\begin{proposition}
\label{P:Tp factor}
Assume that the weight $\kappa$ satisfies $\sum_{\tau\in \Sigma_\gothp} k_\tau \geq ef$ and the following conditions:
\begin{eqnarray}
\label{E:ampleness condition} &
k_{\tau_{\gothp, j}^{l+1}} \geq k_{\tau_{\gothp,j}^l} \quad \textrm{for all } j =1, \dots, f, \textrm{ and }l = 1, \dots, e-1; \quad \textrm{and}
\\
\nonumber
&
pk_{\tau_{\gothp, j}^1} \geq k_{\tau_{\gothp,j-1}^e} \quad \textrm{for all } j =1, \dots, f.
\end{eqnarray}
Then the morphism $p^M\eta$ in $D^b_\mathrm{coh}(\Sh^\tor)$ uniquely factors as 
\begin{equation}
\label{E:Tp factor}
R\pi_{1*}\pi_2^*\omega'^\kappa \xrightarrow{\frac1{p^f} \eta}\omega^\kappa \xrightarrow{\cdot p^{f+M}} \omega^\kappa,
\end{equation}
where the second map is the multiplication by $p^{f+M}$.
\end{proposition}
Note that the condition~\eqref{E:ampleness condition} automatically forces all $k_\tau \geq 0$ for $\tau \in \Sigma_\gothp$.
We also point out that the condition~\eqref{E:ampleness condition} is similar to the conjectural ampleness condition in \cite[Theorem~1.9]{tian-xiao}, at least when $\gothp$ is unramified over $p$.

The proof of Proposition~\ref{P:Tp factor} will be given later in \S\ref{S:proof of Tp factor}.\footnote{After correcting the gap pointed out by Pilloni, we were informed by him that he has a different and simpler proof of Proposition~\ref{P:Tp factor}. But our proof reveals some interesting finer geometry of the map $\pi_1: \calM(\gothp) \to \calM$.}
We assume Proposition~\ref{P:Tp factor} temporarily to complete our construction of the operator $T_\gothp^\rmn$ on cohomology. 
The first map in \eqref{E:Tp factor} gives a canonically defined morphism $\frac{1}{p^f}\eta:R\pi_{1\ast}\pi_{2}^{\ast}\omega'^{\kappa}\rightarrow\omega^{\kappa}$. For $m\in\ZZ_{> 0}\cup\{\infty\}$ we denote by $\omega^\kappa_m$ the sheaf $\omega^\kappa/(\varpi^m)$, with the convention that $\varpi^\infty:=0$. The morphism $\frac{1}{p^f}\eta$ induces a morphism
$$\tilde\eta_m:R\pi_{1\ast}\pi_{2}^{\ast}(\omega'^{\kappa}_m)\rightarrow\omega^{\kappa}_m.$$

\begin{definition} 
We define the action of the \emph{Hecke operator} $T_\gothp^\rmn$ on the cohomology of $\omega^\kappa_m$ as the following composition:
$$T_\gothp^\rmn:H^i(\Sh^\tor,\omega^\kappa_m)  \overset{\pi_2^\ast}\lra H^i(\Sh(\gothp)^\tor,\pi_2^\ast\omega^\kappa_m)\overset{\pi_{1\ast}}\lra H^i(\Sh^\tor,R\pi_{1\ast}\pi_2^\ast(\omega^\kappa_m))\overset{\tilde\eta_m}\lra H^i(\Sh^\tor,\omega^\kappa_m).$$
\end{definition}

\begin{remark}
Let $R$ be either a finite extension of $\QQ_p$ or the ring of integers of a finite extension of $\QQ_\ell$ where $K_\gothq = \GL_2(\calO_{F_\gothq})$ for all primes $\gothq$ above $\ell$. The moduli schemes $\calM_R$ and
$\calM(\gothp)_R$ defined in the obvious way over $\Spec R$ are both smooth, and the natural maps $\pi_{1,R},\pi_{2,R}:\Sh(\gothp)_R\rightarrow\Sh_R$ are finite and flat. In particular, one can define the ``usual'' Hecke operator $T_{\gothp,R}^\rmn$ acting on the cohomology of $\Sh_R$ by means of the finite-flat trace map attached to $\pi_{1,R}$. The compatibility between the finite-flat trace map and the dualizing trace map implies that the operator $T_\gothp^\rmn$ defined above coincides with the classical Hecke operator $T_{\gothp,R}^\rmn$ in these settings. 
\end{remark}

\begin{remark}\label{R:q-exp}
The $q$-expansion of $T_\gothp^\rmn f$ for $f\in H^0(\Sh_m^\tor,\omega^\kappa_m)$ is ``as expected'', namely it is the expression in the second line of \eqref{E:q-expansion Tp}
divided by the normalizing factor $p^f\prod_\tau \tau(\varpi_\gothp)^{(w-k_\tau)/2}$. In particular, for weight $({\bf n}, n-2)$ and $p$ is inert (with $\varpi_\gothp = p$), this normalizing factor disappears.

In particular, if $p$ is inert in $F$ and we take $\varpi_\gothp = p$ and $\vartheta_\gothc = p$, the formula \eqref{E:q-expansion Tp} (at the cusps defined therein) simplifies to
$$a_\xi(T_p^\rmn( f),\calC, \Tate_{\frak{a},\frak{b}}) = a_{p\xi}(f,\calC',\Tate_{\frak{a},\frak{b}})+ \prod_{\tau\in\Sigma} p^{k_\tau -1} a_{p^{-1}\xi}(f,\calC''_{p^{-1}},\Tate_{\frak{a},\frak{b}}).$$
Notice that the formula is meaningful in $\calO/(\varpi^m)$ since $k_\tau\geq 1$. 

Finally, we remind the reader that we have introduced a normalization by the factor $\pi_{\gothp,\kappa, N}$ in the definition of our Hecke operators (see \eqref{E:beta}). In particular, since $\omega^{({\bf 0}, 2)}$ is a canonically trivial vector bundle on $\Sh^\tor$, 
the natural isomorphism
\[
H^*(\Sh^\tor, \omega_m^{((k_\tau), w)}) \xrightarrow{\simeq}
H^*(\Sh^\tor, \omega_m^{((k_\tau), w+2w_0)})
\]
for any integer $w_0$ is equivariant for the action of $T_\gothp^\rmn$.

\end{remark}

\begin{remark}\label{R:compatibility}
 The morphism $KS\circ(\mathrm{tr}_{\pi_1}\circ\xi)\circ KS^{-1}:R\pi_{1\ast}\pi_1^\ast\dot\omega^{({\bf 2},0)}\rightarrow\dot\omega^{({\bf 2},0)}$ that appears in the composition defining $\eta$ coincides, over an open subscheme of $\calM_\gothc$ on which the map $\pi_1$ is finite flat, with the usual finite flat trace map $\pi_{1\ast}\pi_1^\ast\dot\omega^{({\bf 2},0)}\rightarrow\dot\omega^{({\bf 2},0)}$. This follows from the compatibility between the dualizing trace map and the finite flat trace map (cf. \cite{conrad_book}). In particular, when $g=1$ (so that $\pi_1$ is finite flat over the entire space $\calM=\calM_\gothc$), our construction \emph{coincides} with the construction given in \cite[\S4.5]{conrad}.
\end{remark}

\subsection{Outline of the proof of Proposition~\ref{P:Tp factor}}
\label{S:proof of Tp factor}
Since the morphism $\eta$ is obtained by taking invariants under $\calO_F^{\times,+}/(K \cap\calO_{F}^{\times})^2$ of the (homonymous) morphism $R\pi_{1\ast}\pi_2^\ast\dot\omega^\kappa\to\dot\omega'^\kappa$ on $\calM^\tor$, it suffices to prove the result for the latter morphism, as a morphism in $D^b_\mathrm{coh}(\calM^\tor)$.
This will follow from the three propositions below.

\begin{proposition}\label{P:image p^g} 
Suppose that $\sum_{\tau\in \Sigma_\gothp} k_\tau \geq e f$. Let $M$ be as in end of \S \ref{S:def of Tp}.
Restricting $p^M\eta$ to $\calM^{\tor, \ord}$ (noting that $\pi_1$ has no higher derived pushforward over $\calM^{\tor,\ord}$), the homomorphism $p^M\eta:\big(\pi_{1\ast}\pi_{2}^{\ast}\dot\omega'^{\kappa}\big)\big|_{\calM^{\tor, \ord}}\rightarrow\dot\omega^{\kappa}|_{\calM^{\tor, \ord}}$ of coherent sheaves factors uniquely as
\[
\big(\pi_{1\ast}\pi_{2}^{\ast}\dot\omega'^{\kappa}\big)\big|_{\calM^{\tor, \ord}} \to \dot\omega^{\kappa}|_{\calM^{\tor, \ord}} \xrightarrow{\ \cdot p^{f+M}\ } \dot\omega^{\kappa}|_{\calM^{\tor, \ord}}.
\]
\end{proposition}

\begin{proposition}
\label{P:abstract hom}
Suppose that the support of $R^t\pi_{1,*}\pi_2^*\dot \omega'^\kappa_{\FF}$ has codimension at least $t+1$ in the special fiber $\calM_{\FF}$ for all $t \geq 1$.
Then $\eta$ factors uniquely as
\begin{equation}
\label{E:sigma factor}
R\pi_{1*}\pi_2^*\dot\omega'^\kappa \to \dot\omega^\kappa \xrightarrow{\ \cdot p^f\ } \dot\omega^\kappa
\end{equation}
in $D^b_\mathrm{coh}(\calM)$
if the restriction of $\eta$ does on the ordinary locus (as shown in Proposition~\ref{P:image p^g}).
\end{proposition}

\begin{proposition}
\label{P:ample=>small support}
Suppose that the weight $\kappa$ satisfies the condition~\eqref{E:ampleness condition}.
Then the assumption in Proposition~\ref{P:abstract hom} holds.
\end{proposition}

Proposition~\ref{P:abstract hom} is a formal homological algebra result.
Proposition~\ref{P:image p^g} is a straightforward generalization of  \cite[Theorem 4.5.1]{conrad}.
The proof of these two propositions will be given shortly in \S\ref{S:proof of image p^g} and \S\ref{S:proof of abstract hom}, after we summarize the basic idea of the proof of Proposition~\ref{P:ample=>small support} (whose proof will be given in Section~\ref{Sec:section 4}).

By Grothendieck's formal function theorem, the cohomological dimension of $R\pi_{1,*}\pi_2^*\dot \omega'^ \kappa_{\FF}$ is bounded by the dimension of the fiber of the map $\pi_1:\calM(\gothp)_{\FF} \to \calM_{ \FF}$.
Since the source and the target of $\pi_1$ have the same dimension, localizing at a codimension $t$ point $x$ of $\calM_{\FF}$, $R^{>t}\pi_{1,*}\pi_2^*\dot \omega'^\kappa_{\FF}$ is automatically trivial, and the contribution to the localization of $R^t\pi_{1,*}\pi_2^*\dot \omega'^\kappa_{\FF}$ at $x$ only comes from the irreducible components of $\calM(\gothp)_{\FF}$ of (maximal) dimension $g$ (see Lemma~\ref{C:strata iwahori} for the description of these strata).
Recall that $\calM_{\FF}$ admits the Goren--Oort stratification (cf. \S\ref{S:Goren-Oort stratification}). The proof of Proposition~\ref{P:ample=>small support} consists of the following ingredients.
\begin{itemize}
\item
The image of each irreducible component of $\calM(\gothp)_{\FF}$ is a closed Goren--Oort stratum (Proposition~\ref{P:geometry of HMV(p)} below); 
so it suffices to look at those open Goren--Oort stratum $X_\ttT^\circ$ whose closure is the image of some irreducible component of $\calM(\gothp)_{\FF}$;
\item
over the \emph{geometric} generic point of the  open stratum $X_\ttT^\circ$, 
only the dimension $t = \#\ttT$ fibers are relevant (by Proposition~\ref{P:formal function} below), namely the ones given by the irreducible components of $\calM(\gothp)_{\FF}$; and
\item  these dimension $t$ fibers are unions of products of $\PP^1$'s, such that the restriction of $\pi_2^*\dot \omega'^\kappa$ is $\calO(n)$ on each $\PP^1$-factor for some $n \geq -1$ (cf. \S\ref{S:proof of main prop} below)
\end{itemize}

\subsection{Proof of Proposition~\ref{P:image p^g}}
\label{S:proof of image p^g}
Since $\dot\omega^\kappa$ is an invertible sheaf on a scheme $\calM^{\tor, \ord}$ \emph{smooth} over $\Spec\calO$, it suffices to prove the proposition over a Zariski dense open subscheme, for example, the ordinary locus $\calM^\ord\subset\calM^{\tor, \ord}$ before taking the toroidal compactification. This is further equivalent to checking for the stalks at each closed point of $\calM^\ord$.
Let $k$ be a separably closed extension of $\FF$ and let $x:\Spec k\rightarrow\calM^{\textrm{ord}}$ be a map of $\Spec\calO$-schemes. Let
$y:\Spec k\rightarrow\calM(\gothp)^{\textrm{ord}}$ be
an element in the scheme-theoretic fiber $\pi_{1}^{-1}(x)$. Denote by $R_{x}$
the strictly henselian local ring of the stalk of $\calO_{\calM}$ at $x$, and similarly for $R_{y}$ and
$R_{\pi_{2}(y)}$. 

By Proposition \ref{P: fflat}, the morphism $\pi_1$ is finite flat over the ordinary locus of $\calM_{\FF}$, so that the fiber $\pi_{1}^{-1}(x)$ is a finite scheme over $\Spec k$ and the map $R_{x}\rightarrow R_{y}$ is finite flat. By Remark \ref{R:compatibility} it then suffices to show that the composition:
\begin{equation}
\dot\omega_{\pi_{2}(y)}^\kappa\overset{\eta_y^\kappa}{\longrightarrow}\dot\omega_{\gothp,y}^\kappa\simeq
(\pi_1^\ast\dot\omega^\kappa)_y=
R_{y}\otimes_{R_{x}}\dot\omega_{x}^\kappa\overset{\mathrm{Tr}_{y|x}\otimes1}%
{\longrightarrow}\dot\omega_{x}^\kappa\label{comp}%
\end{equation}
has image in $p^f\dot\omega_{x}^\kappa$. Here $\eta_y^\kappa$ is induced by pulling back differentials, the isomorphism is induced by the ``contraction map'', and
$\mathrm{Tr}_{y|x}:R_{y}\rightarrow R_{x}$ is the $R_{x}$-linear \emph{finite flat} trace map. In particular, the information of the Kodaira--Spencer isomorphism is contained in the trace map $\mathrm{Tr}_{y|x}$.

Assume that $x$ corresponds to the abelian scheme $A/\Spec k$
with extra structures, and that $y$ corresponds to an isogeny $A\rightarrow A'$
defined over $\Spec k$ and with kernel $C$ (notice we can forget about the filtrations because
we are working over the ordinary locus).
The closed, finite flat $k$-group scheme $C\subset A[\mathfrak{p}]$ has rank
$p^{f}$ and it comes with an action of $\calO_F/\mathfrak{p}=:\FF_\gothp$.

We now prove the desired result distinguishing two cases, depending on whether $C$ is \'etale or multiplicative (no other possibilities occur in our settings as $A$ is ordinary and $C$ is $\calO_F$-stable). We work with the strict henselization of the stalks modulo $p^f$.

\underline{Case 1}: $C$ is multiplicative. In this case the kernel of the universal isogeny $\phi$ over $R_y/(p^f)$ is isomorphic to $\FF_\gothp\otimes_\ZZ\mu_{p/R_y/(p^f)}$. Recall that $\Sigma_\gothp$ denotes the set of field embeddings $F\to\overline\QQ_p$ inducing the $p$-adic place $\gothp$. For each $\tau\in\Sigma_\gothp$ the pull-back map $(\dot\omega_{\pi_{2}(y)})_\tau\to(\dot\omega_{\gothp,y})_\tau$ is zero modulo $\tau(\varpi_\gothp)$. In particular, $\eta_y^\kappa$ is zero modulo $p^f$ as desired, since hence $\sum_{\tau\in\Sigma_\gothp}k_\tau\geq ef$.

\underline{Case 2}: $C$ is \'{e}tale. The $\hat{R}_{x}/(p^f)$-algebra $\hat{R}_{y}/(p^f)$
(the ``hat'' denotes completion) classifies splittings of the connected-\'{e}tale sequence
of $\calA[\gothp]$, where $\calA$ is the universal ordinary abelian scheme over $\hat{R}_{x}/(p^f)$. 
These splittings are a torsor
under the group-scheme 
\[
\mathrm{Hom}_{\FF_\gothp,\Spec(\hat{R}_{x}/(p^f))}(\mathcal{A}[\mathfrak{p}]^{\mathrm{\acute{e}t}},\mathcal{A}[\mathfrak{p}]^{\circ})\simeq
\FF_\gothp\otimes_\ZZ\mu_{p/\hat{R}_{x}/(p^f)}
\]
defined over $\hat{R}_{x}/(p^f)$. Such torsors are classified by $H_{\operatorname*{fppf}}^{1}(\Spec \hat{R}_{x}/(p^f),
\FF_\gothp\otimes_\ZZ\mu_{p/\hat{R}_{x}/(p^f)})$, which by Kummer theory is equal to $\FF_\gothp\otimes_\ZZ
(\hat R_x/(p^f))^\times/(\hat R_x/(p^f))^{\times p}$. Therefore there are units
$u_{i}\in(\hat R_x/(p^f))^\times$ which are not $p$th powers such that
\[
\hat{R}_{y}/(p^f)\simeq\frac{\hat{R}_{x}/(p^f)[X_{1},\dots,X_{f}]}{(X_{1}^{p}-u_{1}%
,\dots,X_{f}^{p}-u_{f})}.
\]
This implies that the $\hat{R}_{x}/(p^f)$-linear trace map $\hat{R}_{y}/(p^f)\rightarrow\hat{R}_{x}/(p^f)$ is zero: the only thing to check is the
vanishing of the trace on the identity element, which occurs since $\hat
{R}_{y}/(p^f)$ has rank $p^{f}$ as an $\hat{R}_{x}/(p^f)$-module. We deduce that also the trace map $\mathrm{Tr}_{y|x}:{R}_{y}/(p^f)\rightarrow{R}_{x}/(p^f)$ is zero, and hence the map (\ref{comp}) vanishes modulo $p^f$, as desired.
\hfill \qed

\subsection{Proof of Proposition~\ref{P:abstract hom}}
\label{S:proof of abstract hom}
Using the exact sequence
\[
0 \to \pi_2^*\dot \omega'^\kappa \xrightarrow{\cdot \varpi}\pi_2^*\dot \omega'^\kappa \longrightarrow \pi_2^*\dot \omega'^\kappa_{\FF} \to 0,
\]
we see that, if $R^t\pi_{1, *}\pi_2^*\dot \omega'^\kappa_{\FF}$ (for $t>0$) is zero when localized at a codimension $t$ point $x \in \calM_{\FF}$, then multiplication by $\varpi$ is surjective on $R^t\pi_{1, *}\pi_2^*\dot \omega'^\kappa$ when localized at $x$. By Nakayama's Lemma, this means that $R^t\pi_{1, *}\pi_2^*\dot \omega'^\kappa$ is trivial when localized at $x$, as $R^t\pi_{1, *}\pi_2^*\dot \omega'^\kappa$ is a coherent sheaf.
So the condition of this Proposition implies that
\begin{itemize}
\item
the (set-theoretical) support of $R^t\pi_{1,*}\pi_2^*\dot \omega'^\kappa$ has codimension $>t$ in the special fiber $\calM_{\FF}$ for all $t\geq 1$.
\end{itemize}

For two complexes $C^\bullet, D^\bullet \in D^b(\calM)$, we may define the sheaf of morphisms $\calHom_{D^b(\calM)}(C^\bullet, D^\bullet)$ as the (coherent) sheaf on $\calM$ whose evaluation on an open affine subspace $\calU\subseteq \calM$ is
\[
\Hom_{D^b(\calU)}\big(C^\bullet \otimes_{\calO_\calM} \calO_\calU, D^\bullet \otimes_{\calO_{\calM}} \calO_\calU\big).
\]
This the same as the $H^0$ of the derived $R\calHom_{D^b(\calM)}(C^\bullet , D^\bullet)$.
For $i \in \ZZ$, we write $\calHom^i_{D^b(\calM)}(C^\bullet ,D^\bullet)$ for $\calHom_{D^b(\calM)}(C^\bullet ,D^\bullet[i])$.

Consider the long exact sequence obtained by applying the functor $\calHom^\bullet_{D^b(\calM)}(R\pi_{1,*}\pi_2^*\dot \omega'^\kappa, -)$ to $0 \to \dot \omega^\kappa \xrightarrow{\cdot p^f} \dot \omega^\kappa \to \dot \omega^\kappa / p^f \to 0$. To show that the map $\eta$ factors uniquely as \eqref{E:sigma factor}, it is enough to show that
\begin{itemize}
\item[(a)]
(existence of factorization)
the composition $R\pi_{1,*}\pi_2^*\dot \omega'^\kappa \xrightarrow{\eta}
\dot \omega^\kappa \to \dot \omega^\kappa / p^f$ is the zero element in
$
\calH om_{D^b(\calM)}(R\pi_{1,*}\pi_2^*\dot \omega'^\kappa, \dot \omega^\kappa / p^f)
$, and
\item[(b)] (uniqueness of the factorization)
$\calH om^{-1}_{D^b(\calM)}(R\pi_{1,*}\pi_2^*\dot \omega'^\kappa, \dot \omega^\kappa / p^f) = 0$.
\end{itemize}

We claim that the natural map $\pi_{1,*}\pi_2^*\dot \omega'^\kappa \to R\pi_{1,*}\pi_2^*\dot \omega'^\kappa$ induces an injection
\begin{equation}
\label{E:inj in derived category}
\calHom^i_{D^b(\calM)}(R\pi_{1,*}\pi_2^*\dot \omega'^\kappa, \dot \omega^\kappa / p^f) \hookrightarrow \calHom^i_{D^b(\calM)}(\pi_{1,*}\pi_2^*\dot \omega'^\kappa, \dot \omega^\kappa / p^f) \cong \calHom^i_{\calO_\calM}(\pi_{1,*}\pi_2^*\dot \omega'^\kappa, \dot \omega^\kappa / p^f),
\end{equation}
for $i=-1,0$. Here $\calHom^{-1}_{\calO_\calM}(\pi_{1,*}\pi_2^*\dot \omega'^\kappa, \dot \omega^\kappa / p^f)$ is automatically zero because $\calHom$ is left exact; so (b) follows from this injectivity.
Similarly, when $i=0$, the injectivity implies that a map in $\calHom_{D^b(\calM)}(R\pi_{1,*}\pi_2^*\dot \omega'^\kappa, \dot \omega^\kappa / p^f)$ is zero if and only its induced map in $\calHom_{\calO_\calM}(\pi_{1,*}\pi_2^*\dot \omega'^\kappa, \dot \omega^\kappa / p^f)$ is zero, which is in turn equivalent to its restriction to $\calM^\ord$ is zero. So this together with the assumption of the Proposition implies (a). In summary, it is enough to prove the claim on the injectivity of \eqref{E:inj in derived category}.


Represent $R\pi_{1,*}\pi_2^*\dot \omega'^\kappa$ by a complex $C^\bullet$. For $t \in \ZZ_{\geq 0}$, we denote by $\tau_{\leq t}C^\bullet$ its truncation at degree $\leq t$, that is, the unique object in $D^b(\calM)$ (up to quasi-isomorphism) such that $H^{>t}(\tau_{\leq t}C^\bullet) = 0$, and there is a map $\tau_{\leq t}C^\bullet \to C^\bullet$ inducing isomorphisms on the cohomology groups with degree $\leq t$.
In particular, $\tau_{\leq 0}C^\bullet \cong \pi_{1,*}\pi_2^*\dot \omega'^\kappa$ and $\tau_{\leq t}C^\bullet \cong C^\bullet$ for $t \gg 0$.

To (inductively) prove the injectivity of \eqref{E:inj in derived category}, it suffices to prove that the following map is injective for every $t \geq 1$ and $i=-1,0$:
\[
\calHom^i_{D^b(\calM)}(\tau_{\leq t} C^\bullet, \dot \omega^\kappa/p^f) \hookrightarrow
\calHom^i_{D^b(\calM)}(\tau_{\leq t-1} C^\bullet, \dot \omega^\kappa/p^f).
\]
Using the long exact sequence associated to the tautological exact triangle 
\[
\tau_{\leq t-1}C^\bullet \to \tau_{\leq t} C^\bullet \to R^{t}\pi_{1,*}\pi_2^*\dot \omega'^\kappa [-t] \xrightarrow{+1},
\]
it suffices to show that
\[
\calE xt^j_{\calO_\calM}(R^{t}\pi_{1,*}\pi_2^*\dot \omega'^\kappa, \dot \omega^\kappa /p^f) = 0 \quad \textrm{for} \quad j = t-1, t.
\]
Looking at the long exact sequence induced by the following exact sequence $0 \to \dot \omega^\kappa \xrightarrow{\cdot p^f} \dot \omega^\kappa \to \dot \omega^\kappa / p^f \to 0$, it is enough to show that
\[
\calE xt^i_{\calO_\calM}(R^{t}\pi_{1,*}\pi_2^*\dot \omega'^\kappa, \dot \omega^\kappa) = 0. \quad \textrm{for} \quad i = t-1, t, t+1.
\]
Since $\calM$ is regular, the condition that the (set theoretical) support of $R^t\pi_{1,*}\pi_2^*\dot \omega'^\kappa$ has codimension $>t$ in $\calM_{\FF}$ (and hence codimension $>t+1$ in $\calM$), implies that 
\[
\calE xt^{\leq t+1}_{\calO_\calM}(R^{t}\pi_{1,*}\pi_2^*\dot \omega'^\kappa, \dot \omega^\kappa) = 0.
\]
This completes the proof of Proposition~\ref{P:abstract hom}. \hfill \qed

\section{Proof of Proposition~\ref{P:ample=>small support}}
\label{Sec:section 4}
This section is entirely devoted to proving Proposition~\ref{P:ample=>small support}. While this is the essential technical core of the construction of $T_\gothp$, readers may choose to skip it and proceed to the last section directly.
We keep the notation from the previous sections.

\subsection{Recollection of the moduli problems for $\calM$ and $\calM(\gothp)$}
\label{S:recall M(p)}
We first recall from the definition of $\calM$ in \S\ref{S:moduli HBAS} that, for the universal abelian variety $\calA$ over $\calM$, we have a universal filtration
\[
0 = \calF_{\gothp_i,j}^{(0)} \subsetneq \calF_{\gothp_i,j}^{(1)} \subsetneq \cdots \subsetneq  \calF_{\gothp_i,j}^{(e_i)} = \omega_{\calA/\calM,\gothp_i, j} \quad \textrm{for }i=1, \dots, r; \ j=1, \dots, f_i,
\]
with subquotients $\dot \omega_{\tau_{\gothp_i,j}^l}$.
For $l = 1, \dots, e_i$, we put
\[
\calH_{\gothp_i,j}^{(l)}: = \big\{ z \in ( \calF_{\gothp_i,j}^{(l-1)})^\perp / \calF_{\gothp_i,j}^{(l-1)} \ \big|\ [\varpi_i]z - \tau_{\gothp_i,j}^l(\varpi_i)z=0 \big\},
\]
where $(\calF_{\gothp_i,j}^{(l-1)})^\perp$ is the orthogonal complement of $\calF_{\gothp_i,j}^{(l-1)}$ in $H^1_\dR(\calA/\calM)_{\gothp_i,j}$ with respect to the natural pairing induced by the polarization.
By \cite[Corollary~2.10]{reduzzi-xiao}, we have
\begin{enumerate}
\item
each $\calH_{\gothp_i,j}^{(l)}$ is a locally free coherent sheaf of rank two over $\calM$,
\item
$\calF_{\gothp_i,j}^{(l)} / \calF_{\gothp_i,j}^{(l-1)}$ is a rank one subbundle of $\calH_{\gothp_i,j}^{(l)}$ (this is in fact a corollary of the previous point because $\calF_{\gothp_i,j}^{(l)} / \calF_{\gothp_i,j}^{(l-1)}$ is a subbundle of $H^1_\dR(\calA /\calM)_{\gothp_i,j} / \calF_{\gothp_i,j}^{(l-1)}$ contained in $\calH_{\gothp_i,j}^{(l)}$ since $[\varpi_i] - \tau_{\gothp_i,j}^{(l)}(\varpi_i)$ kills $\calF_{\gothp_i,j}^{(l)} / \calF_{\gothp_i,j}^{(l-1)}$),  and
\item
there is a canonical isomorphism
$
\wedge^2_{\calO_\calM} \calH_{\gothp_i,j}^{(l)} \cong \dot \epsilon_{\tau_{\gothp_i,j}^l}.
$
\end{enumerate} Moreover, by \cite[Constructions~3.3 and 3.6]{reduzzi-xiao}, the partial Hasse invariants we recalled in \S\ref{S:partial Hasse invariants} extend to \emph{surjective} homomorphisms over $\calM_\FF$
\begin{equation}
\label{E:partial Hasse}
m_{\varpi_i,j}^{(l)}: \calH_{\gothp_i,j}^{(l)}\twoheadrightarrow \calF_{\gothp_i,j}^{(l-1)} / \calF_{\gothp_i,j}^{(l-2)} \subset \calH_{\gothp_i,j}^{(l-1)} \quad \textrm{and} \quad \mathrm{Hasse}_{\varpi_i,j}: \calH^{(1)}_{\gothp_i,j} \twoheadrightarrow \dot \omega_{\tau_{\gothp_i,j-1}^{e_i}}^{\otimes p} \subset \big(\calH_{\gothp_i,j-1}^{(e_i)} \big)^{(p)},
\end{equation}
for $i = 1, \dots, r$, $j = 1, \dots, f_i$, and $l = 2, \dots, e_i$.

Now, we recall from the definition of $\calM(\gothp)$ in \S\ref{S:setup section 3} that, over $\calM(\gothp)$, we have the following two isogenies of universal abelian varieties.
\[
\phi: \calA \to \calA' \quad \textrm{and} \quad \psi: \calA' \to \calA.
\]
We write $\calF'^{(l)}_{\gothp_i,j}$ and $\calH'^{(l)}_{\gothp_i,j}$ for the corresponding constructions for $\calA'$. By the proof of Proposition~\ref{P:LCI}, we see that $\phi^*$ and $\psi^*$ induce homomorphisms
\begin{equation}
\label{E:phi psi}
\phi^*_{\gothp_i,j,l}:\calH'^{(l)}_{\gothp_i,j} \to \calH^{(l)}_{\gothp_i,j}
\quad \textrm{and} \quad
\psi^*_{\gothp_i,j,l}:\calH^{(l)}_{\gothp_i,j} \to \calH'^{(l)}_{\gothp_i,j},
\end{equation}
such that $\phi_{\gothp_i,j,l}^*$ and $\psi_{\gothp_i,j,l}^*$ are isomorphisms if $\gothp_i \neq \gothp$, and, over $\calM_\FF$, we have
\[
\mathrm{Im}(\phi^*_{\gothp,j,l}) = \Ker (\psi^*_{\gothp,j,l}) \quad \textrm{and} \quad \mathrm{Im}(\psi^*_{\gothp,j,l}) = \Ker (\phi^*_{\gothp,j,l})
\]
and both modules are subbundles of rank one of the corresponding rank two vector bundles over $\calM_\FF$.
One can check from the definition that $\phi_{\gothp_i,j,l}^*$ and $\psi_{\gothp_i,j,l}^*$ respect the partial Hasse invariant maps \eqref{E:partial Hasse}, namely we have commutative diagrams of sheaves over $\calM(\gothp)_\FF$
\[
\xymatrix@C=40pt{
\calH_{\gothp_i,j}^{(l)}\ar[r]^{\psi^*_{\gothp_i,j,l}} \ar[d]^{m_{\varpi_i,j}^{(l)}} & \calH'^{(l)}_{\gothp_i,j} \ar[r]^{\phi^*_{\gothp_i,j,l}}\ar[d]^{m_{\varpi_i,j}^{(l)}} & \calH_{\gothp_i,j}^{(l)}\ar[d]^{m_{\varpi_i,j}^{(l)}}
\\
\calH_{\gothp_i,j}^{(l-1)}\ar[r]^{\psi^*_{\gothp_i,j,l}} & \calH'^{(l-1)}_{\gothp_i,j} \ar[r]^{\phi^*_{\gothp_i,j,l}} & \calH_{\gothp_i,j}^{(l-1)}
} \textrm{and} 
\xymatrix@C=40pt{
\calH_{\gothp_i,j}^{(1)}\ar[r]^{\psi^*_{\gothp_i,j,l}}\ar[d]^{\mathrm{Hasse}_{\varpi_i,j}} & \calH'^{(1)}_{\gothp_i,j} \ar[r]^{\phi^*_{\gothp_i,j,l}}\ar[d]^{\mathrm{Hasse}_{\varpi_i,j}} & \calH_{\gothp_i,j}^{(1)}\ar[d]^{\mathrm{Hasse}_{\varpi_i,j}}
\\
(\calH_{\gothp_i,j-1}^{(e_i)})^{(p)}\ar[r]^{\psi^*_{\gothp_i,j-1,l}} & (\calH'^{(e_i)}_{\gothp_i,j-1})^{(p)} \ar[r]^{\phi^*_{\gothp_i,j-1,l}} & (\calH_{\gothp_i,j-1}^{(e_i)})^{(p)},
} 
\]
for $i =1, \dots, r$, $j =1, \dots, f_i$, and $l = 2, \dots, e_i$.

\subsection{Frobenius factor}
\label{S:Frobenius factor}
Let $k$ be a perfect field.
A map $g: Y \to X$ of $k$-schemes is called a \emph{Frobenius factor} if it induces bijection on closed points, and there are an integer $s\in \NN$ and a morphism $g' : X^{(p^s)} \to Y$ such that the composition $g\circ g'$ is the relative $p^s$-Frobenius on $X$, where $X^{(p^s)}: = X \times_{k, \Frob_{p^r}}k$.

It is proved in \cite[Proposition~4.8]{helm} that if $g: Y \to X$ is a proper morphism of $k$-schemes of finite type that induces bijection on geometric points and if $Y$ is reduced and $X$ is normal, then $g$ is a Frobenius factor.


\begin{notation}

For a subset $\ttS\subseteq \Sigma_\gothp$, we write $\ttS^c$ for $\Sigma_\gothp \backslash \ttS$. Recall that $e$ and $f$ denote the ramification and inertia degree of $\gothp$ over $p$.

Recall that, elements in $\Sigma_\gothp$ come with a chosen order, as fixed in \S\ref{S:setup}.
We consider the map $\theta: \Sigma_\gothp \to \Sigma_\gothp$ given by $\theta(\tau_{\gothp, j}^l) = \tau_{\gothp, j}^{l+1}$ if $l <e$ and $\theta(\tau_{\gothp, j}^e) = \tau_{\gothp, j-1}^1$. 
When $\gothp$ is unramified over $p$, $\theta$ is exactly the action of Frobenius $\sigma$ on $\Sigma_\gothp$.
This way, the partial Hasse invariant map at $\tau$ \eqref{E:partial Hasse} is a map
\[
\dot \omega_{\tau} \longto \dot\omega_{\theta^{-1}\tau} \textrm{ or }\dot \omega_{\theta^{-1}\tau}^{\otimes p}.
\]
We recommend the readers to assume that $\gothp$ is unramified over $p$ when reading this section first time. It helps to understand the key point of the argument.

To better imitate the unramified case,  we rename $\tau_{\gothp,j}^l$ into $\tau_{(j-1)e+l}$ and set $\tau_{a} = \tau_{a \bmod ef}$, so that we have $\theta(\tau_a) = \tau_{a+1}$.
Accordingly, we write
$\calH_{\tau_{(j-1)e+l}}$ for $\calH_{\gothp_i,j}^{(l)}$, write $\phi^*_{\tau_{(j-1)e+l}}$ for $\phi^*_{\gothp,j,l}$, write $\psi^*_{\tau_{(j-1)e+l}}$ for $\psi^*_{\gothp,j,l}$, and write $\mathrm{Ha}_{\tau_{(j-1)e+l}}$  for the corresponding partial Hasse invariant map \eqref{E:partial Hasse} on $\calA_\FF$, namely $m_{\varpi, j}^{(l)}$ if $l \neq 1$ and $\mathrm{Hasse}_{\varpi,j}$ if $l =1$.
Similarly, $\Ha'_{\tau}$ denotes the partial Hasse invariant maps between $\calH'_\tau$'s that are defined for $\calA'_\FF$.
\end{notation}

\subsection{A combinatorics construction}
\label{S:combinatorics}
For later convenience, we make explicit the following combinatorics construction:
for each $\ttS$ a nonempty proper subset of $\Sigma_\gothp$, we decompose the set $\ttS$ into the union of ``adjacent places" (according to the action of $\theta$ on $\Sigma_\gothp$), namely, we write
\begin{equation}
\label{E:ttS}
\ttS = \{\tau_{a_t-1}, \dots, \tau_{a_t-\lambda_t}, \tau_{a_{t-1}-1}, \dots, \tau_{a_{t-1}-\lambda_{t-1}}, \dots, \tau_{a_1-1}, \dots, \tau_{a_1-\lambda_1}
\},
\end{equation}
where the numbers
\begin{itemize}
\item[(1)]
$a_1, \dots, a_t \in \{1, \dots, ef\}$ are in increasing order such that $a_{i+1} - a_i \geq 2$ for all $i\geq 2$ and $a_1 +ef - a_t \geq 2$, and
\item[(2)]
$\lambda_i \in \{1, \dots, a_i- a_{i-1}-1\}$ for all $i \geq 2$ and $\lambda_1 \in \{1, \dots, a_1+ef-a_t-1\}$.
\end{itemize}
Then we have
\[\ttT:= \theta(\ttS) \backslash \ttS = \{\tau_{a_1}, \dots, \tau_{a_t}\},  \quad \textrm{and} \quad \ttS \backslash\theta(\ttS) = \{\tau_{a_1-\lambda_1}, \tau_{a_2-\lambda_2}, \dots, \tau_{a_t-\lambda_t}\}.
\]

Recall the stratification $\dot X_\ttT$ of $\calM_{\FF}$ in \S \ref{S:Goren-Oort stratification} and the stratification $\dot Y_{\ttS, \ttS'}$ of $\calM(\gothp)_\FF$ in Notation~\ref{N:stratification Iwahori}. The following is a key geometric result on the restriction of the natural map $\pi_1: \calM(\gothp)_\FF \to \calM_\FF$ when restricted to these strata. 
\begin{proposition}
\label{P:geometry of HMV(p)}
Let $\ttS$ be a subset of $\Sigma_\gothp$, and put $\ttT := \theta(\ttS) \backslash \ttS$. Then $\pi_1(\dot Y_{\ttS, \ttS^c}) = \dot X_\ttT$.
More precisely, if we write $\mathrm{pr}: \dot Y'_{\ttS, \ttS^c} \to \dot X_\ttT$ for the $(\PP^1)^{\#\ttT}$-bundle given by $\prod_{\tau \in \ttS \backslash\theta(\ttS)} \PP(\calH_\tau)$ (where the product is taken over $\dot X_\ttT$), then $\pi_1|_{\dot Y_{\ttS, \ttS^c}}$ factors as
\[
\dot Y_{\ttS, \ttS^c} \xrightarrow{\ g\ } \dot Y'_{\ttS, \ttS^c} \xrightarrow{\ \mathrm{pr} \ }\dot  X_\ttT, 
\]
where $g$ is a Frobenius factor.
\end{proposition}
\begin{proof}
The cases of $\ttS= \emptyset$ and $\ttS= \Sigma_\gothp$ correspond to the closure of the ordinary locus of $\calM(\gothp)_{\FF}$. In this case, $\pi_1$ is the Frobenius map and an isomorphism, respectively.
From now on, we assume that $\ttS$ is a nonempty proper subset of $\Sigma_\gothp$.
We may assume that $\ttS$ is as described in \S\ref{S:combinatorics} so that $\ttT= \{\tau_{a_1}, \dots, 
\tau_{a_t}\}$.

We first show that $\pi_1(\dot Y_{\ttS, \ttS^c}) \subseteq \dot X_\ttT$. Take an $S$-point of $\dot Y_{\ttS, \ttS^c}$, and we write $\scrH_\tau$ and $\scrH'_\tau$ for the evaluation of $\calH_\tau$ and $\calH'_\tau$ at this $S$-point.
For $\tau \in \ttT = \ttS \backslash \theta(\ttS)$, namely $\tau \in \ttS$ and $\theta^{-1} \tau \in \ttS^c$, consider the following commutative diagram
\begin{equation}
\label{E:phi psi hasse}
\xymatrix@C=40pt{
\scrH_\tau \ar[r]^{\psi_\tau^*} \ar[d]^{\Ha_\tau} &
\scrH'_\tau \ar[r]^{\phi_\tau^*} \ar[d]^{\Ha'_\tau} &
\scrH_\tau \ar[d]^{\Ha_\tau}
\\
(\scrH_{\theta^{-1}\tau})^{(p)} \ar[r]^{\psi_{\theta^{-1}\tau}^*} &
(\scrH'_{\theta^{-1}\tau})^{(p)} \ar[r]^{\phi_{\theta^{-1}\tau}^*} &
(\scrH_{\theta^{-1}\tau}
)^{(p)}.
}
\end{equation}
Here and later, if $\tau = \tau_{a}$ with $a \not\equiv 1 \bmod e$, we loose all the Frobenius twists on the modules at $\theta^{-1}\tau$.
The condition $\theta^{-1}\tau \in \ttS^c$ implies that $\Ker(\phi^*_{\theta^{-1}\tau}) = \dot \omega_{A',\theta^{-1}\tau}^{\otimes p}$ (which would be $\dot \omega_{A',\theta^{-1}\tau}$ if $\tau = \tau_a$ with $a \not\equiv 1 \bmod e$, as said above).
Since the image of $\Ha'_{\tau}: \scrH'_\tau \to  \scrH'^{(p)}_{\theta^{-1}\tau}$ is exactly $\dot \omega_{A',\theta^{-1}\tau}^{\otimes p}$, we see that the composition $\phi^*_{\theta^{-1}\tau} \circ \Ha'_\tau$ is the zero map, so is the composition $\Ha_\tau \circ \phi_\tau^*$.
In particular, this means that $\mathrm{Im}(\phi_\tau^*) \subseteq \Ker(\Ha_\tau)$.
But $\tau \in \ttS$ implies that $\dot \omega_{A,\tau} \subseteq \Ker(\psi_\tau^*) = \mathrm{Im}(\phi_\tau^*)$. This shows that $\Ha_\tau(\dot \omega_{A,\tau}) = 0$, meaning that the image of this $S$-point under $\pi_1$ lies in $\dot X_\tau$ (for our chosen $\tau \in \ttT$). Applying this to each $\tau \in \ttT$ shows that $\pi_1(\dot Y_{\ttS, \ttS^c}) \subseteq \dot X_\ttT$.

We now construct the map morphism $g$ that factors the morphism $\pi_1$. Given an $S$-point $y$ of $\dot Y_{\ttS, \ttS^c}$, $\pi_1(y)$ belongs to $\dot X_\ttT$ as shown above.
For each $\tau \in \ttS \backslash \theta(\ttS)  = \{\tau_{a_1-\lambda_1}, \tau_{a_2-\lambda_2}, \dots, \tau_{a_r-\lambda_r}\}$,
the image
$
\phi_\tau^*(\scrH'_\tau) \subseteq \scrH_\tau
$
defines a rank one subbundle of the latter.
So this lifts the map $\pi_1$ to a map 
\[
g: \dot Y_{\ttS, \ttS^c} \longrightarrow  \PP(\calH_{\tau_{a_1-\lambda_1}}) \times_{\dot X_\ttT} \cdots \times_{\dot X_\ttT} \PP(\calH_{\tau_{a_r-\lambda_r}}) =:\dot Y'_{\ttS, \ttS^c}.
\]

Since $\dot Y_{\ttS, \ttS^c}$ is reduced and $\dot Y'_{\ttS, \ttS^c}$ is normal (as both are smooth over $\FF$), we may use the criterion of Frobenius factors as recalled in \S\ref{S:Frobenius factor}. To prove the Proposition, it suffices to show that $g$ induces a bijection on $k$-points for any algebraically closed field $k$ containing $\FF$.

We start with a $k$-point $\dot x = (A = \calA_{\dot x}, \lambda, \alpha, \underline \scrF = \underline \calF_{\dot x})$ of $\dot X_\ttT \cap \calM_\gothc$, with $k$ algebraically closed.
We write $\scrH_{\tau} $ for $\calH_{\tau, {\dot x}}$, and we ignore the Frobenius twist on each $\scrH_\tau$ ($\tau \in \Sigma_\gothp$), as it is a two-dimensional vector space over an algebraically closed field.

\underline{\bf Claim:} Giving a $k$-point $y$ in $\dot Y_{\ttS, \ttS^c} \cap \pi_1^{-1}({\dot x})$ is equivalent to specifying, for each $\tau \in \Sigma_\gothp$, a one-dimensional subspace $M_\tau \subseteq \scrH_\tau$ (which will be the image $\phi_\tau^*(\scrH'_\tau)$), such that
\begin{itemize}
\item[(i)] $
\Ha_\tau(M_\tau) \subseteq M_{\theta^{-1}\tau}$, and 
\item[(ii)]
for each $\tau \in \ttS^c$, $M_\tau  = \dot \omega_{A, \tau}$, and
\item[(iii)]
for each $\tau \in \theta(\ttS)$, $M_{\tau} = \Ker(\Ha_{\tau})$.
\end{itemize}
We temporarily assume this Claim and deduce the Proposition.
Note that, for $\tau \in \ttS^c \cup \theta(\ttS)$, $M_\tau$ is already uniquely determined by (ii) and (iii). Actually, for $\tau \in \ttS^c \cap \theta(\ttS) = \ttT$, $M_\tau$ is ``over-determined" as required to be equal to both $\dot \omega_{A, \tau}$ by (ii) and  $\Ker(\Ha_\tau)$ by (iii), but these two subspaces of $\scrH_\tau$ are equal as the Hasse invariant $\dot h_\tau$ vanishes at ${\dot x}$.
The only unspecified $M_\tau$'s are those with $\tau \in \ttS \backslash \theta(\ttS)$, or explicitly with $\tau \in \{ \tau_{a_r-\lambda_r}, \tau_{a_{r-1}- \lambda_{r-1}}, \dots, \tau_{a_1-\lambda_1}\}$ in terms of the combinatorics of \S\ref{S:combinatorics}. 
Their choices are equivalent to specifying a point in $Y'_{\ttS, \ttS^c}$ over $\dot x$.
To conclude, we observe that (i) holds for any such choice of $M_\tau$.
\begin{itemize}
\item
If $\tau \in \theta(\ttS^c)$, $\Ha_\tau(M_\tau) \subseteq \mathrm{Im}(\Ha_\tau) = \dot \omega_{A, \theta^{-1}\tau} = M_{\theta^{-1}\tau}$ by (ii).
\item
If $\tau \in \theta(\ttS)$, $\Ha_\tau(M_\tau) = 0$ by (iii).
\end{itemize}
This completes the proof of the Proposition assuming the Claim.

Now we turn to the proof of the Claim.
First, given a point $\dot y = ((A, \lambda, \alpha, \underline \scrF); (A', \lambda',\alpha', \underline \scrF'); \phi, \psi)   $ in $\dot Y_{\ttS, \ttS^c} \cap \pi_1^{-1}({\dot x})$, we set $M_\tau : = \phi_\tau^*(\scrH'_\tau)$. Since $\dot y \in \dot Y_{\ttS, \ttS^c}$, we deduce (ii) and (iii) as follows, which further imply (i) formally as shown above.
\begin{itemize}
\item
for each $\tau \in \ttS^c$, $M_\tau: = \mathrm{Im}(\phi^*_\tau) = \mathrm{Ker}(\psi^*_\tau) = \dot \omega_{A, \tau}$, and
\item
for each $\tau \in \theta(\ttS)$, $\dot \omega_{A,\theta^{-1}\tau}= \Ker(\phi_{\theta^{-1}\tau}^*)$; so $\phi_{\theta^{-1}\tau}^* \circ \Ha'_{ \tau} = 0$, or equivalently $\Ha_{\tau} \circ \phi^*_{\tau} =0$. From this we see that $M_{\tau} := \mathrm{Im}(\phi^*_\tau) = \Ker(\Ha_\tau)$.

\end{itemize}

Conversely, given $M_\tau$ for $\tau \in \Sigma_\gothp$ satisfying (i)--(iii), let $\tilde \calD(A)$ denote the contravariant Dieudonn\'e module, which decomposes into the direct sum
\[
\tilde \calD(A) = \bigoplus_{i=1}^r \bigoplus_{j=1}^{f_i} \tilde \calD(A)_{\gothp_i,j}
\]
according to the $\calO_F$-action, so that each summand $\tilde \calD(A)_{\gothp_i,j}$ is a free module of rank two over $\calO_{F_{\gothp_i}} \otimes_{W(\FF_{\gothp_i}), \tau_{\gothp_i,j}} W(k)=:\calO_{F_{\gothp_i},k}$ (which is a complete discrete valuation ring). In particular, $\tilde \calD(A)_{\gothp_i,j}/p \cong H^1_\dR(A)_{\gothp_i,j}$ is a $k[x_i]/(x_i^{e_i})$-module free of rank two, where the $x_i$ acts by $[\varpi_i]$. 
In particular, we write $x$ for the action of $[\varpi]$ on $\tilde \calD(A)_{\gothp,j}$.
We write 
\[
\mu_{\gothp_i,j}: \tilde \calD(A)_{\gothp_i,j} \to \tilde \calD(A)_{\gothp_i,j}/p \cong H^1_\dR(A)_{\gothp_i,j}\] for the naturally induced map.
For each $l =0, 1, \dots, e$, we write $\tilde \scrF_{\gothp,j}^{(l)}$ for the preimage of $\scrF_{\gothp,j}^{(l)}$ under the map $\mu_{\gothp,j}$, and for $l=1, \dots, e$, write $\tilde \scrH_{\gothp,j}^{(l)}: = x^{-1} \tilde \scrF_{\gothp,j}^{(l)}$ so that its image under \[
\mu_{\gothp,j}^{(l)}:
\tilde \calD(A)_{\gothp,j}  \xrightarrow {\mu_{\gothp,j}} H^1_\dR(A)_{\gothp,j} \to  H^1_\dR(A)_{\gothp,j} / \scrF_{\gothp,j}^{(l-1)}
\]
is exactly $\scrH_{\gothp,j}^{(l)}$. In particular, $\tilde \scrH_{\gothp,j}^{(l)}$ contains $\tilde \scrF_{\gothp,j}^{(l)} = x \tilde \scrH_{\gothp,j}^{(l-1)}$ as $\calO_{F_{\gothp},k}$-modules with colength $1$, and $\tilde \scrH_{\gothp,j}^{(1)} = x^{-1} p\tilde \calD(A)_{\gothp,j}$. Here and later, we say an inclusion $M \subseteq N$ of $\calO_{F_\gothp,k}$-modules have \emph{colength $i$} if $N/M$ is a successive extension of $i$ copies of $\calO_{F_\gothp,k} / (\varpi) \cong k$ as an $\calO_{F_{\gothp},k}$-module.

Now for each $l = 1, \dots, e$, we define 
$\tilde \scrH'^{(l)}_{\gothp,j}$ to be the preimage of $M_\tau \subseteq \scrH_\tau$ for $\tau = \tau_{\gothp_i,j}^{l}$ under the above map $\mu_{\gothp,j}^{(l)}$; it is
an $\calO_{F_{\gothp},k}$-submodule  of $\tilde \scrH^{(l)}_{\gothp,j}$ of colength $1$. Consider the $\calO_F$-stable submodule $\tilde M = \oplus_{i=1}^r\oplus_{j=1}^{f_i} \tilde M_{\gothp_i,j}$ of $\tilde \calD(A)$ with
\[
\tilde M_{\gothp_i,j} := \begin{cases}
\tilde \calD(A)_{\gothp_i,j} & \textrm{if }\gothp_i \neq \gothp,
\\ p^{-1}x\tilde \scrH'^{(1)} _{\gothp,j}& \textrm{if }\gothp_i = \gothp.
\end{cases}
\]
We shall show that $\tilde M$ is a Dieudonn\'e submodule of $\tilde \calD(A)$, namely, $p\tilde M \subseteq V\tilde M \subseteq \tilde M$.
This is clear for the summands with $\gothp_i \neq \gothp$. At $\gothp$, in fact, we shall prove the following stronger statements.
\begin{itemize}
\item[(a)]
For each $j = 1, \dots, f$ and $l =1, \dots, e-1$, we have an inclusion $x \tilde  \scrH'^{(l+1)}_{\gothp,j}\subset \tilde  \scrH'^{(l)}_{\gothp,j}$ of colength $1$ (which in particular implies that $\tilde  \scrH'^{(l)}_{\gothp,j} \subset \tilde  \scrH'^{(l+1)}_{\gothp,j}$).
\item[(b)]
For each $j=1, \dots, f$, we have an inclusion $x^{1-e}V(\tilde \scrH'^{(1)}_{\gothp,j}) \subset \tilde \scrH'^{(e)}_{\gothp,j-1}$ of colength $1$, which in particular implies that $\tilde \scrH'^{(e)}_{\gothp,j-1} \subset p^{-1} V(\tilde \scrH'^{(1)}_{\gothp,j})$ (note that $x^e$ and $p$ defines the same ideal in $\calO_{F_\gothp,k}$).
\end{itemize}
Provided (a) and (b), we deduce the following inclusion
\[
V \tilde M_{\gothp,j} = V(p^{-1}x \tilde \scrH'^{(1)}_{\gothp,j}) \subset  \tilde \scrH'^{(e)}_{\gothp,j-1} \subset x^{-1}  \tilde \scrH'^{(e-1)} _{\gothp,j-1}\subset \cdots \subset x^{1-e}  \tilde \scrH'^{(1)}_{\gothp,j-1}= \tilde M_{\gothp,j-1},
\]
which
has total colength $e$ (accumulating $1$ from each inclusion). So in particular, $p \tilde M_{\gothp,j-1}$ is also contained in $V\tilde M_{\gothp,j}$.
We now check (a) and (b). For (a), we first note that the  multiplication by $x$ map: $\cdot x: \tilde \scrH_{\gothp, j}^{(l+1)} \to \tilde \scrH_{\gothp,j}^{(l)}$ is injective and its reduction modulo $x$ is exactly the partial Hasse invariant $\Ha_{\tau_{\gothp,j}^{l+1}}: \scrH_{\gothp, j}^{(l+1)} \to \scrH_{\gothp,j}^{(l)}$. The fact $\Ha_{\tau_{\gothp,j}^{l+1}} (M_{\tau_{\gothp,j}^{l+1}}) \subseteq M_ {\tau_{\gothp,j}^{l}}$ from (i) immediately implies that $\cdot x$ sends $\tilde \scrH'^{(l+1)}_{\gothp,j}$ into $\tilde \scrH'^{(l)}_{\gothp,j}$, where the cokernel can be easily computed to have length $1$.
To see (b), we first claim that the map $x^{1-e}V$ takes $\tilde \scrH^{(1)}_{\gothp,j}$ into $\tilde \scrH^{(e)}_{\gothp,j-1}$. But this is clear, as \[
x^{1-e}V(\tilde \scrH^{(1)}_{\gothp,j})  = 
x^{1-e}V(x^{-1}p\tilde \calD(A)_{\gothp,j}) = \tilde \omega_{A, \gothp,j-1},
\]
and it is an $\calO_{F_\gothp,k}$-submodule of $\tilde \scrH^{(e)}_{\gothp,j-1}$ of colength $1$.
Moreover, inspecting the construction of $\mathrm{Hasse}_{\varpi,j}$ recalled in \S\ref{S:partial Hasse invariants}, we see that the reduction of the map $x^{1-e}V$ modulo $x$ is exactly $\mathrm{Hasse}_{\varpi,j}$.
The condition (i) $\mathrm{Hasse}_{\varpi,j}(M_{\tau_{\gothp,j}^1}) \subseteq M_{\tau_{\gothp,j-1}^{e}}$ in turn implies that $x^{1-e}V$ takes $\tilde \scrH'^{(1)}_{\gothp,j}$ into $ \tilde \scrH'^{(e)}_{\gothp,j-1}$, and the cokernel of this map is one-dimensional over $k$. This completes checking (a) and (b).

Now by standard Dieudonn\'e theory, the inclusions $p\tilde M \subseteq V\tilde M \subseteq \tilde M$ we just proved above implies that there exists an abelian variety $A'$ together with $\calO_F$-equivariant isogenies
$\phi:A \to A'$ such that the induced map on Dieudonn\'e modules $\phi^*: \tilde \calD(A') \to \tilde \calD(A)$ can be identified with the natural inclusion $\tilde M \subseteq \tilde \calD(A)$.
By construction $\tilde \calD(A) / \tilde M$ is a free module of rank one over $\FF_\gothp \otimes_{\FF_p} k$ as opposed to having rank one over $\FF_\gothp[x]/(x^e) \otimes_{\FF_p}k$. (Note that the latter condition would give rise to an isogeny such that $\phi\circ \psi$ is multiplication by $p$, but not ``multiplication by the ideal $\gothp$".) We see that there exists a (dual) isogeny $\psi: A' \to A$ and a polarization $\lambda': A'^\vee \to A'  \otimes_{\calO_F} \gothc'$ satisfying condition (3)(a)--(d) in \S\ref{S:setup section 3}.
The tame level structure $i$ on $A$ naturally propagates to $A'$. So it suffices to define a filtration $\underline \scrF'$ on $\omega_{A'}$ satisfying (3)(e) of \S\ref{S:setup section 3} and check that the point defined lies on $\dot Y_{\ttS, \ttS^c}$.
If $\gothp_i \neq \gothp$, $\phi^*$ and $\psi^*$ induce isomorphisms between $H^1_\dR(A)_{\gothp_i,j}$ and $H^1_\dR(A')_{\gothp_i,j}$; this forces $\scrF'^{(l)}_{\gothp_i,j} = \phi^*(\scrF^{(l)}_{\gothp_i,j})$.
We now consider the case  at $\gothp$. For $j =1, \dots, f$ and $l = 1, \dots, e-1$, we set 
\[
\scrF'^{(l)}_{\gothp,j}: = x\tilde \scrH'^{(l+1)}_{{\gothp,j}} / p\tilde \calD(A')_{\gothp,j}\quad \textrm{and} \quad \scrF'^{(e)}_{\gothp,j}: = V\tilde \calD(A')_{\gothp,j+1} / p\tilde \calD(A')_{\gothp,j} = \omega_{A', \gothp,j};
\]
they are subspaces of $H^1_\dR(A')_{\gothp,j}$.
Note that by (a) above, we have an inclusion $\scrF'^{(l)}_{{\gothp,j}} \subseteq \scrF'^{(l+1)}_{\gothp,j}$  of colength $1$, for $l = 1, \dots, e-1$.
Similarly, by (b) above, we have an inclusion
\[
x \tilde \scrH'^{(e)}_{\gothp,j} \subset xp^{-1} V(\tilde \scrH'^{(1)}_{\gothp,j}) = xp^{-1} V(px^{-1} \tilde M_{\gothp,j}) = V\tilde \calD(A')_{\gothp,j+1},
\]
which has colength $1$. In other words, we have an inclusion $\scrF'^{(e-1)}_{\gothp,j} \subset \scrF'^{(e)}_{\gothp,j}$ of colength $1$.
All these imply that $\scrF'^{(l)}_{\gothp,j}$'s define the needed filtration $\underline \scrF$ on $\omega_{A'}$.
Analogous to the situation on $A$, for $j=1, \dots, f$ and $l = 1, \dots, e$, we set $\dot \omega'_{\tau^l_{\gothp,j}}: = \scrF'^{(l)}_{\gothp,j} / \scrF'^{(l-1)}_{\gothp,j}$.

Finally, we check that the point we constructed belongs to $\dot Y_{\ttS, \ttS^c}$, For this, it is enough to check the following.
\begin{itemize}
\item For $\tau \in \ttS^c$,  $\dot \omega_\tau = \Ker(\psi^*_\tau) = \mathrm{Im}(\phi^*_\tau)$, which follows from the construction and condition (ii).
\item
For $\tau \in \ttS$, $\dot \omega'_\tau = \Ker(\phi^*_\tau)$, which is equivalent to $\Ha_{\theta\tau} \circ \phi^*_{\theta\tau} = 0$, which further follows from the construction and condition (iii).
\end{itemize}

This concludes the verification of the Claim and hence completes the proof of the Proposition.
\end{proof}

\begin{corollary}
\label{C:property of pi1}
Let $\ttT$ be a subset of $\Sigma_\gothp$ with $t = \#\ttT$.
\begin{enumerate}
\item
Over the open stratum $\dot X_\ttT^\circ$, the fiber dimension of $\pi_1$ is less than or equal to $t$.
The equality holds only when $\ttT$ is \emph{sparse} (this is called \emph{spaced} in \cite{goren-kassaei}), namely, $\tau$ and $\theta \tau$ does not belong to $\ttT$ simultaneously for any $\tau \in \Sigma_\gothp$.

\item
Suppose that $\ttT$ is a sparse subset of $\Sigma_\gothp$, then the $t$-dimensional fibers of $\pi_1^{-1} (\dot X_\ttT)$ are the (not necessarily disjoint) union of $\dot Y_{\ttS, \ttS^c}$ such that $\ttT = \theta(\ttS) \backslash \ttS$.
\item
We write $\ttS_{\underline \lambda}$ for the subset $\ttS$ given in \eqref{E:ttS} as determined by $\ttT$ and a tuple $\underline \lambda = (\lambda_1, \dots, \lambda_t)$. Then for any two tuples $\underline \lambda, \underline \lambda'$, if $|\lambda'_i - \lambda_i| \geq 2$ for some $i \in \{1, \dots, r\}$, then
$\pi_1(\dot Y_{\ttS_{\underline \lambda}, \ttS^c_{\underline \lambda}} \cap \dot Y_{\ttS_{\underline \lambda'}, \ttS^c_{\underline \lambda'}})$ is disjoint from the open stratum $\dot   X_\ttT^\circ$.
\item
Let $\underline \lambda$ and $\lambda'$ be two tuples such that $\lambda'_i = \lambda_i$ for $i \neq i_0$ and $\lambda'_{i_0} = \lambda_{i_0} +1$. Then the ideal sheaf defined by the inclusion $\dot Y_{\ttS_{\underline \lambda}, \ttS^c_{\underline \lambda}} \cap \dot Y_{\ttS_{\underline \lambda'}, \ttS^c_{\underline \lambda'}} \hookrightarrow \dot Y_{\ttS_{\underline \lambda}, \ttS^c_{\underline \lambda}}$ is
\[
\dot \omega_{\calA', \tau_{a_i-\lambda'_i}}  \otimes \dot \omega_{\calA, \tau_{a_i-\lambda'_i}}^{-1}.
\]
\end{enumerate}
\end{corollary}
\begin{proof}
(1)
By Proposition~\ref{P:geometry of HMV(p)}, each irreducible component $\dot Y_{\ttS, \ttS^c}$ is a fiber bundle of pure dimension $\#(\theta(\ttS)\backslash \ttS)$ over $\dot X_{\theta(\ttS)\backslash \ttS}$. The base is the union of all open strata $\dot X_{\ttT'}^\circ$ with $\ttT' \supseteq \theta(\ttS)\backslash \ttS$. In particular, the codimension of $\dot X_{\ttT'}^\circ$ is $\#\ttT'$ which is greater than or equal to the fiber dimension $\#\theta(\ttS)\backslash \ttS$ and the equality holds exactly when $\ttT' = \theta(\ttS)\backslash \ttS$ (which implies that $\ttT'$ is sparse). Since this holds true for all irreducible components, (1) is clear.

(2) This is an immediate corollary of (1).

(3) We know that $\dot Y_{\ttS_{\underline \lambda}, \ttS^c_{\underline \lambda}} \cap \dot Y_{\ttS_{\underline \lambda'}, \ttS^c_{\underline \lambda'}} = \dot Y_{\ttS_{\underline \lambda}\cup \ttS_{\underline \lambda'}, \ttS^c_{\underline \lambda}\cup \ttS^c_{\underline \lambda'}}.$ by Lemma~\ref{C:strata iwahori}(2). If $\lambda'_i - \lambda_i \geq 2$ for some $i$, then this intersection is contained in $\dot Y_{\ttS', \ttS'^c}$ for 
\[
\ttS' = \ttS_{\lambda} \cup \{ \tau_{a_i- \lambda_i-2}\}.
\]
But then $\theta(\ttS') \backslash \ttS' = \ttT\cup \{\tau_{a_i- \lambda_i-1}\}$. So we have
\[
\pi_1(\dot Y_{\ttS_{\underline \lambda}, \ttS^c_{\underline \lambda}} \cap\dot  Y_{\ttS_{\underline \lambda'}, \ttS^c_{\underline \lambda'}}) = \pi_1( \dot Y_{\ttS_{\underline \lambda}\cup \ttS_{\underline \lambda'}, \ttS^c_{\underline \lambda}\cup \ttS^c_{\underline \lambda'}}) \subseteq \pi_1(\dot Y_{\ttS', \ttS'^c}) \subseteq\dot  X_{\ttT \cup \{\tau_{a_i-\lambda_i-1}\}}.
\]
In particular, $\pi_1(\dot Y_{\ttS_{\underline \lambda}, \ttS^c_{\underline \lambda}} \cap \dot Y_{\ttS_{\underline \lambda'}, \ttS^c_{\underline \lambda'}})$ does not intersect with $\dot X_\ttT^\circ$.

(4) Note that the condition implies that $\ttS_{\underline \lambda'} = \ttS_{\underline \lambda} \cup \{\tau_{a_i - \lambda'_i}\}$.
So $\dot Y_{\ttS_{\underline \lambda}, \ttS^c_{\underline \lambda}} \cap \dot Y_{\ttS_{\underline \lambda'}, \ttS^c_{\underline \lambda'}} = \dot Y_{\ttS_{\underline \lambda'}, \ttS^c_{\underline \lambda}}$ is a closed subscheme in $\dot Y_{\ttS_{\underline \lambda}, \ttS^c_{\underline \lambda}}$ defined as the vanishing locus of
\[
\phi_{\tau_{a_i-\lambda'_i}}^*: \dot \omega_{\calA', \tau_{a_i-\lambda'_i}} \longrightarrow \dot \omega_{\calA, \tau_{a_i-\lambda'_i}}.
\]
The statement of (4) follows, because the ideal sheaf of the vanishing locus of a section of a line bundle is the inverse line bundle.
\end{proof}

Before proceeding, we need some additional geometric information regarding the map $\pi_1: \dot Y_{\ttS, \ttS^c} \to \dot X_\ttT$.

\begin{proposition}
\label{P:description of line bundles}
Let $\ttS$ be a subset of $\Sigma_p$ and put $\ttT: = \theta(\ttS) \backslash \ttS$.
\begin{enumerate}
\item
We have the following isomorphisms of line bundles on $\dot Y_{\ttS, \ttS^c}$:
\begin{equation}
\label{E:pull back line bundle description 1}
\textrm{if }\tau, \theta^{-1}\tau \in \ttS, \quad \dot\omega_{\calA', \tau} \cong 
\begin{cases}
\dot\omega^{\otimes p}_{\calA, \theta^{-1}\tau} & \textrm{if }\tau = \tau_a \textrm{ with }a \equiv 1 \bmod e,
\\
\dot\omega_{\calA, \theta^{-1}\tau} & \textrm{otherwise},
\end{cases}
\end{equation}
\begin{equation}
\label{E:pull back line bundle description 2}
\textrm{if }\tau, \theta^{-1}\tau \in \ttS^c, \quad \dot\omega_{\calA, \tau} \cong 
\begin{cases}
\dot\omega^{\otimes p}_{\calA', \theta^{-1}\tau} & \textrm{if }\tau = \tau_a \textrm{ with }a \equiv 1 \bmod e,
\\
\dot\omega_{\calA', \theta^{-1}\tau} & \textrm{otherwise}.
\end{cases}
\end{equation}
\item
For each $\tau \in \ttS \backslash \theta(\ttS)$, we write 
$\calO_\tau(-1)$ for the canonical sub line bundle on $\dot Y'_{\ttS, \ttS^c}$ at the $\PP^1$-factor indexed by $\tau$. Then we have
\[
\dot\omega_{\calA',\tau_{a_i-\lambda_i}} \cong g^* \calO_{\tau_{a_i-\lambda_i}}(1), \quad \textrm{and} \quad g^* \calO_{\tau_{a_i-\lambda_i}}(-1) \cong \begin{cases}
\dot \omega_{\calA',\tau_{a_i-\lambda_i-1}}^{\otimes p} & \textrm{if }a_i- \lambda_i\equiv 1 \bmod e
\\
\dot \omega_{\calA',\tau_{a_i-\lambda_i-1}} & \textrm{otherwise}.
\end{cases}.
\]
\end{enumerate}
\end{proposition}
\begin{proof}
(1) We shall prove \eqref{E:pull back line bundle description 2} and the proof of \eqref{E:pull back line bundle description 1} is similar. So we assume that $\tau, \theta^{-1}\tau \in \ttS^c$.
For simplicity, we assume that $\tau = \tau_a$ with $a \equiv 1 \mod e$; the argument for the other case is similar by loosing all the Frobenius twists in the proof (and hence getting $\dot \omega_{\calA', \theta^{-1}\tau}$ as opposed to  $\dot \omega_{\calA', \theta^{-1}\tau}^{\otimes p}$ on the right hand side of \eqref{E:pull back line bundle description 2}).
Take an $S$-point of $\dot Y_{\ttS, \ttS^c}$; we look at the commutative diagram \eqref{E:phi psi hasse} which we copy to below
\begin{equation}
\label{E:phi psi Hasse copy}
\xymatrix@C=40pt{
\scrH_\tau \ar[r]^{\psi_\tau^*} \ar[d]^{\Ha_\tau} &
\scrH'_\tau \ar[r]^{\phi_\tau^*} \ar[d]^{\Ha'_\tau} &
\scrH_\tau \ar[d]^{\Ha_\tau}
\\
(\scrH_{\theta^{-1}\tau})^{(p)} \ar[r]^{\psi_{\theta^{-1}\tau}^*} &
(\scrH'_{\theta^{-1}\tau})^{(p)} \ar[r]^{\phi_{\theta^{-1}\tau}^*} &
(\scrH_{\theta^{-1}\tau}
)^{(p)}.
}
\end{equation}
Since $\tau, \theta^{-1}\tau \in \ttS^c$, we have
\[
\dot \omega_{A,\tau} \cong \mathrm{Ker}(\psi_{\tau}^*: \scrH_\tau \to  \scrH'_\tau ) =  \mathrm{Im}(\phi_{\tau}^*:  \scrH'_\tau \to  \scrH_\tau ).
\]
Note that $\dot \omega_{A', \theta^{-1}\tau}^{\otimes p}$ is also the image of $\scrH'_{\tau}$ but under the map $\Ha'_\tau$.
To prove the desired isomorphism, it suffices to show that
\begin{equation}
\label{E:ker V = im psi}
\mathrm{Ker}(\Ha'_{\tau}) = \mathrm{Ker}(\phi_\tau^*) = \mathrm{Im}(\psi_\tau^*).
\end{equation}
Since both sides are subbundle of $\scrH'_\tau $ of rank one, it suffices to show that $\Ha'_ \tau \circ \psi_\tau^* = 0$, which is equivalent to show that $\psi^*_{\theta^{-1}\tau}\circ \Ha_ \tau=0$.
But the image of $\Ha_\tau$ is exactly $\dot \omega_{A, \theta^{-1}\tau}^{\otimes p}$ which lies in the kernel of $\psi^*_{\theta^{-1}\tau}$ by the assumption $\theta^{-1}\tau \in \ttS^c$.
So we conclude \eqref{E:ker V = im psi} and hence prove (1).

(2) The first equality follows from the equality
\begin{align*}
\wedge^2\calH'_{\tau_{a_i-\lambda_i}} &\cong \mathrm{Ker} \phi_{\tau_{a_i-\lambda_i}}^* \otimes \mathrm{Im} \phi_{\tau_{a_i-\lambda_i}}^* \cong \dot \omega_{\calA',\tau_{a_i-\lambda_i}} \otimes g^*\calO_{\tau_{a_i-\lambda_i}}(-1),
\end{align*}
because the left hand side $\wedge^2\calH'_{\tau_{a_i-\lambda_i}}$ can be canonically trivialized over $\calM(\gothp)$ by \S\ref{S:recall M(p)}(3) and \cite[Lemma~2.5]{reduzzi-xiao} (through pulling back along $\pi_2$). 

For the second equality, we shall only prove it when $a_i-\lambda_i \equiv 1 \bmod e$; the other case is similar but without the additional Frobenius pullbacks.
We take an $S$-point of $\dot Y_{\ttS, \ttS^c}$ as above; we may look at \eqref{E:phi psi Hasse copy} for $\tau= \tau_{a_i-\lambda_i}$.
We note that $\dot \omega_{A',\tau_{a_i-\lambda_i-1}}^{\otimes p}$ is the image of $\scrH'_{\tau_{a_i-\lambda_i}}$ under $\Ha'_{ \tau_{a_i-\lambda_i}}$, and $g^*\calO_{\tau_{a_i-\lambda_i}}(-1)$ is the image of $\scrH'_{\tau_{a_i-\lambda_i}}$ under $\phi_{\tau_{a_i-\lambda_i}}^*$. So it suffices to prove that
\[
\mathrm{Ker}( \Ha'_{ \tau_{a_i-\lambda_i}}) \cong \mathrm{Ker}(\phi_{\tau_{a_i-\lambda_i}}^*) = \mathrm{Im} (\psi_{\tau_{a_i-\lambda_i}}^*).
\]
Similar to the argument in (1), for rank reasons, it suffices to show that 
\[
\Ha'_{ \tau_{a_i-\lambda_i}} \circ \psi_{\tau_{a_i-\lambda_i}}^*: \scrH_{\tau_{a_i-\lambda_i}} \longrightarrow (\scrH'_{\tau_{a_i-\lambda_i-1}})^{(p)}
\]
is the zero map.
But this follows from that $\psi_{\tau_{a_i-\lambda_i-1}}^*(\dot \omega_{A,\tau_{a_i-\lambda_i-1}}^{(p)}) =0$ because $\tau_{a_i-\lambda_i-1} \in \ttS^c$.
\end{proof}

The following corollary of Grothendieck's formal function theorem will reduce Proposition~\ref{P:ample=>small support} to a calculation at the component described in Corollary~\ref{C:property of pi1}(2).
\begin{proposition}
\label{P:formal function}
Let $h: X \to Y$ be a projective morphism between noetherian schemes, and let $t = \max\{ \dim X_y| y \in Y\}$. Let $\calF$ be a coherent sheaf on $X$.

(1) Then $R^ih_*(\calF) =0$ for all $i>t$.

(2) Suppose that $X$ is the union of two components $X_1 \cup X_2$, such that $\max\{ \dim X_{2, y} | y \in Y\} < t$. Then we have
\[
R^th_*(\calF) \cong R^th_*(\calF|_{X_1}).
\]

\end{proposition}
\begin{proof}
(1) is a corollary of Grothendieck's formal function theorem; see  e.g. [Hartshorne, Cor 11.2].

(2) Write $i: X_1 \to X$ for the natural inclusion.
Let $\calG$ denote the kernel of the surjective morphism $\calF \to i_*\calF|_{X_1}$; then $\calG$ is supported on $X_2$.
By (1), $R^rh_*(\calG) = R^{r+1}h_*(\calG) = 0$. So we proved (2).
\end{proof}

\subsection{Proof of Proposition~\ref{P:ample=>small support}}
\label{S:proof of main prop}
We are now ready to prove Proposition~\ref{P:ample=>small support}.

By Corollary~\ref{C:property of pi1}(1), for every point $\dot x$ of $\calM$ of codimension $t$, $\dim \pi_1^{-1}(x) \leq t$ and the equality holds only when $\dot x$ is a generic point of $\dot X_\ttT^\circ$ for some sparse set $\ttT \subseteq \Sigma_\gothp$ with $t = \# \ttT$.
By  Proposition~\ref{P:formal function}(1), the stalk of $R^{\geq t}\pi_{1, *}\pi_2^*\dot \omega'^\kappa_\FF$ is zero at all other points (of codimension $t$).
So to prove Proposition~\ref{P:ample=>small support}, it suffices to show that, for each sparse set $\ttT\subseteq \Sigma_\gothp$ with $t = \#\ttT$, $R^t\pi_{1,*}\pi_2^*\dot \omega'^\kappa_{\FF}$ vanishes on every \emph{geometric} generic point $\eta_\ttT$ of $\dot X_\ttT^\circ$.
By Corollary~\ref{C:property of pi1}(2), the $t$-dimension fibers of $\dot X_\ttT$ are exactly those of $\dot Y_{\ttS, \ttS^c}$ for which $\theta(\ttS) \backslash \ttS= \ttT$.
Write $\dot Z_\ttT$ for the union of these $\dot Y_{\ttS, \ttS^c}$ (with the reduced scheme structure).
Using Proposition~\ref{P:formal function}(2), we see that
\begin{equation}
\label{E:isom at eta T}
\big(R^r\pi_{1,*}\pi_2^* \dot \omega'^\kappa_{\FF}\big)_{\eta_\ttT} \cong \big( R^r\pi_{1,*} (\pi_2^* \dot  \omega'^\kappa)|_{\dot Z_\ttT} \big)_{\eta_\ttT}.
\end{equation}
The proof of Proposition~\ref{P:ample=>small support} is then reduced to prove the vanishing of \eqref{E:isom at eta T} for each non-empty sparse set $\ttT \subseteq \Sigma_\gothp$.
Here and after, we shall frequently write $(-)_{\eta_\ttT}$ to indicate the base change to the point $\eta_\ttT$.
We shall  prove Proposition~\ref{P:ample=>small support} in the following two steps:
\begin{enumerate}
\item
Let $(\dot Y_{\ttS, \ttS^c})_{\eta_\ttT}^\mathrm{red}$ denote the reduced subscheme of $(\dot Y_{\ttS, \ttS^c})_{\eta_\ttT}$.
We shall show that the natural map $g^\mathrm{red}_{\eta_\ttT}: (\dot Y_{\ttS, \ttS^c})_{\eta_\ttT}^\mathrm{red} \to (\dot Y'_{\ttS, \ttS^c})_{\eta_\ttT} \cong (\PP^1)^t_{\eta_{\ttT}}$ defined below in \eqref{E:gred} is the $p$-Frobenius in the factor labeled by $\tau$ for which $\tau = \tau_a$ with $a \equiv 1 \bmod e$; so in particular, $(\dot Y_{\ttS, \ttS^c})_{\eta_\ttT}^\mathrm{red}$ itself is isomorphic to $(\PP^1)^t_{\eta_{\ttT}}$.
\item
We shall prove that
$(\pi_2^*\dot \omega'^\kappa)_{(\dot Z_\ttT)_{\eta_\ttT}}$ is a successive extension of line bundles $L_\ttS$ supported on each $(\dot Y_{\ttS, \ttS^c})_{\eta_\ttT}$, and $L_\ttS|_{(\dot Y_{\ttS, \ttS^c})_{\eta_\ttT}^\mathrm{red}}$ is the external tensor product of line bundles on $\PP^1_{\eta_\ttT}$ of the form $\calO(n)$ with $n \geq -1$ (assuming our conditions on weights in Proposition~\ref{P:ample=>small support}).
\end{enumerate}

We start with (1). Note that the map $g: \dot Y_{\ttS, \ttS^c} \to \dot Y'_{\ttS, \ttS^c}$ is a Frobenius factor, so the base change $(\dot Y_{\ttS, \ttS^c})_{\eta_\ttT}$ to the \emph{geometric} generic point may not be reduced; we write $(\dot Y_{\ttS, \ttS^c})_{\eta_\ttT}^\mathrm{red}$ for its reduced subscheme.
Then the base change of $g$ over to $\eta_\ttT$, denoted by $g_{\eta_\ttT}$, gives a Frobenius factor (over the residue field $\kappa_\ttT$ at $\eta_\ttT$):
\begin{equation}
\label{E:gred}
g_{\eta_\ttT}^\mathrm{red}: (\dot Y_{\ttS, \ttS^c})_{\eta_\ttT}^\mathrm{red} \to (\dot Y'_{\ttS, \ttS^c})_{\eta_\ttT}.
\end{equation}
We claim that this is in fact the $p$-Frobenius in the factor labeled by $\tau$ for which $\tau = \tau_a$ with $a \equiv 1 \bmod e$, and an isomorphism on other factors.

To easy the presentation, we may extend both $\calM$ and $\calM(\gothp)$ from over $\calO$ to over the completion of maximal unramified extension of $\calO$. This way, all closed points of $\calM$ and $\calM(\gothp)$ are defined over $\overline \FF_p$. 
We follow the proof of Proposition~\ref{P:LCI}
to take a small enough Zariski open neighborhood $\calU \subset \calM$ of $\eta_\ttT$ (in the integral model) and then take a small enough Zariski open subset $\calV \subset \pi_1^{-1}(\calU)\subset \calM(\gothp)$ intersecting the fiber $\dot Y_{\ttS, \ttS^c}$, such that the tuple 
\[
(\calH_\tau|_\calV, \calH'_\tau|_\calV, \phi_\tau^*,\psi_\tau^*)_{\tau \in \Sigma_\gothp}\quad \textrm{
is isomorphic to }
\quad
\big(\calO_\calV^{\oplus 2}, \calO_\calV^{\oplus 2}, \big( \begin{smallmatrix}
1 & 0 \\ 0 & \tau(\varpi)
\end{smallmatrix} \big),\big( \begin{smallmatrix}
\tau(\varpi) & 0 \\ 0 & 1
\end{smallmatrix} \big) \big)_{\tau \in \Sigma_\gothp}.
\]
Let $\ttF$ denote the moduli problem of rank one $\calO_\calV$-subbundle $M_\tau \subseteq \calO_\calV^{\oplus 2}$ for each $\tau \in \Sigma_p \backslash \ttT$ corresponding to the subbundle $\dot \omega_{\calA, \tau}|_\calV \subset \calH_\tau|_\calV$.
Let $\ttG$ denote the moduli problem of rank one $\calO_\calV$-subbundles $M_\tau \subset \calO_\calV^{\oplus 2}$ for each $\tau \in \ttS$ corresponding to the subbundles $\dot \omega_{ \calA,\tau}|_\calV \subset \calH_\tau|_\calV$ and rank one subbundle $M'_\tau \subset \calO_\calV^{\oplus 2}$ for each $\tau \in \ttS^c$ corresponding to the subbundle $\dot \omega_{\calA',\tau}|_\calV \subseteq \calH'_\tau|_\calV$.
The theory of local model says that $\dot X_{\ttT}$ (resp.  $\dot Y_{\ttS, \ttS^c}$) is \'etale locally isomorphic to $\ttF_{\overline \FF_p}$ (resp. $\ttG_{\overline \FF_p}$).
The local parameters of $\ttF_{\overline \FF_p}$ (at a point) are $u_\tau$ for $\tau \in \Sigma_p \backslash \ttT$ which measures the position $\dot \omega_{\calA, \tau}|_\calV \subset \calH_\tau|_\calV$. In particular, the completion of $\dot X_\ttT$ at a closed $\overline \FF_p$-point $x$ is isomorphic to $\overline \FF_p\llbracket (u_\tau)_{\tau \in \Sigma_\gothp \backslash \ttT}\rrbracket$.
The local parameters of $\ttG_{\overline \FF_p}$ (at a point) are $u_\tau$ for $\tau \in \ttS$ which measures the position of $\dot \omega_{\calA, \tau}|_\calV \subset \calH_\tau|_\calV$, and $v_\tau$ for $\tau \in \ttS^c$ which measures the position of $\dot \omega_{\calA', \tau}|_\calV \subset \calH'_\tau|_\calV$.
In particular, the completion of $\dot Y_{\ttS, \ttS^c}$ at a closed $\overline \FF_p$-point $y \in \pi_1^{-1}(x) \cap \dot Y_{\ttS, \ttS^c}$ is isomorphic to $\overline \FF_p\llbracket (u_\tau)_{\tau \in \ttS}, (v_\tau)_{\tau \in \ttS^c}\rrbracket$.
Note that we can use the same notation $u_\tau$ for local parameters on $\ttF_{\overline \FF_p}$ and on $\ttG_{\overline \FF_p}$ because in  the homomorphism 
\begin{equation}
\label{E:map between local models}
\calO_{\calU, x}^\wedge \cong \overline \FF_p\llbracket (u_\tau)_{\tau \in \Sigma_p \backslash \ttT}\rrbracket \longrightarrow 
\calO_{\calV, y}^\wedge \cong  \overline \FF_p\llbracket (u_\tau)_{\tau \in \ttS}, (v_\tau)_{\tau \in \ttS^c}\rrbracket
\end{equation}
on the completions induced by $\pi_1$, one may choose the local parameters in a compatible way so that $u_\tau$ for $\tau \in \ttS$ is taken to $u_\tau$.

To understand the image of $u_\tau$ for $\tau \in (\Sigma_\gothp \backslash \ttT)\backslash \ttS= \ttS^c \cap \sigma(\ttS^c)$, we consider a variant of the argument of Proposition~\ref{P:description of line bundles}(1).
If $\tau, \theta^{-1}\tau \in \ttS^c$ and if $\tau = \tau_a$ with $a \equiv 1 \bmod e$, the proof of Proposition~\ref{P:description of line bundles}(1) implies that $\Ker (\Ha'_\tau) = \mathrm{Ker} (\phi^*_{\tau, \overline \FF_p})$. So we may choose an isomorphism $\zeta_\tau: \calH_{\tau, \overline \FF_p} \cong \calH'^{(p)}_{\theta^{-1}\tau, \overline \FF_p}$ such that $\Ha'_\tau$ is the same as $\zeta_\tau \circ \phi_{\tau, \overline \FF_p}^*$. Under this identification, we have
\[
\zeta_\tau (\dot \omega_{\calA, \tau, \overline \FF_p} ) = \zeta_\tau(\mathrm{Im}(\phi^*_{\tau, \overline \FF_p})) = \mathrm{Im}(\Ha'_\tau) = \dot \omega_{\calA', \theta^{-1}\tau, \overline \FF_p}^{\otimes p}.
\]
So we see that we can rearrange the choices of local parameters so that the local parameter $u_{\tau}$ (for $\tau \in \ttS^c \cap \theta(\ttS^c)$ and $\tau = \tau_a$ with $a \equiv 1 \bmod e$) is sent to $v_{\theta^{-1}\tau}^p$ under the map \eqref{E:map between local models}.
The same argument shows that, when $\tau \in \ttS^c \cap \theta(\ttS^c)$ and $\tau = \tau_a$ with $a \not\equiv 1 \bmod e$, we can rearrange the choices of local parameters so that $u_\tau$ is sent to $v_{\theta^{-1}\tau}$.

Using this, we see that the completion at a closed point $y_{\eta_\ttT}$ of $(\dot Y_{\ttS, \ttS^c})_{\eta_\ttT}$ is isomorphic to
\[
\overline \FF_p((u_\tau)_{\tau \in \Sigma_\gothp \backslash \ttT})^\mathrm{alg} \otimes_{\overline \FF_p[(u_\tau)_{\tau \in \Sigma_\gothp \backslash \ttT}]} \overline \FF_p[(u_\tau)_{\tau \in \ttS}, (v_{\theta^{-1}\tau})_{\tau \in \ttS^c \cap \theta(\ttS^c)} ]\llbracket (v_{\tau})_{\tau \in \theta^{-1}(\ttS)\cap \ttS^c}\rrbracket.
\]
Using the identification of $u_{\tau}$ with $v_{\theta^{-1}\tau}^p$ (resp.  $v_{\theta^{-1}\tau}$) for $\tau \in \ttS^c \cap \theta(\ttS^c)$ with $\tau =\tau_a$ for $a \equiv 1\bmod e$ (resp. $a \not \equiv 1 \bmod e$), we see that the completion of $(\dot Y_{\ttS, \ttS^c})_{\eta_\ttT}^\mathrm{red}$
at a closed point $y_{\eta_\ttT}$ is isomorphic to 
\[
\kappa_\ttT\llbracket (v_\tau)_{\tau \in \theta^{-1}(\ttS)\cap \ttS^c}\rrbracket.
\]
Here we recall that $\kappa_\ttT$ is the residue field of $\eta_\ttT$ and $v_\tau$ is the coordinate for the subbundle $\dot \omega_{\calA', \tau} \subseteq H'_\tau$.

On the other hand, the completion of $(\dot Y'_{\ttS, \ttS^c})_{\eta_\ttT}$ at $f_{\eta_\ttT}(y_{\eta_\ttT})$ is isomorphic to $\kappa_\ttT\llbracket (v'_\tau)_{\tau \in \ttS\backslash \theta(\ttS)} \rrbracket$, where $v'_\tau$ is the coordinate of the chosen subbundle of $\calH_\tau$ in the definition of $\dot Y'_{\ttS, \ttS^c}$.
We need to show that, up to adjusting the local parameter $v'_\tau$,
\begin{equation}
\label{E:frobenius on P1}\textrm{ for every }\tau \in \ttS\backslash \theta(\ttS), \quad
f^*_{\eta_\ttT}(v'_\tau) = \begin{cases} v_{\theta^{-1}\tau}^p & \textrm{if }\tau = \tau_a\textrm{ with }a \equiv 1 \bmod e;\\
v_{\theta^{-1}\tau} & \textrm{if }\tau = \tau_a\textrm{ with }a \not\equiv 1 \bmod e.
\end{cases}
\end{equation}
For this, we fix one such $\tau \in \ttS\backslash \theta(\ttS)$. We assume that $\tau = \tau_a$ with $a \equiv 1 \bmod e$ (and the other case can be proved in the same way by removing all the Frobenius twists).
Following exactly the same argument as above, we start by noticing that $\Ker(\Ha'_\tau ) = \Ker(\phi^*_{\tau, \overline \FF_p})$.
So we may choose an isomorphism $\zeta_\tau: \calH_{\tau, \overline \FF_p} \cong \calH'^{(p)}_{\theta^{-1}\tau, \overline \FF_p}$ such that $\Ha'_\tau$ is the same as $\zeta_\tau \circ \phi_{\tau, \overline \FF_p}^*$. Under this identification, we have
\[
\zeta_\tau(\mathrm{Im}(\phi^*_{\tau, \overline \FF_p})) = \mathrm{Im}(\Ha'_\tau) = \dot \omega_{\calA', \theta^{-1}\tau, \overline \FF_p}^{\otimes p}.
\]
So it follows that, up to adjusting the local parameter, \eqref{E:frobenius on P1} holds. Since $g_{\eta_\ttT}^\mathrm{red}$ is already a Frobenius factor, it must take the form as described in (1).\\

Now we may identify $(\dot Y_{\ttS, \ttS^c})_{\eta_\ttT}^\mathrm{red}$ with $(\PP^1_{\eta_{\ttT}})^t$.
Write $\calO_i(1)$ for the canonical quotient bundle from the $i$th factor. In particular, $g_{\eta_\ttT}^*\calO_{\tau_{a_i-\lambda_i}}(1) $ is equal to $ \calO_i(p)$ if $a_i -\lambda_i \equiv 1 \bmod e$ and to $\calO_i(1)$ otherwise.
As a corollary of this and Proposition~\ref{P:description of line bundles}, we have
\begin{equation}
\label{E:omega restricted to P1}
\dot \omega_{\calA', \tau}|_{(\dot Y_{\ttS, \ttS^c})_{\eta_\ttT}^\mathrm{red}}
\cong \begin{cases}
\calO_i(p) & \textrm{if } \tau = \tau_{a_i-\lambda_i} \textrm{ and }a_i-\lambda_i \equiv 1 \bmod e,
\\
\calO_i(1) & \textrm{if } \tau = \tau_{a_i-\lambda_i}\textrm{ and }a_i-\lambda_i \not\equiv 1 \bmod e,
\\
\calO_i(-1) & \textrm{if } \tau = \tau_{a_i-\lambda_i-1},
\\
\calO_{(\dot Y_{\ttS, \ttS^c})_{\eta_\ttT}^\mathrm{red}} 
& \textrm{otherwise}.
\end{cases}
\end{equation}

We now turn to (2). Corollary~\ref{C:property of pi1}(3) and (4) explained the intersection relation among $\dot Y_{\ttS_{\underline \lambda}, \ttS^c_{\underline \lambda}}$'s.
Put $s_i = a_i - a_{i-1} - 1$ for $i \geq 2$ and $s_1 = a_1+ef-a_t-1$.
For example when $t = 2$, the following diagram shows the intersection relation, where two irreducible components of $\dot Z_\ttT$ intersect in codimension $1$ if they are linked by a line, and they intersect in codimension $2$ if they are at the opposite vertices of a square:
\[
\xymatrix@R=10pt@C=10pt{
\dot Y_{\ttS_{1, 1}, \ttS_{1,1}^c} \ar@{-}[r]
\ar@{-}[d] & \dot Y_{\ttS_{1, 2}, \ttS_{1,2}^c} \ar@{-}[r]
\ar@{-}[d] & \cdots \ar@{-}[r] & \dot Y_{\ttS_{1, s_2}, \ttS_{1,s_2}^c} \ar@{-}[d]
\\
\dot Y_{\ttS_{2, 1}, \ttS_{2,1}^c} \ar@{-}[r]
\ar@{-}[d] & \dot Y_{\ttS_{2, 2}, \ttS_{2,2}^c} \ar@{-}[r]
\ar@{-}[d] & \cdots \ar@{-}[r] & \dot Y_{\ttS_{2, s_2}, \ttS_{2,s_2}^c} \ar@{-}[d]
\\
\vdots \ar@{-}[d] & \vdots \ar@{-}[d] & \ddots & \vdots \ar@{-}[d]
\\
\dot Y_{\ttS_{s_1, 1}, \ttS_{s_1,1}^c} \ar@{-}[r] & \dot Y_{\ttS_{s_1, 2}, \ttS_{s_1,2}^c} \ar@{-}[r] & \cdots \ar@{-}[r] & \dot Y_{\ttS_{s_1, s_2}, \ttS_{s_1,s_2}^c}. 
}
\]

Moreover, these $\dot Y_{\ttS_{\underline\lambda}, \ttS^c_{\underline \lambda}}$'s have proper intersections by the proof of Proposition~\ref{P:LCI}. So by Corollary~\ref{C:property of pi1}(4), $\dot\omega^\kappa|_{\dot Z_\ttT}$ is the successive extension of
\begin{equation}
\label{E:subquotients}
\dot\omega^{\kappa}|_{\dot Y_{\ttS_{\underline\lambda}, \ttS^c_{\underline \lambda}}} \otimes \bigotimes_{i=1: \lambda_i \neq s_i}^t \big(
\omega_{\calA', \tau_{a_i-\lambda_i-1}}  \otimes\omega_{\calA, \tau_{a_i-\lambda_i-1}}^{-1}
\big) \quad \textrm{for all }\lambda_i \in \{1, \dots, s_i\},
\end{equation}
in which the term with $\underline \lambda = \underline 1$ is the subobject and the term with $\underline \lambda = (s_i)_{i =1, \dots, t}$ is the quotient object.
Restricting this to the $(\PP^1_{\eta_\ttT})^t$-bundle $(\dot Y'_{\ttS_{\underline\lambda}, \ttS_{\underline\lambda}^c})_{\eta_\ttT}^\mathrm{red}$, this is equal to
\[
\bigotimes_{i=1}^t \begin{cases} \calO_i(pk_{a_i-\lambda_i}-k_{a_i-\lambda_i-1}) & \textrm{if } \lambda_i = s_i \textrm{ and }a_i-\lambda_i \equiv 1 \bmod e ,
\\
\calO_i(k_{a_i-\lambda_i}-k_{a_i-\lambda_i-1}) & \textrm{if } \lambda_i = s_i \textrm{ and }a_i-\lambda_i \not\equiv 1 \bmod e ,
\\
\calO_i(pk_{a_i-\lambda_i}-k_{a_i-\lambda_i-1}-1) & \textrm{if } \lambda_i \neq s_i \textrm{ and }a_i-\lambda_i \equiv 1 \bmod e ,
\\
\calO_i(k_{a_i-\lambda_i}-k_{a_i-\lambda_i-1}-1) & \textrm{if } \lambda_i \neq s_i \textrm{ and }a_i-\lambda_i \not\equiv 1 \bmod e .
\end{cases}
\]
By the assumption of Proposition~\ref{P:ample=>small support}, the numbers in the parentheses of the right hand side are always $\geq 1$. 
Since $H^1(\PP^1, \calO(n)) = 0$ for $n \geq -1$, we see that
\[
R^{>0}\pi_{1,*}\Big(\textrm{\eqref{E:subquotients}}\big|_{(\dot Y'_{\ttS_{\underline\lambda}, \ttS_{\underline\lambda}^c})_{\eta_\ttT}^\mathrm{red}} \Big)= 0.
\]
It then follows that
\begin{equation}
\label{E: Rt vanishes red}
R^t\pi_{1,*} \big((\pi_2^*  \dot \omega'^\kappa)|_{(\dot Z_\ttT)_{\eta_\ttT}^\mathrm{red}} \big) = 0.
\end{equation}

To prove the needed vanishing of \eqref{E:isom at eta T} and hence Proposition~\ref{P:ample=>small support}, we observe that, due the cohomological dimension, by Proposition~\ref{P:formal function}(1), $R^t\pi_{1,*}(-)$ is a \emph{right exact} functor on sheaves set-theoretically supported on $(\dot Z_\ttT)_{\eta_\ttT}^\mathrm{red}$, and it is trivial on any coherent sheaf set-theoretically supported in dimension $<t$ subspace of $(\dot Z_\ttT)_{\eta_\ttT}^\mathrm{red}$.
We show below that, by (a variant of) Nakayama lemma, this implies that
\[
R^t\pi_{1,*} \big((\pi_2^* \dot  \omega'^\kappa)|_{(\dot Z_\ttT)_{\eta_\ttT}} \big) = 0.
\]
Indeed, write $\calF$ for $(\pi_2^* \dot  \omega'^\kappa)|_{(\dot Z_\ttT)_{\eta_\ttT}}$. If $\calI$ is the ideal sheaf of $(\dot Z_\ttT)_{\eta_\ttT}^\mathrm{red}$ in $(\dot Z_\ttT)_{\eta_\ttT}$, it is enough to show that
\[
R^t\pi_{1,*}(\calI^i\calF/ \calI^{i+1}\calF) = 0 \quad \textrm{for every }i \geq0.
\]
But $\calI^i/\calI^{i+1} \otimes \calF \twoheadrightarrow \calI^i\calF/ \calI^{i+1}\calF$. By the right exactness of $R^t{\pi_{1,*}}(-)$, it suffices to show the vanishing of $R^t\pi_{1,*}(\calI^i/ \calI^{i+1} \otimes \calF)$. But $\calI^i/\calI^{i+1} \otimes \calF$ is (scheme-theoretically) supported on $(\dot Z_\ttT)^\mathrm{red}_{\eta_\ttT}$ and it receives generic surjective maps from finite direct sums of $\calF|_{(\dot Z_\ttT)^\mathrm{red}_{\eta_\ttT}}$ (for example induced by local generators of $\calI^i$). By the properties of $R^t\pi_{1,*}(-)$ recalled above and the vanishing result \eqref{E: Rt vanishes red}, we deduce that $R^t\pi_{1,*}(\calI^i/ \calI^{i+1} \otimes \calF)=0$. This concludes Proposition~\ref{P:ample=>small support}. \hfill \qed

\section{Results on the unramifiedness of modular representations in weight ${\bf 1}$}
\label{Sec:section 5}

Recall that $\calO$ denotes the ring of integers in a large enough finite extension $E$ of $\QQ_p$, with uniformizer $\varpi$ and residue field $\FF$. For simplicity, we assume for the entirety of this section that the prime $p$ is \emph{inert} in $F$, so that the Hecke operator $T_\gothp$ will be denoted by $T_p$. We denote by $\epsilon:G_F\to\calO^\times$ the $p$-adic cyclotomic character of $G_F$, and by $\epsilon_m$ its reduction modulo $\varpi^m$.

Recall that $\Sh$ denotes the Hilbert modular Shimura scheme, smooth over $\Spec\calO$, of tame level $K^p$ satisfying Hypothesis~\ref{H:Kp neat}. For any positive integer $m$, denote by $\Sh_m$ the base change of $\Sh\rightarrow\Spec\calO$ to $\Spec(\calO/(\varpi^m))$, and write similarly $\Sh_m^\tor,\omega^\kappa_m$, etc. 

We assume throughout this section that $p$ is odd.

\subsection{Shimura varieties with auxiliary level structures} 
\label{S:Sh(Q)01}
We follow \cite{CG} for most of the notation and constructions of this section. Recall that $\calS$ denotes the finite set of places including the archimedean places,  $p$-adic places, and all the places $\gothq$ where $K_\gothq \neq \GL_2(\calO_{F_\gothq})$ (cf. \S \ref{S:geometric HMF}). Let $\calQ$ denote a finite set of finite places of $F$ disjoint from $\calS$. (We will fix later suitable sets $\calQ$ consisting of Taylor--Wiles primes; for now, $\calQ=\emptyset$ is allowed). With abuse of notation, we will often use the letters $\calS$ and $\calQ$ also to denote the ideals of $\calO_F$ determined by the ``product'' of the finite places in $\calS$ and $\calQ$, respectively.

Denote by $\Sh(\calQ)$ (resp. $\Sh(\calQ)_1$) the Shimura scheme over $\Spec\calO$ of tame level 
\begin{equation}
\label{E:Q level}
K^p(\calQ): =\left\{ \begin{pmatrix}
a&b\\c& d
\end{pmatrix} \in K^p; c \equiv 0 \bmod \calQ\right\} \quad \textrm{and}\quad 
K^p(\calQ)_1: =\left\{ \begin{pmatrix}
a&b\\c& d
\end{pmatrix} \in K^p(\calQ); d \equiv a \bmod \calQ\right\},
\end{equation}
respectively.
For each $\gothc \in \gothC$, there is a natural \'etale morphism $\calM_\gothc(\calQ)_1 \to \calM_\gothc(\calQ) \to \calM_\gothc$, where the first map has Galois group $(\calO_F/\calQ\calO_F)^\times\cong \prod_{\gothq\in \calQ}(\calO_F/\gothq)^\times$.
Note that we have equalities
\begin{equation}
\label{same quotient group}
K^pK_p \cap \calO_F^\times = 
K^p(\calQ)K_p \cap \calO_F^\times = K^p(\calQ)_1K_p \cap \calO_F^\times.
\end{equation}
So when passing to the quotient by this group (as in the beginning of \S\ref{S:integral model Sh var}) and summing over all $\gothc \in \gothC$, we obtain a natural \'etale cover
\begin{equation}
\label{E:cover Sh1Q to ShQ to Sh}
\Sh(\calQ)_1 \to \Sh(\calQ) \to \Sh,
\end{equation}
where the first map has Galois group $(\calO_F / \calQ\calO_F)^\times$.\footnote{In an earlier version of this paper, we used another tame level structure that causes further complication. We thank the anonymous referee for pointing out this.}
Let $(\calO_F/\calQ\calO_F)^\times\twoheadrightarrow\Delta$ be a quotient map and denote by $\calM_\gothc(\calQ)_\Delta$ (resp.  $\Sh(\calQ)_{\Delta}$) the corresponding subcover over $\calM_\gothc(\calQ)$ (resp.  $\Sh(\calQ)$) with Galois group $\Delta$. 

We now explain the extension of \eqref{E:cover Sh1Q to ShQ to Sh} to the toroidal compactification.  Using what we have recalled on the toroidal compactifications in \S \ref{S:toroidal}, we see that over each cusp of $\calM_\gothc$ labeled by $\calC= (\gotha, \gothb, L, i, j, \lambda, \alpha)$, the cusps of $\calM_\gothc(\calQ)$ are labeled by subsets $\calR \subseteq \calQ$:
\[
\calC_\calR = (\gotha, \gothb, L, i, j, \lambda, \alpha_\calR)
\]
where $\alpha_\calR$ is a $K^p(\calQ) / K(N\calQ)^p$-orbit of isomorphisms
\[
\alpha \oplus \bigoplus_{\gothq \in \calQ} \alpha_{\calR, \gothq}: (\calO_F/N\calO_F)^{\oplus 2} \oplus \bigoplus_{\gothq \in \calQ} (\calO_F/ \gothq \calO_F)^{\oplus 2} \xrightarrow \sim N^{-1}L / L \oplus \bigoplus_{\gothq \in \calQ} \gothq^{-1} L / L,
\]
where $\alpha_{\calR, \gothq}$ is given by the matrix $\big(\begin{smallmatrix}
1&0\\0&1
\end{smallmatrix} \big)$ if $\gothq \in \calR$, and is given by the matrix  $\big(\begin{smallmatrix}
0&1\\1&0
\end{smallmatrix} \big)$ if $\gothq \notin \calR$.
Rigorously speaking, to literally apply \S \ref{S:toroidal}, we need to use principal level structure  with \emph{integer} levels. But we can easily modify the above definition by introducing a positive prime-to-$pN$ integer $Q$ that is divisible by $\calQ$, and then take $K(N\calQ)^p / K(NQ)^p$ invariants.

We say that this cusp $\calC_{\calR}$ of $\calM_\gothc(\calQ)$ is \emph{unramified} at $\calR\subseteq \calQ$.
In terms of the recipe in \S \ref{S:toroidal}, we have $X_{[\gamma_{\calC_{\calR}}]} = \frac{\calR}{\calQ}X_{[\gamma_\calC]}$ (in particular, if $\calR= \calQ$, the map $\calM_\gothc(\calQ) \to \calM_\gothc$ is an isomorphism at that cusp if we use the pullback cone decomposition.)
In general, we give an $\calO_F^{\times, +}$-stable smooth admissible cone decomposition of $X^{*+}_{[\gamma_{\calC_\calR}]}$ at each cusp $\calC_\calR$ that refines the restriction of the cone decomposition of 
$X^{*+}_{[\gamma_{\calC}]}$. This way we obtain a morphism $\calM_\gothc(\calQ)^\tor \to \calM_\gothc^\tor$ and then a morphism $\Sh(\calQ)^\tor \to \Sh^\tor$ that extends the second map in \eqref{E:cover Sh1Q to ShQ to Sh}.

For the map $\calM_\gothc(\calQ)_\Delta \to \calM_\gothc(\calQ)$, over each such cusp $\calC_\calR$ of $\calM_\gothc(\calQ)$, the cusps of $\calM_\gothc(\calQ)_\Delta$ are parametrized by $\Delta$. Precisely speaking, for each $\delta \in \Delta$, there is a cusp with label 
\[
\calC_\calR = (\gotha, \gothb, L, i, j, \lambda, \alpha_{\calR, \delta}),
\]
where $\alpha_{\calR, \delta} = \alpha \oplus \bigoplus_{\gothq \in \calQ} 
\big(\alpha_{\calR, \gothq} \cdot \big(\begin{smallmatrix}
\tilde \delta &0\\0&1
\end{smallmatrix}\big)\big)$, where $\tilde \delta$ is a lift of $\delta$ for the quotient map $(\calO_F/Q\calO_F)^\times\twoheadrightarrow\Delta$. We note that $X_{[\gamma_{\calC_{\calR, \delta }}]} \cong X_{[\gamma_{\calC_\calR}]}$. So we may pullback the cone decomposition on $X_{[\gamma_{\calC_\calR}]}^{*+}$ to a (smooth admissible) cone decomposition on $X_{[\gamma_{\calC_{\calR, \delta}}]}^{*+}$. For the rest of this paper, we shall always take the cone decomposition on $\calM_\gothc(\calQ)_\Delta$ this way. Therefore, we have natural \emph{\'etale} covering maps
$$\calM_\gothc(\calQ)_\Delta^\tor \twoheadrightarrow \calM_{\gothc}(\calQ)^\tor \quad \textrm{and}\quad \Sh(\calQ)_{\Delta}^\tor\twoheadrightarrow\Sh(\calQ)^\tor$$
with Galois group $\Delta$.

\subsection{Hecke algebras with Taylor--Wiles primes} 
\label{S:Hecke TW primes}
For a finite set of places $\calQ$ as above (allowing $\calQ = \emptyset$) and a choice of quotient $(\calO_F/\calQ\calO_F)^\times \twoheadrightarrow \Delta$, we define the \emph{abstract tame Hecke algebra} to be
\[
\TT_\calQ^\univ: = \calO[\Delta][t_\gothq; \gothq \notin \calS\cup \calQ][ s_\gothq^\rmn;\gothq \textrm{ finite}].
\]
Let $\TT_{\calQ}^{\{0\}}$ denote the image of the abstract tame Hecke algebra acting on $\oplus_{m\geq 1}H^0(\Sh(\calQ)_{\Delta,m}^\tor,\omega^{({\bf 1}, -1)}_m)$, by sending $t_\gothq \mapsto T_\gothq$ ($\gothq \notin \calS \cup \calQ$), $s_\gothq^\rmn \mapsto S_\gothq^\rmn$, and $[a]\mapsto $ the Diamond operator $\langle \tilde a \rangle$ for $\tilde a \in (\calO_F/\calQ\calO_F)^\times$ lifting $a \in \Delta$.  We remind the reader again that, due to our normalization of tame Hecke operators in \S\ref{S:tame Hecke operators}, weight $({\bf n}, n-2))$ is the parallel weight $n$ in many other literatures. 

Let $\bar \rho: G_F \to \GL_2(\FF)$ be an absolute irreducible representation (which we do not assume to be unramified at $p$ at this moment). Let $\gothm'_\emptyset$ denote the maximal ideal of $\TT_\emptyset^{\{0\}}$, generated by 
\begin{equation}
\label{E:generators of mrho}
\varpi, \quad t_\gothq - \mathrm{tr}(\bar{\rho}(\Frob_\gothq))\ (\gothq \notin \calS), \quad \textrm{and} \quad  s_\gothq^\rmn - \det(\bar \rho(\Frob_\gothq)) \textrm{ (for all finite }\gothq),
\end{equation}
where $\det(\bar \rho(\Frob_\gothq))$ is independent of the choice of the Frobenius elements at $\gothq$.
We assume that $\calQ$ satisfies the following additional conditions:
\begin{itemize}
\item for each $\gothq\in \calQ$ we have $\NN(\gothq) \equiv 1 \mod p$,
\item for each $\gothq\in \calQ$ the polynomial $X^2-T_\gothq X + S_\gothq^\rmn\in\TT_\emptyset^{\{0\}}[X]$ has distinct roots modulo $\gothm_\emptyset$; we choose for each $\gothq\in \calQ$ one such root $\alpha_\gothq\in\FF$ (and enlarging the field $\FF$ if necessary), and
\item
The group $\Delta$ is a $p$-group.

\end{itemize} 

Let $\gothm'_\calQ$ denote the maximal ideal of $\TT_\calQ^{\{0\}}$ containing the generators \eqref{E:generators of mrho} and the elements $U_\gothq-\alpha_\gothq$ for $\gothq\in \calQ$. 
It follows from the main theorem of \cite{emerton-reduzzi-xiao} that there is a Galois representation $\rho_\calQ:G_F\to\GL_2(\TT^{\{0\}}_{\calQ,\gothm'_\calQ})$ lifting $\bar\rho$, unramified outside $\calS\cup \calQ$, and such that $\mathrm{tr}(\rho_\calQ(\Frob_\gothq))=T_\gothq$ and $\det(\rho_\calQ(\Frob_\gothq)) = S_\gothq^\rmn$ for all $\gothq\notin \calS \cup \calQ$. 
\begin{proposition}\label{P:nr}
Assume that for any lift $\Frob_p\in G_F$ of the arithmetic Frobenius at $p$, the eigenvalues of $\bar\rho(\Frob_p)$ in $\overline\FF$ are distinct. Then there exists a unique deformation
$$\rho_\calQ:G_F\to\GL_2(\TT_{\calQ,\gothm'_\calQ}^{\{0\}})$$
of $\bar\rho$ unramified outside $\calS \cup \calQ$ and such that for all primes $\gothq\notin \calS \cup \calQ$ we have $\mathrm{tr}(\rho_\calQ(\Frob_\gothq))=T_\gothq$ and $\det (\rho_\calQ(\Frob_\gothq)) = S_\gothq^\rmn$. In particular, $\rho_\calQ$ is unramified at $p$.
\end{proposition}

\begin{proof}
Recall that we did not insist that $\bar \rho$ to be unramified at $p$. If $\bar \rho$ is ramified at $p$, the representation $\rho_\calQ$ we obtain is the zero representation, i.e. $\TT_{\calQ, \gothm'_\calQ}^{\{0\}} =0$.

Thanks to the existence and properties of the operator $T_p^\rmn$ acting on weight one forms (cf. \S\ref{S:def of Tp}), we can prove this Proposition \emph{exactly} as in \cite[Theorem 3.11]{CG}, where the case $F=\QQ$ is treated (see also \cite{deligne-serre} and \cite[Proposition 2.7]{edixhoven}). For completeness, we sketch the argument below, but the reader is referred to \emph{loc.cit.} for further details.

Let $M$ be a positive integer divisible by $p^{m-1}$ and denote by $\tilde{h}_M\in H^0(\Sh(\calQ)_{\Delta,m}^\tor,\omega_m^{(M({\bf p}-{\bf 1}),M(p-1))})$ a lift to $\calO/(\varpi^m)$ of the $M$th power of the total Hasse invariant $h\in H^0(\Sh(\calQ)_{\Delta,\FF}^\tor,\omega_\FF^{({\bf p}-{\bf 1},p-1)})$ (cf. \cite[\S3.3.1]{emerton-reduzzi-xiao}). Let $U_p^\rmn$ denote the action of the Hecke operator $T_p^\rmn$ on modular forms of paritious weight $({\bf n},n-2):=({\bf 1}+M({\bf p}-{\bf 1}),M(p-1)-1)$ over $\calO/(\varpi^m)$.
Moreover, we assume that $n$ is sufficiently large so that $H^{>0}(\Sh(\calQ)_{\Delta,\FF}^\tor,\omega_\FF^{({\bf n}, n-2)}) =0$.
We remark that, because of our choice of $w = n-2$ (for parallel weight $n$ forms), the normalization factor as explained in Remark~\ref{R:q-exp} is just $1$. So we can temporarily ease ourselves from the scrutiny of normalizations.

Define the operator
$$V_M^\rmn:=\tilde{h}_M\circ T_p^\rmn - U_p^\rmn\circ\tilde{h}_M
$$
sending $H^0(\Sh(\calQ)_{\Delta,m}^\tor,\omega_m^1)$ into $H^0(\Sh(\calQ)_{\Delta,m}^\tor,\omega_m^{({\bf n},n)})$. (If $m=1$ we can choose $M=1$ and then $V_M^\rmn$ coincides with the classical operator $V_p$ (up to a twist) induced by Frobenius base-change on the abelian schemes parametrized by $\Sh\otimes\FF$.)
Notice that $V_M^\rmn$ is well defined, since the hypothesis on our weights guarantees that $T_p^\rmn$ and $U_p^\rmn$ are defined.

Now, we study  the effect of these operators on the $q$-expansions. First, the $q$-expansion of $\tilde{h}_M$ is one at each cusp of $\Sh_m^\tor$.
By Remark~\ref{R:q-exp} and the setup of \S\ref{S:tame Hecke operators}, at a cusp $\calC = (\gotha, \gothb, L, i,j , \lambda, \alpha)$ of $\calM_\gothc$, if we put $\calC': = (\gotha, \gothb, L', i, pj, p\lambda, \big( \begin{smallmatrix}
1&0\\0&p^{-1}
\end{smallmatrix} \big) \alpha)$ and $\calC'': = (\gotha, \gothb, L'', p^{-1}i, j, p\lambda, \big( \begin{smallmatrix}
p^{-1}&0\\0&1
\end{smallmatrix} \big) \alpha)$, where $L'$ (resp. $L''$) is the natural pullback (resp. pushout) of $L$ via the inclusion $\gothb \subseteq p^{-1}\gothb$ (resp. $\gotha^* \subseteq p^{-1}\gotha^*$), we can write down explicit the action of $U_p^\rmn$ and $T_p^\rmn$ on the level of $q$-expansions (modulo $\varpi^m$):
\[
a_\xi(U_p^\rmn(f), \calC, \Tate_{\gotha, \gothb}) = a_{p\xi}(f, \calC', \Tate_{\gotha, \gothb}) \quad\textrm{and}
\]
\[
a_\xi(T_p^\rmn(f), \calC, \Tate_{\gotha, \gothb}) = a_{p\xi}(f, \calC', \Tate_{\gotha, \gothb}) + a_{p^{-1}\xi}(f, \calC'', \Tate_{\gotha, \gothb}).
\]
So it follows that
\[
a_\xi(V_M^\rmn(f), \calC, \Tate_{\gotha, \gothb}) = a_{p^{-1}\xi}(f, \calC'', \Tate_{\gotha, \gothb})
\]

We claim that the following natural map
\begin{equation}
\label{E:doubling}
\psi_m = (\tilde h_M, V_M^\rmn):H^0(\Sh(\calQ)^\tor_{\Delta,m},\omega_m^{({\bf 1}, -1)})^{\oplus 2}\longrightarrow
 H^0(\Sh(\calQ)^\tor_{\Delta,m},\omega^{({\bf n},n-2)}_m)
\end{equation}
is injective.
It is enough to prove the injectivity when $m=1$; indeed, if $a\tilde h_M f = b V_M^\rmn g$ for $a, b \in \calO/\varpi^m$ and $(f, g)\in H^0(\Sh(\calQ)^\tor_{\Delta, m}, \omega_m^{({\bf 1}, -1)})^{\oplus 2}$, then writing $a = \varpi^r \bar a$ and $b = \varpi^r \bar b$ with $(\bar a, \bar b) \in (\calO/\varpi^{m-r})^{\oplus 2}  \backslash \{(0,0)\}$, we must have $\bar a \tilde h_M f = \bar b V_M^\rmn g$ in $H^0(\Sh(\calQ)^\tor_{\Delta,1}, \omega_1^{({\bf n}, n-2)})$. So injectivity of $\psi_1$ implies the injectivity of $\psi_m$.
We quickly remark that both $V_M^\rmn$ and multiplication by $h^M$ are clearly (individually) injective as can be seen from the map on the $q$-expansions.
Now, suppose that  we have
\begin{equation}\label{E:new}
h^M \bar{f}=a\cdot V_M^\rmn \bar{g} \quad \textrm{for}  \quad (\bar f, \bar g)\in H^0(\Sh(\calQ)^\tor_{\Delta, 1}, \omega_1^{({\bf 1}, -1)})  \textrm{ and } a \in \FF_p^\times.
\end{equation}
We now pullback all forms to $\calM(\calQ)_{\Delta, 1}^\tor$ instead.
Recall that there is a differential operator $\theta$ acting on Hilbert modular forms over $\FF$ and increasing weight by $(p+1,\dots,p+1)$, whose action on $q$-expansion mod $p$ (at any cusps) is given by: $\sum_\alpha a_\alpha q^\alpha \mapsto \sum_\alpha \overline{\Nm_{F/\QQ}(\alpha)}\cdot a_\alpha q^{\alpha}$ (cf. \cite[\S16.2]{AG}).
Applying $\theta$ to both sides of (\ref{E:new}), we obtain that $\theta(h^M \bar{f})=0$. Since $\theta$ and $h$ commute (we can check this on $q$-expansion), the injectivity of $h$ implies that $\bar{f}\in\ker\theta$. We conclude that $\bar{f}=0$, since $\theta$ has trivial kernel in weight $({\bf 1},-1)$. (This last fact follows from the arguments of \cite[IV]{katz2}, suitably extended to the settings of Hilbert modular forms).
By \eqref{E:new} again, we see that $V_M^\rmn(\bar  g) = 0$, by the injectivity of $V_M^\rmn$ we conclude $\bar g=0$.
Therefore, we have proved the claim above, namely the injectivity of $\psi_1$ and hence of \eqref{E:doubling}.

We point out that \eqref{E:doubling} is also equivariant under the action of $\TT^{\{0\}}_\calQ$ (as can be seen on the $q$-expansions). The action of $U_p^\rmn$ on the domain of $\psi_m$ is then given, via $\psi_m^{-1}$, by the matrix
$$\left( \begin{array}{cc} 
T_p^\rmn & S_p^\rmn\\
-1 & 0\\
\end{array} \right),
$$
which can be seen by the explicit $q$-expansions.  Here $S_p^\rmn$ is as defined in \S\ref{S:tame Hecke operators}. Therefore, $U_p^\rmn$ satisfies $X^2-T_p^\rmn X+S_p^\rmn =0$.
Denote by $\alpha$ and $\beta$ the distinct eigenvalues of $\bar\rho(\Frob_p)$ in $\FF^\times$, and choose lifts $\tilde\alpha$ and $\tilde\beta$ of $\alpha$ and $\beta$ respectively to $\calO^\times$ (for this we might need to enlarge $E$). We have $\alpha\beta\equiv S_p^\rmn \mod \gothm_\calQ$, and the Hecke operator $(U_p^\rmn-\tilde\alpha)(U_p^\rmn-\tilde\beta)$ acts nilpotently on $\mathrm{Im}(\psi_m)_{\gothm'_\calQ}$. 
In particular, this implies that $U_p^\rmn$ is invertible on $\mathrm{Im}(\psi_m)_{\gothm'_\calQ}$ and $T_p^\rmn = S_p^\rmn (U_p^\rmn)^{-1} + U_p^\rmn$.

We denote by $\TT_{\calQ,n}^{\{0\}}$ the Hecke algebra acting on $\oplus_{m\geq 1}H^0(\Sh(\calQ)_{\Delta,m}^\tor,\omega^{({\bf n},n-2)}_m)$ generated by the operators $T_\gothq$ for $\gothq\notin \calS \cup \calQ$, $U_\gothq$ for $\gothq \in \calQ$, $S_\gothq^\rmn$, and the diamond operators. We set $\tilde\TT_{\calQ,n}^{\{0\}}:=\TT_{\calQ,n}^{\{0\}}[U_p]$. 
Denote by $\gothm'_{\calQ, n}$ the maximal ideal of $\TT_{\calQ,n}^{\{0\}}$ generated by \eqref{E:generators of mrho} and $U_\gothq-\alpha_\gothq$ for all $\gothq\in \calQ$. Let moreover $\tilde\gothm_\alpha$ (resp. $\tilde\gothm_\beta$) denote the maximal ideal of $\tilde\TT_{\calQ,n}^{\{0\}}$ containing $\gothm'_{\calQ,n}$ and $U_p^\rmn-\alpha$ (resp. $U_p^\rmn-\beta$).

Let $I_m$ denote the annihilator in $\TT_{\calQ,\gothm'_\calQ}^{\{0\}}$ of $H^0(\Sh(\calQ)^\tor_{\Delta,m},\omega^{({\bf 1}, -1)}_m)_{\gothm'_\calQ}$. 
As in \cite[\S3.5]{CG}, we see that $\TT_{\calQ,\gothm_\calQ}^{\{0\}}/I_m[U_p^\rmn]\subset\End_\calO(\mathrm{Im}\psi)_{\tilde\gothm_\alpha}$ contains $T_p^\rmn$ and is naturally a quotient of $\tilde\TT_{\calQ,n,\tilde\gothm_\alpha}^{\{0\}}$.
Denote by $\tilde \rho_{\calQ,n,\alpha}:G_F\to\GL_2(\tilde \TT_{\calQ,n,\tilde \gothm_\alpha}^{\{0\}})$ the Galois representation attached to the ordinary Hecke algebra acting in weight $({\bf n},n-2)$ and cohomological degree zero. 
Composing this representation with the quotient map considered above, we obtain representations $\tilde \rho_{\calQ,m}: G_F \to \GL_2(\TT_{\calQ,\gothm'_\calQ}^{\{0\}}/I_m[U_p^\rmn])$. 
Since all traces and determinants lie in the smaller ring $\TT_{\calQ, \gothm'_\calQ}^{\{0\}}/I_m$, we obtain a representation $ \rho_{\calQ,m}: G_F \to \GL_2(\TT_{\calQ,\gothm'_\calQ}^{\{0\}}/I_m)$.
The representation $\rho_\calQ:=\varprojlim_m \rho_{\calQ,m}$ satisfies the desired properties, except possibly the condition of being unramified at $p$. We observe that

$$\tilde\rho_{\calQ,m}|_{G_{F_p}}\simeq 
\left( \begin{array}{ccc}
\epsilon_m^{M(p-1)}\lambda_{\tilde\beta} & * \\
0 & \lambda_{\tilde\alpha}
\end{array} \right) 
$$
where $\lambda_x:G_{F_p}\to(\calO/(\varpi^m))^\times$ denotes the unramified character of $G_{F_p}$ sending a geometric Frobenius element to $x$. Notice that $\epsilon_m^{M(p-1)}$ is trivial since $p^{m-1}$ divides $M$.

The Galois representation $\tilde\rho_{\calQ,m}$ can be equivalently (by the Chebotarev density theorem) defined using the eigenvalue $\beta$, so that:
$$\tilde\rho_{\calQ,m}|_{G_{F_p}}\simeq 
\left( \begin{array}{ccc}
\lambda_{\tilde\alpha} & * \\
0 & \lambda_{\tilde\beta}
\end{array} \right) \simeq
\left( \begin{array}{ccc}
\lambda_{\tilde\beta} & * \\
0 & \lambda_{\tilde\alpha}
\end{array} \right),
$$
Since $\tilde\alpha\neq\tilde\beta$ we deduce that the extension classes denoted by $*$ are trivial. Thus $\tilde \rho_{\calQ, m}$ and hence $\rho_{\calQ,m}$ is unramified at $p$. 
\end{proof}

\begin{remark}
It seems that the methods of \cite[\S3.6--7]{CG} would allow us to prove the above result also when $\alpha=\beta$. See e.g. \cite{dimitrov-wiese}.
\end{remark}

\subsection{Unramifiedness in the case of surfaces}\label{S:unr sur}
We assume in this section that $g=2$. Recall that we are moreover requiring for simplicity that $p$ is inert in $F$. We will prove, under the assumption of Frobenius-distinguishedness introduced in Proposition \ref{P:nr}, that Galois representations arising from Hilbert modular classes of paritious weights $\kappa=({\bf 1},-1)$ are unramified at $p$.
We keep the notation as in \S \ref{S:Hecke TW primes} and Proposition~\ref{P:nr}.

Let $\chi$ denote the Teichm\"uller lift of $\det\bar\rho$ and denote by $R_\calQ$ the complete local Noetherian $\calO$-algebra representing the functor of framed $\calO$-deformations of $\bar\rho|_{G_{F, \calS \cup \calQ}}$ with determinant $\chi$.

Denote by $\TT_\calQ$ the image of the abstract tame Hecke algebra acting on
\[
\bigoplus_{m \geq 1, k \geq 0} H^k (\Sh(\calQ)^\tor_{\Delta, m}, \omega_m^{({\bf 1}, -1)}),
\]
by sending $t_\gothq \mapsto T_\gothq$ ($\gothq \notin \calS \cup \calQ$), $s_\gothq^\rmn \mapsto S_\gothq^\rmn$, and $[a]\mapsto $ the Diamond operator $\langle \tilde a \rangle$ for $\tilde a \in (\calO_F/\calQ\calO_F)^\times$ lifting $a \in \Delta$.
There is a natural surjective map $\TT_\calQ \twoheadrightarrow \TT_\calQ^{\{0\}}$ (where the latter is the Hecke action on $H^0$ only). Let $\gothm_\calQ$ denote the preimage of the maximal ideal $\gothm'_\calQ$.
The main result of  \cite{emerton-reduzzi-xiao} implies that there is a natural continuous homomorphism of $\calO$-algebras $R_\calQ\rightarrow\TT_{\calQ,\gothm_\calQ}$. In particular, we can view $H^k(\Sh(\calQ)_{\Delta,m}^\tor, \omega_m^{({\bf 1}, -1)})_{\gothm_\calQ}$ as a module over $R_\calQ$.

We shall frequently use the following.
The exact sequence
\[
0\rightarrow\omega^{({\bf 1}, -1)}\overset{\cdot\varpi^m}\longrightarrow\omega^{({\bf 1}, -1)}\longrightarrow\omega_m^{({\bf 1}, -1)}\rightarrow 0
\]
of coherent sheaves on $\Sh(\calQ)_{\Delta}^\tor$ induces a long exact sequence in cohomology which, after localization at $\gothm_\calQ$, is given by
\begin{align}\label{E:les}
\dots\rightarrow & H^i(\Sh(\calQ)_{\Delta}^\tor,\omega^{({\bf 1}, -1)})_{\gothm_\calQ}\xrightarrow{\varpi^m} H^i(\Sh(\calQ)_{\Delta}^\tor,\omega^{({\bf 1}, -1)})_{\gothm_\calQ}\rightarrow H^i(\Sh(\calQ)_{\Delta,m}^\tor,\omega^{({\bf 1}, -1)}_m)_{\gothm_\calQ}\\
\nonumber \rightarrow&
H^{i+1}(\Sh(\calQ)_{\Delta}^\tor,\omega^{({\bf 1}, -1)})_{\gothm_\calQ}\rightarrow\cdots
\end{align}


Denote by $\mathscr{I}_\calQ$ the ideal of $R_\calQ$ characterized by the following property: a lifting $\rho:G_{F,\calS \cup \calQ}\ra \GL_2(A)$ of $\bar\rho$ with values in a complete local noetherian $\calO$-algebra $A$ is unramified at the unique prime of $F$ above $p$ if and only if the corresponding map $R_\calQ\ra A$ factors through $R_\calQ/\scrI_\calQ$.

\begin{lemma}\label{L:unramified}
If $i\in\{0,2\}$, the $R_\calQ$-module $H^i(\Sh(\calQ)_{\Delta,m}^\tor,\omega_m^{({\bf 1}, -1)})_{\gothm_\calQ}$ is supported on $\Spec(R_\calQ/\mathscr{I}_\calQ)$.
\end{lemma}
\begin{proof}
When $i=0$ this follows from the Frobenius distinguishedness assumption and Proposition \ref{P:nr}. Assume $i=2$. By Lemma~\ref{L:hecke operator duality}, the Serre duality gives a natural isomorphism
\[
H^2(\Sh(\calQ)_{\Delta,m}^\tor, \omega_m^\kappa) \cong H^0(\Sh(\calQ)_{\Delta, m}^\tor, \omega_m^\kappa(-\ttD))^
\vee
\]
intertwining the action of $T_\gothq$ with $T_{\gothq}^\vee$ for $\gothq \notin \calS \cup \calQ$.
But the Galois representation appearing in $H^0(\Sh(\calQ)_{\Delta, m}^\tor, \omega_m^\kappa(-\ttD))_{\gothm_\calQ} \subseteq (\Sh(\calQ)_{\Delta, m}^\tor, \omega_m^\kappa)_{\gothm_\calQ}$ is unramified by Proposition~\ref{P:nr}. This implies our lemma when $i=2$.
\end{proof}

Note that we did not require the residual representation $\bar{\rho}$ to be unramified at $p$. (But of course if $\bar \rho$ is ramified at $p$, the localizations $H^i(\Sh(\calQ)_{\Delta,m}^\tor, \omega_m^{({\bf 1}, -1)})_{\gothm_\calQ}$ for $i=0,2$ are zero.) The following lemma shows this for $i=1$.

\begin{lemma}\label{L:rho unram}
Suppose that $\bar \rho$ is modular, i.e. $\gothm_\calQ$ is not the entire $\TT_{\calQ}$.
The representation $\bar\rho$ is unramified at $p$.
\end{lemma}
\begin{proof}
Suppose by contradiction that $\bar\rho$ is ramified at $p$. 
By Lemma \ref{L:unramified} we see that $H^i(\Sh(\calQ)_{\Delta,m}^\tor,\omega_m^{({\bf 1}, -1)})_{\gothm_\calQ}=0$ for $i\in\{0,2\}$ and for any $m$. We need to prove this for $i=1$.

The exact sequence (\ref{E:les}) gives: 
\begin{equation*}
\begin{split}
0 & \to  H^1(\Sh(\calQ)_{\Delta}^\tor,\omega^{({\bf 1}, -1)})_{\gothm_\calQ}
\xrightarrow{\varpi^m}
H^1(\Sh(\calQ)_{\Delta}^\tor,\omega^{({\bf 1}, -1)})_{\gothm_\calQ}
\rightarrow 
H^1(\Sh(\calQ)_{\Delta,m}^\tor,\omega^{({\bf 1}, -1)}_m)_{\gothm_\calQ} \\ 
 & \rightarrow
H^{2}(\Sh(\calQ)_{\Delta}^\tor,\omega^{({\bf 1}, -1)})_{\gothm_\calQ}\xrightarrow{\varpi^m}
H^{2}(\Sh(\calQ)_{\Delta}^\tor,\omega^{({\bf 1}, -1)})_{\gothm_\calQ}\rightarrow 0.
\end{split}
\end{equation*}
The injectivity of the second map implies that $H^1(\Sh(\calQ)_{\Delta}^\tor,\omega^{({\bf 1}, -1)})_{\gothm_\calQ}$ is $\varpi$-torsion free. But by Fact~\ref{F:qK cohomology} below, $H^1(\Sh(\calQ)_{\Delta,E}^\tor,\omega_E^{({\bf 1}, -1)})_{\gothm_\calQ}$ sees only representations that are unramified at $p$.
From this, we deduce that $H^1(\Sh(\calQ)_{\Delta}^\tor,\omega^{({\bf 1}, -1)})_{\gothm_\calQ} = 0$.
Moreover, the surjectivity of the multiplication-by-$\varpi^m$ map between the degree-two cohomology groups implies that those localized modules are zero. We conclude that $H^1(\Sh(\calQ)_{\Delta,m}^\tor,\omega_m^{({\bf 1}, -1)})_{\gothm_\calQ}=0$, and hence the ideal $\gothm_\calQ$ is not in the support of any cohomology, contradicting the modularity of $\bar\rho$.
\end{proof}

\begin{fact}
\label{F:qK cohomology}
All Galois representations appeared in $H^i(\Sh(\calQ)^\tor_{\Delta, \CC}, \omega_\CC^{({\bf 1}, -1)})$ are unramified at $p$.
\end{fact}
\begin{proof}
It follows from \cite[Theorem~2.4.4]{harris} that
\[
H^i(\Sh(\calQ)^\tor_{\Delta, \CC}, \omega_\CC^{({\bf 1}, -1)}) \cong H^i(\gothq, K_\infty; C_{si} \otimes \CC(1)).
\]
After localizing at a non-Eisenstein ideal, an easy computation of $(\gothq, K_\infty)$-cohomology gives an isomorphisms of Hecke modules
\[
H^i(\Sh(\calQ)_{\Delta,\CC}^\tor,\omega^{({\bf 1}, -1)}_{\CC})_{\gothm_\calQ} \simeq H^0(\Sh(\calQ)_{\Delta,E}^\tor,\omega^{({\bf 1}, -1)}_{\CC})_{\gothm_\calQ} \otimes \wedge^i(\CC^2).
\]
Our Fact follows from the unramifiedness of Galois representation arising from $H^0$ by \cite{deligne-serre}.
\end{proof}

From now on, we assume that $\bar \rho$ is modular.

\begin{proposition}\label{P:n=4}
There is a positive integer $n$ such that the $R_\calQ$-modules $H^i(\Sh(\calQ)_{\Delta}^\tor,\omega^{({\bf 1}, -1)})_{\gothm_\calQ}$ and $H^i(\Sh(\calQ)_{\Delta,m}^\tor,\omega_m^{({\bf 1}, -1)})_{\gothm_\calQ}$ are annihilated by $\mathscr{I}_\calQ^n$ for all $i,m$. Moreover, the value $n=3$ suffices (when $g=2$).
\end{proposition}
\begin{proof}
We argue using the long exact sequence \eqref{E:les} of $R_\calQ$-modules, which we spell out here.
\begin{align}
\label{E:long exact sequence}
0\rightarrow & H^0(\Sh(\calQ)_{\Delta}^\tor,\omega^{({\bf 1}, -1)})_{\gothm_\calQ}\xrightarrow{\varpi^m} H^0(\Sh(\calQ)_{\Delta}^\tor,\omega^{({\bf 1}, -1)})_{\gothm_\calQ}\rightarrow H^0(\Sh(\calQ)_{\Delta,m}^\tor,\omega^{({\bf 1}, -1)}_m)_{\gothm_\calQ}
\\
\nonumber\rightarrow & H^1(\Sh(\calQ)_{\Delta}^\tor,\omega^{({\bf 1}, -1)})_{\gothm_\calQ}\xrightarrow{\varpi^m} H^1(\Sh(\calQ)_{\Delta}^\tor,\omega^{({\bf 1}, -1)})_{\gothm_\calQ}\rightarrow H^1(\Sh(\calQ)_{\Delta,m}^\tor,\omega^{({\bf 1}, -1)}_m)_{\gothm_\calQ}
\\
\nonumber\rightarrow & H^2(\Sh(\calQ)_{\Delta}^\tor,\omega^{({\bf 1}, -1)})_{\gothm_\calQ}\xrightarrow{\varpi^m} H^2(\Sh(\calQ)_{\Delta}^\tor,\omega^{({\bf 1}, -1)})_{\gothm_\calQ}\rightarrow H^2(\Sh(\calQ)_{\Delta,m}^\tor,\omega^{({\bf 1}, -1)}_m)_{\gothm_\calQ} \to 0.
\end{align}

First, $H^0(\Sh(\calQ)_{\Delta}^\tor,\omega^{({\bf 1}, -1)})_{\gothm_\calQ}$ is $\varpi$-torsion free, hence it is a Hecke-equivariant subspace of $H^0(\Sh(\calQ)^\tor_{\Delta, E},\omega^{({\bf 1}, -1)}_E)_{\gothm_\calQ}$, which is annihilated by $\scrI_\calQ$ thanks to \cite{deligne-serre}. By Lemma \ref{L:unramified} we also know that $H^0(\Sh(\calQ)_{\Delta,m}^\tor,\omega^{({\bf 1}, -1)}_m)_{\gothm_\calQ}$ and $H^2(\Sh(\calQ)_{\Delta,m}^\tor,\omega^{({\bf 1}, -1)}_m)_{\gothm_\calQ}$ are annihilated by $\scrI_\calQ$.
Moreover, by Serre duality and its compatibility with Hecke operators (cf. Lemma~\ref{L:hecke operator duality}), we have an isomorphism of $R_\calQ$-modules
\[
H^2(\Sh(\calQ)_{\Delta}^\tor,\omega^{({\bf 1}, -1)})_{\gothm_\calQ} \cong H^0(\Sh(\calQ)_{\Delta}^\tor,\omega^{({\bf 1}, -1)} \otimes E/\calO)_{\gothm_\calQ}^\vee.
\]
We know that $\scrI_\calQ$ annihilate the right hand side, so it also kills the left hand side.

We now look at the following exact sequence of $R_\calQ$-modules (to separate the torsion part and torsion-free part)
\begin{equation}\label{E:H1tor}
0 \ra H^1(\Sh(\calQ)_\Delta^\tor,\omega^{({\bf 1}, -1)})_{\gothm_\calQ}[\varpi^\infty] \ra H^1(\Sh(\calQ)_\Delta^\tor,\omega^{({\bf 1}, -1)})_{\gothm_\calQ} \ra 
E\otimes_\calO H^1(\Sh(\calQ)_\Delta^\tor,\omega^{({\bf 1}, -1)})_{\gothm_\calQ}.
\end{equation}
The torsion part in the above sequence is a quotient of $H^0(\Sh(\calQ)_{\Delta,m'}^\tor,\omega^{({\bf 1}, -1)}_{m'})_{\gothm_\calQ}$ for some large enough $m'>0$ by \eqref{E:long exact sequence}; in particular as an $R_\calQ$-module it is annihilated by $\scrI_\calQ$. The last term of (\ref{E:H1tor}) coincides with $H^1(\Sh(\calQ)_{\Delta, E}^\tor,\omega^{({\bf 1}, -1)}_E)_{\gothm_\calQ}$ and is therefore annihilated by $\scrI_\calQ$ according to Fact~\ref{F:qK cohomology}. We conclude that $\scrI_\calQ^2\cdot H^1(\Sh^\tor,\omega^{({\bf 1}, -1)})_{\gothm_\calQ}=0$.

Finally, by the exact sequence \eqref{E:long exact sequence}, the term $H^1(\Sh^\tor_m,\omega_m^{({\bf 1}, -1)})_{\gothm_\calQ}$ is an extension of a submodule of $H^2(\Sh^\tor,\omega^{({\bf 1}, -1)})_{\gothm_\calQ}$ and a quotient of $H^1(\Sh^\tor_m,\omega_m^{({\bf 1}, -1)})_{\gothm_\calQ}$. So it is annihilated by $\scrI_\calQ^3$. This concludes the proof of the Proposition.
\end{proof}

\subsection{The patching argument}\label{S:patch}
We make use here the techniques of patching complexes of \cite[\S6]{CG}. All the ring dimensions computed below are absolute Krull dimensions, unless otherwise stated. Recall that $p$ is assumed to be inert in the totally real field $F$, and that we assume $g=2$ (but most of the arguments below continue to hold for arbitrary $g$).

Let $R_p^\square$ denote the universal \emph{framed} $\calO$-deformation ring of $\bar\rho|_{G_{F_p}}$ corresponding to lifts of determinant $\chi|_{G_{F_p}}$. Let $\scrI_p$ denote the ideal defining the locus where the universal (framed) deformation of $\bar  \rho$ is unramified. For each positive integer $n$ we set $R_p^{(n)}:= R_p^\square/\mathscr{I}_p^n$. Since unramified lifts are determined by the matrix of a Frobenius element, we deduce by simple calculations that:
$$ \dim R_p^{(n)} = 4.$$

Recall the finite set $\calS$ of places of $F$ defined in \S \ref{S:tame Hecke operators}.
For each finite place $\gothq\in\calS \backslash\{p\}$ we similarly denote by $R_\gothq^\square$ the universal framed $\calO$-deformation ring of $\bar\rho|_{G_{F_\gothq}}$ corresponding to lifts of determinant $\chi|_{G_{F_\gothq}}$. Since $\gothq\nmid p$ we have (cf. \cite[Theorem 3.3.1]{boe}):
$$\dim R^\square_\gothq=4, \hspace{0.3cm} \gothq\in \calS \backslash\{p\}.$$

Put
$$R^{(n)}_{\mathrm{loc}}:=R_p^{(n)}\hat\otimes_\calO\,\hat\bigotimes_{\gothq\in\calS \backslash\{p\}}R^\square_\gothq.$$

We put $R_\calQ^{(n)}: = R_\calQ / \scrI_\calQ^n$, and let  $R_\calQ^{(n),\square_{\calS}}$ denote the deformation ring with frames at finite places in $\calS$. In particular, $R_\calQ^{(n),\square_\calS}$ is a free power series ring in $j : =4|\calS|-1$ variables over $R_\calQ^{(n)}$.
Moreover, restriction to decomposition groups at the finite places in $\calS$ induces a natural morphism $R_{\mathrm{loc}}^{(n)}\to R_\calQ^{(n),\square_\calS}$.

Assume that $\bar\rho(G_{F})$ contains $\SL_2(\FF_p)$ and that $p>3$. Recall that $\bar\rho$ is totally odd, since it is modular. By \cite[Proposition 5.9]{gee} there exists an integer $q\geq 1$ with the following property: for any $N\geq 1$ there is a set $\calQ_N$ consisting of finite primes of $F$ such that:
\begin{itemize}
\item $\calQ_N$ has cardinality $q$ and is disjoint from the set $\calS$;
\item for each $\gothq\in \calQ_N$, $\bar\rho(\mathrm{Frob}_\gothq)$ has two distinct eigenvalues $\alpha_\gothq,\beta_\gothq\in\FF$;
\item $\mathrm{Nm}(\gothq)\equiv 1$ (mod $p^N$) for all $\gothq\in \calQ_N$;
\item $R_{\calQ_N}^{(n),\square_\calS}$ is topologically generated over $R_{\mathrm{loc}}^{(n)}$ by $h:=q+|\calS|-1-[F:\QQ]=q+|S|-3$ elements.
\end{itemize}

For each $N$ we now fix a choice of Taylor--Wiles primes $\calQ_N$, and for each such set a choice of distinguished eigenvalue $\alpha_\gothq$ of $\bar\rho(\Frob_\gothq)$ for $\gothq\mid \calQ_N$ .

Set $n=3$ and define
$$R_\infty:=R^{(n)}_{\mathrm{loc}}[[x_1,\dots,x_h]],$$
so that
\begin{equation}\label{E:num}
\dim R_\infty = 3|\calS|+1+h=1+q+j-[F:\QQ].
\end{equation}

We choose for each $N \geq 1$ a surjection 
\begin{equation}\label{E:surj}
R_\infty\twoheadrightarrow R_{\calQ_N}^{(3),\square_S}.
\end{equation}

We let $S_N=\calO[\Delta_N]$, where $\Delta_N=(\ZZ/p^N\ZZ)^q$, and we set $S_\infty=\varprojlim_N S_N\simeq\calO[[(\ZZ_p)^q]]$. If $M\geq N\geq 0$ and if $I$ is an ideal of $\calO$, we regard $S_N/I$ as a quotient of $S_M$ via the natural surjective maps $\calO\to\calO/I$ and $\Delta_M\to\Delta_N$.

We denote the operation of complete tensor product over $\calO$ with $\calO^\square:=\calO[[z_1,\dots,z_j]]$ by the superscript $\square$. For example, $S^\square_\infty: = S_\infty \widehat \otimes_\calO \calO[[z_1, \dots, z_j]]$.

Denote by $\gothm_{Q_N}$ the maximal ideal of the Hecke algebra $\TT_{\calQ_N}$ contracting to $\gothm_\emptyset\subset\TT_\emptyset$ and containing $U_\gothq-\alpha_\gothq$ for each $\gothq \mid \calQ_N$, where the eigenvalues $\alpha_\gothq$ are fixed as above. Applying to our settings the construction of \cite[\S7.2]{CG} (noting that $\Sh(\calQ)_{\Delta} \to \Sh(\calQ)$ is \'etale with Galois group $(\calO_F/\calQ\calO_F)^\times$ by \S \ref{S:Sh(Q)01}), we deduce the existence of a perfect complex $$0\to C_{N,2}\to C_{N,1}\to C_{N,0}\to 0$$ of $S_N/(\varpi^N)$-modules with an action of $\TT_{\calQ_N,\gothm_{\calQ_N}}$ whose $i$th homology is $\TT_{\calQ_N,\gothm_{\calQ_N}}$-equivariantly isomorphic to:
\[
H_i(\Sh(\calQ_N)^\tor_{\Delta_N,N},\omega_N^{({\bf 1},-1)}(-\ttD))_{\gothm_{\calQ_N}}  := H^i(\Sh(\calQ_N)^\tor_{\Delta_N,N},\omega_N^{({\bf 1},-1)})^\vee_{\gothm_{\calQ_N}}.
\]

Here  the superscript $^\vee$ denotes taking $\calO/(\varpi^N)$-dual.

\begin{remark}
The complex that we have denoted here by $C_{N,\ast}$ is denoted in the end of \S 7.2 of \cite{CG} by $\varinjlim_m T^m C_n$, where $T$ is a suitable Hecke operator constructed therein. Taking the limit is what allows to obtain the cohomology localized at ${\gothm_{\calQ_N}}$ as in \cite[\S7.2]{CG}.
\end{remark}

Let $D_{N}^\ast$ denote the chain complex obtained from $C_{N,\ast}$ by taking $\calO/(\varpi^N)$-duals, \emph{i.e.}, $D_{N}^i:=C_{N,i}^\vee$. Observe that
\begin{equation}\label{E:D_N}
H^i(D_{N}^\ast)\simeq H^i(\Sh(\calQ_N)^\tor_{\Delta_N,N},\omega^{({\bf 1}, -1)}_N)_{\gothm_{\calQ_N}}
\end{equation}
is the cohomology in which we are interested. 
The main result of \cite{emerton-reduzzi-xiao} together with Proposition \ref{P:n=4} imply the existence of a canonical map $R_{\calQ_N}^{(4),\square_\calS}\to\TT_{\calQ_N,\gothm_{\calQ_N}}^\square$, so that the global deformation ring $R_{\calQ_N}^{(4),\square_\calS}$ acts on the cohomology of $D_N^{\square,\ast}: = D_N^\ast \widehat\otimes_{\calO} \calO^\square$. In particular for each $M\geq N\geq 0$ with $M\geq 1$ and for each $m\geq 1$, the cohomology $H^i(D_M^{\square,\ast}\otimes_{S_M}S_N/(\varpi^m))$ is also an $R_\infty$-module via (\ref{E:surj}), and the actions of $R_\infty$ and $S_M^\square$ on this space commute.

We let $H:=H^2(\Sh^\tor,\omega^{({\bf 1}, -1)}\otimes_\calO E/\calO)_{\gothm_\emptyset}$ and we defined a chain complex $T$ with trivial differentials $d=0$ by setting $T^i:=H^i(\Sh^\tor_\FF,\omega^{({\bf 1}, -1)}_\FF)_{\gothm_\emptyset}$, so that $H^2(T)\simeq H/(\varpi)$. We also set $R=R_\emptyset^{(4),\square_S}$ to be the global deformation ring defined earlier attached to the empty set of Taylor--Wiles primes. Notice that $H$ is an $R$-module.

By (\ref{E:num}) we have the numerical equality:

$$\dim R_\infty = \dim S_\infty^\square - [F:\QQ].$$

We then see that the hypotheses of \cite[Theorem 6.3]{CG} are satisfied (the notation we introduced for $S_N,S_\infty,R,R_\infty,H,T,$ and $D_N$ matches the notation of the statement of the theorem in \emph{loc.cit.}, with $l_0$ there being equal to $g=2$).
In particular we can patch the complexes $D_{N}^\ast$ to produce a perfect chain complex

$$0\to P_\infty^{\square,0}\to P_\infty^{\square,1} \to P_\infty^{\square,2} \to 0$$
of finitely generated $S_\infty^\square$-modules which is a projective resolution of $H^2(P_\infty^{\square,\ast})$. Moreover, the cohomology of $P_\infty^{\square,\ast}$ carries an action of $R_\infty\hat\otimes_\calO S^\square_\infty$. We have therefore an isomorphism of $R_\infty\otimes_\calO \calO/(\varpi^m)$-modules:

$$\Tor_1^{S_\infty^\square}(H^2(P_\infty^{\square,\ast}),\calO/(\varpi^m))\simeq H^1(P_\infty^{\square,\ast}\otimes_{S_\infty^\square}\calO/(\varpi^m)).$$

Observe that the action of $R_\infty$ on $H^2(P_\infty^{\square,\ast})$ factors through the quotient by $\mathscr{I}_pR_\infty$ since the top-degree cohomology $H^2(P_\infty^{\ast})$ is constructed by patching the duals of a suitable system of $H^0(D_{M_i}^\ast\otimes_{S_{M_i}}S_{N_i}/(\varpi^{N_i}))$ for various $M_i\geq N_i\geq 1$ (cf. proof of \cite[Theorem 6.3]{CG}), and these $H^0$ are all supported on the unramified locus by (\ref{E:D_N}) and Proposition \ref{P:nr}. So $H^1(P_\infty^{\square,\ast}\otimes_{S_\infty^\square}\calO/(\varpi^m))\simeq H^1(\Sh^\tor_m,\omega^{({\bf 1}, -1)}_m)_{\gothm_\emptyset}^\square$ is annihilated by $\mathscr{I}_p\subset R_p^\square$. Similarly, computing $\Tor_1^{S_\infty^\square}(H^2(P_\infty^{\square,\ast}),\calO)$ we see that $H^1(\Sh^\tor,\omega^{({\bf 1}, -1)})_{\gothm_\emptyset}^\square$ is annihilated by $\mathscr{I}_p$.\\

We conclude:

\begin{theorem}\label{T:main}
Assume that $p>3$, that $\bar\rho$ is Frobenius-distinguished at $p$, and that $\bar\rho(G_{F})$ contains $\SL_2(\FF_p)$. Then
the cohomology groups $H^1(\Sh^\tor_m,\omega^{({\bf 1}, -1)}_m)_{\gothm_\emptyset}$ and $H^1(\Sh^\tor,\omega^1)_{\gothm_\emptyset}$ are supported on $\Spec(R_\emptyset/\scrI_\emptyset)$, \emph{i.e.}, they give rise to Galois representations unramified at $p$.
\end{theorem}

\begin{remark}
It seems that if we assume minimality condition analogous to \cite[Theorem 1.3]{CG},  we might be able to prove certain $R=T$ theorem for minimal deformations.
\end{remark}

\subsection{A conjecture}

The conjecture below follows from the expected properties of modular Galois representations arising from forms of weight $(1,\dots,1)$:

\begin{conjecture}\label{C:alpha}
Let $F$ be a totally real number field of degree $g$ over $\QQ$ and let $p$ be an arbitrary prime number. Fix a prime $\gothp$ of $F$ lying above $p$, and let $\kappa_\gothp=((k_\tau)_{\tau\in\Sigma},w)$ be a paritious weight such that $k_\tau=1$ for all $\tau\in\Sigma_\gothp$, and that $w=-1$.
Let $\TT$ denote the image of the universal tame Hecke algebra acting on $H^\bullet(\Sh^\tor,E/\calO\otimes_\calO\omega^{\kappa_\gothp})$ and let $\gothm$ denote a non-Eisenstein maximal ideal of $\TT$, with associated Galois representation $\bar\rho$. Then:
\begin{enumerate}
\item $\bar\rho$ is unramified at $\gothp$;
\item Let $R$ denote the universal ring for $\calO$-deformations of $\bar\rho$ with fixed central character, and let $\scrI$ denote the proper ideal of $R$ cutting out the locus of lifts that are unramified at $ \gothp$. Then there exists a positive integer $n$ depending on $g$ such that $\scrI^n$ annihilates the $R$-module $H^\bullet(\Sh^\tor,E/\calO\otimes_\calO\omega^{\kappa_\gothp})_\gothm$.
\end{enumerate}
\end{conjecture}

Assuming the above conjecture and applying the arguments of the previous section one can prove that:

$$\Tor_i^{S_\infty^\square}(H^g(P_\infty^{\square,\ast}),\calO/(\varpi^m))\simeq H^i(P_\infty^{\square,\ast}\otimes_{S_\infty^\square}\calO/(\varpi^m))\simeq H^i(\Sh^\tor_m,\omega^{\kappa_\gothp}_m)_\gothm^\square$$
for all $i$ and all $m$. (Here the notation is as before). Using a suitable generalization of Proposition \ref{P:nr} to the case of non-Frobenius-distinguished representations, together with Grothendieck-Serre-Verdier duality, we see that the action of $R$ on $H^g(P_\infty^{\square,\ast})$ factors through $\scrI$. We then obtain:

\begin{proposition}\label{T:conj_main}
Fix a prime $\gothp$ of $F$ above $p>3$ and a paritious weight $\kappa_\gothp$ with $k_\tau=1$ for all $\tau\in\Sigma_\gothp$ and $w=-1$. Let $\TT$ denote the image of the universal tame Hecke algebra acting on $H^\bullet(\Sh^\tor,E/\calO\otimes_\calO\omega^{\kappa_\gothp})$ and let $\gothm$ denote a non-Eisenstein maximal ideal of $\TT$. Assume the validity of Conjecture \ref{C:alpha}, and suppose Proposition \ref{P:nr} holds without the assumption of Frobenius distinguishness at $\gothp$.  Then $H^\bullet(\Sh^\tor,E/\calO\otimes_\calO\omega^{\kappa_\gothp})_\gothm$ is supported on the unramified locus of $\Spec R$.
\end{proposition}

\end{document}